\documentclass[11pt,reqno,tbtags,draft]{amsart}
\usepackage{amssymb}
\usepackage{etoolbox} 
\usepackage{url}
\usepackage[square,numbers]{natbib}
\bibpunct[, ]{[}{]}{;}{n}{,}{,}

\usepackage{mathabx}

\title[sum of powers of subtree sizes]
{The sum of powers of subtree sizes for conditioned
 Galton--Watson trees}

\date{April 4, 2021}

\newcommand\urladdrx[1]{{\urladdr{\def~{{\tiny$\sim$}}#1}}}
\author{James Allen Fill}
\address{Department of Applied Mathematics and Statistics,
The Johns Hopkins University,
3400 N.~Charles Street,
Baltimore, MD 21218-2682 USA}
\email{jimfill@jhu.edu}
\urladdrx{http://www.ams.jhu.edu/~fill/}
\thanks{Research of the first author supported 
by the Acheson~J.~Duncan Fund for the Advancement of Research in
Statistics.}

\author{Svante Janson}
\thanks{Research of the second author supported by the Knut and Alice
  Wallenberg Foundation} 
\address{Department of Mathematics, Uppsala University, PO Box 480,
SE-751~06 Uppsala, Sweden}
\email{svante.janson@math.uu.se}
\urladdrx{http://www2.math.uu.se/~svante/}

\makeatletter
\@namedef{subjclassname@2020}{\textup{2020} Mathematics Subject Classification}
\makeatother


\numberwithin{equation}{section}

\renewcommand\le{\leqslant}
\renewcommand\ge{\geqslant}

\allowdisplaybreaks

\theoremstyle{plain}
\newtheorem{theorem}{Theorem}[section]
\newtheorem{lemma}[theorem]{Lemma}

\newtheorem{corollary}[theorem]{Corollary}

\theoremstyle{definition}
\newtheorem{example}[theorem]{Example}

\newtheorem{problem}[theorem]{Problem}
\newtheorem{remark}[theorem]{Remark}

\newtheorem*{acks}{Acknowledgements}

\AtEndEnvironment{remark}{\null\hfill\qedsymbol}
\AtEndEnvironment{example}{\null\hfill\qedsymbol}

\newenvironment{romenumerate}[1][-10pt]{
\addtolength{\leftmargini}{#1}\begin{enumerate}
 }{\end{enumerate}}

\newcounter{oldenumi}
{\setcounter{oldenumi}{\value{enumi}}
\begin{romenumerate} \setcounter{enumi}{\value{oldenumi}}}
{\end{romenumerate}}

\newcounter{thmenumerate}
\newenvironment{thmenumerate}
{\setcounter{thmenumerate}{0}%
 \def\item{\par
 \refstepcounter{thmenumerate}\textup{(\roman{thmenumerate})\enspace}}
}
{}

\newcounter{xenumerate}   

\newcommand\pfitemx[1]{\par#1:}
\newcommand\pfitemref[1]{\pfitemx{\ref{#1}}}

\newcommand\pfcase[2]{\smallskip\noindent\emph{Case #1: #2} \noindent}

\newcommand{\refT}[1]{Theorem~\ref{#1}}
\newcommand{\refTs}[1]{Theorems~\ref{#1}}

\newcommand{\refL}[1]{Lemma~\ref{#1}}
\newcommand{\refLs}[1]{Lemmas~\ref{#1}}
\newcommand{\refR}[1]{Remark~\ref{#1}}
\newcommand{\refS}[1]{Section~\ref{#1}}
\newcommand{\refSs}[1]{Sections~\ref{#1}}
\newcommand{\refSS}[1]{Section~\ref{#1}}

\newcommand{\refE}[1]{Example~\ref{#1}}

\newcommand{\refApp}[1]{Appendix~\ref{#1}}
\newcommand{\refApps}[1]{Appendices~\ref{#1}}

\newcommand\nopf{\qed}   

\DeclareMathOperator*{\sumx}{\sum\nolimits^{*}}
\DeclareMathOperator*{\sumxx}{\sum\nolimits^{**}}

\newcommand{\sumjo}{\sum_{j=0}^\infty}
\newcommand{\sumj}{\sum_{j=1}^\infty}
\newcommand{\sumko}{\sum_{k=0}^\infty}
\newcommand{\sumk}{\sum_{k=1}^\infty}
\newcommand{\summo}{\sum_{m=0}^\infty}
\newcommand{\summ}{\sum_{m=1}^\infty}
\newcommand{\summk}{\sum_{m=0}^k}
\newcommand{\summl}{\sum_{m=0}^\ell}
\newcommand{\sumno}{\sum_{n=0}^\infty}
\newcommand{\sumn}{\sum_{n=1}^\infty}
\newcommand{\sumik}{\sum_{i=1}^k}
\newcommand{\sumim}{\sum_{i=1}^m}
\newcommand{\sumin}{\sum_{i=1}^n}

\newcommand{\sumkn}{\sum_{k=1}^n}
\newcommand{\prodik}{\prod_{i=1}^k}
\newcommand{\prodim}{\prod_{i=1}^m}

\newcommand\set[1]{\ensuremath{\{#1\}}}

\newcommand\xpar[1]{(#1)}
\newcommand\bigpar[1]{\bigl(#1\bigr)}
\newcommand\Bigpar[1]{\Bigl(#1\Bigr)}
\newcommand\biggpar[1]{\biggl(#1\biggr)}
\newcommand\lrpar[1]{\left(#1\right)}
\newcommand\bigsqpar[1]{\bigl[#1\bigr]}
\newcommand\Bigsqpar[1]{\Bigl[#1\Bigr]}
\newcommand\biggsqpar[1]{\biggl[#1\biggr]}
\newcommand\lrsqpar[1]{\left[#1\right]}

\newcommand\abs[1]{|#1|}
\newcommand\bigabs[1]{\bigl|#1\bigr|}
\newcommand\Bigabs[1]{\Bigl|#1\Bigr|}

\def\rompar(#1){\textup(#1\textup)}    
\newcommand\xfrac[2]{#1/#2}

\newcommand\parfrac[2]{\lrpar{\frac{#1}{#2}}}

\newcommand\Bigparfrac[2]{\Bigpar{\frac{#1}{#2}}}

\def\xexp(#1){e^{#1}}
\newcommand\ceil[1]{\lceil#1\rceil}
\newcommand\floor[1]{\lfloor#1\rfloor}

\newcommand\ntoo{\ensuremath{{n\to\infty}}}
\newcommand\Ntoo{\ensuremath{{N\to\infty}}}

\newcommand\ktoo{\ensuremath{{k\to\infty}}}

\newcommand\ttoo{\ensuremath{{t\to\infty}}}

\newcommand\bmin{\wedge}
\newcommand\norm[1]{\|#1\|}

\newcommand\downto{\searrow}
\newcommand\upto{\nearrow}
\newcommand\half{\tfrac12}

\newcommand\punkt{.\spacefactor=1000}    
    
\newcommand\ie{i.e\punkt}
\newcommand\eg{e.g\punkt}

\newcommand\cf{cf\punkt}
\newcommand{\as}{a.s\punkt}
\newcommand{\aex}{a.e\punkt}

\newcommand\ii{\mathrm{i}}

\newcommand{\tend}{\longrightarrow}
\newcommand\dto{\overset{\mathrm{d}}{\tend}}
\newcommand\pto{\overset{\mathrm{p}}{\tend}}
\newcommand\asto{\overset{\mathrm{a.s.}}{\tend}}
\newcommand\eqd{\overset{\mathrm{d}}{=}}

\newcommand\bbR{\mathbb R}
\newcommand\bbC{\mathbb C}

\newcommand\bbZ{\mathbb Z}

\newcounter{CC}
\newcommand{\CC}{\stepcounter{CC}\CCx} 
\newcommand{\CCx}{C_{\arabic{CC}}}     
\newcommand{\CCdef}[1]{\xdef#1{\CCx}}     
\newcommand{\CCreset}{\setcounter{CC}0} 
\newcounter{cc}
\newcommand{\cc}{\stepcounter{cc}\ccx} 
\newcommand{\ccx}{c_{\arabic{cc}}}     
\newcommand{\ccdef}[1]{\xdef#1{\ccx}}     

\renewcommand\Re{\operatorname{Re}}
\renewcommand\Im{\operatorname{Im}}

\newcommand\E{\operatorname{\mathbb E{}}}
\renewcommand\P{\operatorname{\mathbb P{}}}
\newcommand\Var{\operatorname{Var}}
\newcommand\Cov{\operatorname{Cov}}

\newcommand\Exp{\operatorname{Exp}}
\newcommand\Po{\operatorname{Po}}
\newcommand\Bi{\operatorname{Bi}}

\newcommand\Ge{\operatorname{Ge}}

\newcommand\ga{\alpha}
\newcommand\gb{\beta}
\newcommand\gd{\delta}
\newcommand\gD{\Delta}
\newcommand\gf{\varphi}
\newcommand\gam{\gamma}
\newcommand\gG{\Gamma}
\newcommand\gk{\varkappa}

\newcommand\go{\omega}

\newcommand\gs{\sigma}
\newcommand\gss{\sigma^2}
\newcommand\gS{\Sigma}
\newcommand\gth{\theta}
\newcommand\eps{\varepsilon}

\newcommand\cA{\mathcal A}

\newcommand\cF{\mathcal F}

\newcommand\cH{\mathcal H}
\newcommand\cI{\mathcal I}

\newcommand\cL{{\mathcal L}}

\newcommand\cP{\mathcal P}

\newcommand\cS{{\mathcal S}}
\newcommand\cT{{\mathcal T}}

\newcommand\ett[1]{\boldsymbol1_{#1}}

\newcommand\qw{^{-1}}
\newcommand\qww{^{-2}}
\newcommand\qq{^{1/2}}
\newcommand\qqw{^{-1/2}}
\newcommand\qqq{^{1/3}}

\newcommand\intii{\int_{-1}^1}
\newcommand\intoi{\int_0^1}
\newcommand\intoo{\int_0^\infty}
\newcommand\intoooo{\int_{-\infty}^\infty}
\newcommand\oi{[0,1]}
\newcommand\ooo{[0,\infty)}

\newcommand\dtv{d_{\mathrm{TV}}}

\newcommand\dd{\,\mathrm{d}}
\newcommand\ddx{\mathrm{d}}
\newcommand\ddxx{\frac{\ddx}{\ddx x}}
\newcommand{\pgf}{probability generating function}

\newcommand{\chf}{characteristic function}
\newcommand{\ui}{uniformly integrable}

\newcommand\lhs{left-hand side}
\newcommand\rhs{right-hand side}

\newcommand\GW{Galton--Watson}
\newcommand\GWt{\GW{} tree}
\newcommand\cGWt{conditioned \GW{} tree}

\newcommand\tX{{\widetilde X}}
\newcommand\tY{{\widetilde Y}}
\newcommand\spann[1]{\operatorname{span}(#1)}
\newcommand\tn{\cT_n}
\newcommand\tnv{\cT_{n,v}}
\newcommand\tnV{\cT_{n,V}}
\newcommand\xT{\mathfrak T}
\newcommand\rea{\Re\ga}
\newcommand\rga{{\Re\ga}}
\newcommand\rgb{{\Re\gb}}
\newcommand\wgay{{-\ga-\frac12}}
\newcommand\qgay{{\ga+\frac12}}
\newcommand\rqgay{{\rga+\frac12}}
\newcommand\ex{\mathbf e}
\newcommand\hX{\widehat X}

\newcommand\sgrt{simply generated random tree}
\newcommand\hh[1]{d(#1)}
\newcommand\WW{\widehat W}
\newcommand\coi{C\oi}
\newcommand\out{\gd^+}
\newcommand\zne{Z_{n,\eps}}
\newcommand\ze{Z_{\eps}}
\newcommand\gatoo{\ensuremath{\ga\to\infty}}
\newcommand\rtoo{\ensuremath{r\to\infty}}
\newcommand\Yoo{Y_\infty}
\newcommand\bes{R}
\newcommand\tex{\tilde{\ex}}
\newcommand\tbes{\tilde{\bes}}
\newcommand\Woo{W_\infty}
\newcommand{\hm}{m_1}
\newcommand{\thm}{\tilde m_1}
\newcommand{\bbb}{B^{(3)}}
\newcommand{\rr}{r^{1/2}}
\newcommand\coo{C[0,\infty)}
\newcommand\coT{\ensuremath{C[0,T]}}
\newcommand\expx[1]{e^{#1}}
\newcommand\gdtau{\gD\tau}
\newcommand\ygam{Y_{(\gam)}}
\newcommand\EE{V}
\newcommand\pigsqq{\sqrt{2\pi\gss}}
\newcommand\pigsqqw{\frac{1}{\sqrt{2\pi\gss}}}
\newcommand\gapigsqqw{\frac{(\ga-\frac12)\qw}{\sqrt{2\pi\gss}}}
\newcommand\gdd{\frac{\gd}{2}}
\newcommand\gdq{\frac{\gd}{4}}

\newcommand\eit{e^{\ii t}}
\newcommand\emit{e^{-\ii t}}
\newcommand\tgf{\widetilde\gf}
\newcommand\txi{\tilde\xi}
\newcommand\intT{\frac{1}{2\pi}\int_{-\pi}^\pi}
\newcommand\intpi{\int_{-\pi}^\pi}
\newcommand\Li{\operatorname{Li}}
\newcommand\xq{\setminus\set{\frac12}}
\newcommand\xo{\setminus\set{0}}
\newcommand\tqn{t/\sqrt n}
\newcommand\intpm[1]{\int_{-#1}^{#1}}
\newcommand\gnaxt{g_n(\ga,x,t)}
\newcommand\gaxt{g(\ga,x,t)}
\newcommand\gssx{\frac{\gss}2}
\newcommand\tq{\tilde q}
\newcommand\gao{\ga_0}
\newcommand\ppp{\cP_1}
\newcommand\tpsi{\widetilde\psi}
\newcommand\tgD{\tilde\gD}
\newcommand\xinn{\xi_{n-1,N}}
\newcommand\zzn{\frac12+\ii y_n}
\newcommand\tgdn{\tgD_N}
\newcommand\xgdn{\gD^*_N}

\newcommand\act{|\cT|}
\newcommand\gna{g_{n,\ga}}
\newcommand\hna{h_{n,\ga}}
\newcommand\hnax{\hna^*}
\newcommand\doi{D_{1}}
\newcommand\cHoi{\cH(\doi)}
\newcommand\db{D}
\newcommand\dbm{D_-}
\newcommand\dbmx{\widehat D}
\newcommand\dbmB{D_-^B}
\newcommand\dbmxB{\widehat{D}^B}
\newcommand\DB{D^B}
\newcommand\bDB{\overline{\DB}}
\newcommand\Dbad{D_0}
\newcommand\Dbadm{D_-}
\newcommand\Dbadx{D^*}
\newcommand\OHD{O_{\cH(\Dbad)}}
\newcommand\OHDx{O_{\cH(\Dbadx)}}
\newcommand\Dgdx{D}  
\newcommand\DEgdx{D^\gd}
\newcommand\VZ{V}
\newcommand\GDD{$\gD$-domain}
\newcommand\gdaf{\gda{} function}
\newcommand\gda{$\gD$-analytic}
\newcommand\xxl{^{(\ell)}}
\newcommand\xxll{^{(\ell_1,\ell_2)}}
\newcommand\xxllo{^{(\ell,\ell)}}
\newcommand\xxm{^{(m)}}

\newcommand\ctn{\cT_n}
\newcommand\oz[1]{o\bigpar{|1-z|^{#1}}}
\newcommand\Oz[1]{O\bigpar{|1-z|^{#1}}}
\newcommand\Ozo[1]{O\bigpar{|1-z|^{#1}+1}}

\newcommand\kk{\kappa}
\newcommand\kkk{\chi}

\newcommand\REA{\Re A}
\newcommand\yz{\bigpar{y(z)}}
\newcommand\gax{\ga'}
\newcommand\bga{\overline\ga}
\newcommand\QX{\mathbf{X}}
\newcommand\NNo{\set{0,1,2,\dots}}
\newcommand\dF{\widecheck F}
\newcommand\zddz{\vartheta}
\newcommand\ddz{\frac{\ddx}{\dd z}}

\newcommand\gaoo{\ga_\infty}
\newcommand\gthx{\theta_1}
\newcommand\gthy{\theta_2}
\newcommand\vpx{\varpi}
\newcommand\xxi{\tau}

\newcommand{\Holder}{H\"older}
\newcommand\CS{Cauchy--Schwarz}
\newcommand\CSineq{\CS{} inequality}
\newcommand{\Levy}{L\'evy}
\newcommand{\Takacs}{Tak\'acs}

\begin{document}

\begin{abstract} 
We study the additive functional $X_n(\ga)$ on \cGWt{s} given, for arbitrary complex~$\ga$, by summing the $\ga$th power of all subtree sizes.  Allowing complex~$\ga$ is advantageous, even for the study of real~$\ga$, since it allows us to use powerful results from the theory of analytic functions in the proofs.  

For $\rea < 0$, we prove that $X_n(\ga)$, suitably normalized, has a complex normal limiting distribution; moreover, as processes in~$\ga$, the weak convergence holds in the space of analytic functions in the left half-plane.  We establish, and prove similar process-convergence extensions of, limiting distribution results for~$\ga$ in various regions of the complex plane.  We focus mainly on the case where $\rea > 0$, for which $X_n(\ga)$, suitably normalized, has a limiting distribution that is \emph{not} normal but does not depend on the offspring distribution~$\xi$ of the \cGWt, assuming only that $\E \xi = 1$ and $0 < \Var \xi < \infty$.  Under a weak extra moment assumption on~$\xi$, we prove that the convergence extends to moments, ordinary and absolute and mixed, of all orders.

At least when $\rea > \frac12$, the limit random variable $Y(\ga)$ can be expressed as a function of a normalized Brownian excursion.  
\end{abstract}

\maketitle

\section{Introduction and main results}\label{S:intro}

In the study of random trees, one important part is the study of
\emph{additive functionals}. These are functionals of rooted trees 
of the type
\begin{equation}\label{F}
  F(T):=\sum_{v\in T} f(T_v),
\end{equation}
where $v$ ranges over all nodes of the tree $T$, $T_v$ is the subtree
consisting of $v$ and all its descendants, and $f$ is a given functional of
trees, often called the \emph{toll function}.
Equivalently, additive functionals may be defined by the recursion
\begin{equation}\label{F2}
  F(T):=f(T)+\sum_{i=1}^d F(T_{v(i)}),
\end{equation}
where $d$ is the degree of the root $o$ of $T$ and $v(1),\dots,v(d)$ are the
children of $o$.
(All trees in this paper are rooted.)

We are mainly interested in the case when $T=\tn$ is some random tree of  
order $|\tn|=n$, and we study asymptotics of $F(\tn)$ as \ntoo.
Such problems have been studied by many authors, for different classes of
functionals $f$ and different classes of random trees $\cT_n$; 
some examples are
\cite{
HwangR,
FillK04,
FillFK, 
FillK05,
SJ296, 
Wagner15,
SJ285, 
Delmas18,
RWagner19,
SJ347, 
Delmas20, 
Caracciolo20}. 

In the present paper we consider the case where 
the toll function is $f_\ga(T):=|T|^\ga$ for some constant $\ga$, and
$\tn$ is a \cGWt,
defined by some offspring distribution $\xi$ with $\E\xi=1$ and
$0<\Var\xi<\infty$; 
see \refS{SSGW}
for definitions and note that this includes for example uniformly random
labelled trees, ordered trees, and binary trees.
(We use these standing assumptions on $\tn$ and~$\xi$ throughout the
paper, whether said explictly or not.)
Some previous papers dealing with this situation, in varying generality, are
\cite{
FillK04,
FillFK, 
Delmas18,
Delmas20,
Caracciolo20}.
We denote the corresponding additive functional \eqref{F} by $F_\ga$;
thus $F_\ga(T)$ is the sum of the $\ga$th power of all subtree sizes for $T$.
We also introduce the following notation:
\begin{align}
  X_n(\ga)&:=F_\ga(\tn):=\sum_{v\in\tn}|\tnv|^\ga,\label{Xn}
\\
\tX_n(\ga)&:=X_n(\ga)-\E X_n(\ga).\label{tX}
\end{align}
Note that for $\ga=0$, we trivially have $X_n(0)=F_0(\tn)=n$.
The case $\ga=1$ yields, as is well known, 
the \emph{total pathlength},
see \refE{Ega=1}.

Previous papers have studied
the case when~$\ga$ is real, but 
we consider these variables for arbitrary complex~$\ga$. 
This is advantageous, even for the study of real~$\ga$, 
since it allows us to use powerful results from the
theory of analytic functions in the proofs. 
We also find new phenomena for non-real $\ga$ (for example \refT{TEgd}).
Note that $X_n(\ga)$ and $\tX_n(\ga)$ are random
entire functions of $\ga$, for any given $n$. [The expectation in
\eqref{tX} exists because, for a given $n$, the variable $X_n(\ga)$ takes
only a finite number of different values.]

We begin with the case $\rea<0$, where $X_n(\ga)$ is asymptotically normal
as an easy consequence of \cite[Theorem~1.5 and Remark~1.6]{SJ285}.
More precisely, the following holds. 
(Proofs of this and other theorems stated here are given later.)
We say that a complex random variable $\zeta$ is \emph{normal}
if $(\Re\zeta,\Im\zeta)$ has a two-dimensional normal distribution.
(See \cite[Section 1.4]{SJIII}, and note that a real normal variable is a
special case.)

\begin{theorem}\label{T<0}
Let $\tn$ be a \cGWt{} defined  by an offspring distribution $\xi$ with
$\E\xi=1$ and $0<\gss:=\Var\xi<\infty$.
Then there exists a family of centered complex
normal random variables $\hX(\ga)$,
$\rea<0$, such that, as \ntoo,
\begin{equation}\label{t0}
 n^{-1/2}\tX_n(\ga)=  \frac{X_n(\ga)-\E X_n(\ga)}{\sqrt n}\dto \hX(\ga),
\qquad \rea<0.
\end{equation}
Moreover, $\hX(\ga)$ is a (random) analytic function of $\ga$, and 
the convergence \eqref{t0} holds in the space $\cH(H_-)$ of analytic
functions in the left half-plane $H_-:=\set{\ga:\rea<0}$.
Furthermore,
\begin{equation}\label{t0symm}
\overline{\hX(\ga)}=\hX(\overline{\ga}),
\qquad \ga\in H_-.  
\end{equation}
The covariance function $\E\bigpar{\hX(\ga)\hX(\gb)}$ is an analytic
function of two variables $\ga,\gb\in H_-$, and, as \ntoo, 
\begin{equation}\label{t0cov}
 n\qw \Cov\bigpar{X_n(\ga),X_n(\gb)}\to \E\bigsqpar{\hX(\ga)\hX(\gb)},
\qquad \ga,\gb\in H_-.
\end{equation}
\end{theorem}

The convergence in $\cH(H_-)$ 
means  uniform convergence on compact sets and
implies joint convergence for different $\ga$
in \eqref{t0}; see \refSS{SSanalytic}.

The distribution of the limit $\hX(\ga)$ depends on the offspring
distribution $\xi$ in a rather complicated way. 
Since the variables $\hX(\ga)$ are complex normal, 
and \eqref{t0symm} holds, 
the joint distribution of all $\hX(\ga)$ is determined by the covariance
function $\E\bigpar{\hX(\ga)\hX(\gb)} $, $\ga,\gb\in H_-$. We give a formula
for this in \eqref{r0cov}, but we do not know any simple way to evaluate it.

In most parts of the paper we assume $\rea>0$.
We introduce a normalization that will turn out to be correct 
for $\rea>0$ and define
\begin{align}
Y_n(\ga)&:=n^{-\ga-\frac12}X_n(\ga),\label{Yn}
\\
\tY_n(\ga)&:=n^{-\ga-\frac12}\tX_n(\ga)=Y_n(\ga)-\E Y_n(\ga)
.\label{tY}
\end{align}

Then the following holds. 

\begin{theorem}\label{T1}
There exists a family of complex
random variables $\tY(\ga)$, $\rea>0$, such that if 
$\tn$ is a \cGWt{} defined  by an offspring distribution $\xi$ with
$\E\xi=1$ and $0<\gss:=\Var\xi<\infty$, then, as \ntoo,
\begin{equation}\label{t1}
\gs n^{-\ga-\frac12}\tX_n(\ga)=  \gs\tY_n(\ga)\dto \tY(\ga),
\qquad \rea>0.
\end{equation}
Moreover, $\tY(\ga)$ is a (random) analytic function of $\ga$, and 
the convergence \eqref{t1} holds in the space $\cH(H_+)$ of analytic
functions in the right half-plane $H_+:=\set{\ga:\rea>0}$.
\end{theorem}

Here $\tY(\ga)$ is \emph{not} normal. [In fact, it follows from \eqref{Y}
and \eqref{tx1} below that if $\ga>\frac12$, then $\tY(\ga)$ is bounded below.]
On the other hand, note that the family $\tY(\ga)$ does \emph{not} depend on the
offspring 
distribution $\xi$; it is the same for all \cGWt{s} satisfying our
conditions
$\E\xi=1$ and $0<\gss<\infty$, and thus the asymptotics of $\tX_n$ depends
on $\xi$ only through the scaling factor $\gs$.
Hence, we have universality of the limit when $\rea>0$, but not when $\rea<0$.

We can add moment convergence to \refT{T1}, at least
provided we add a weak extra moment assumption.

\begin{theorem}\label{T1mom}
Assume, in addition to the conditions on $\xi$ in \refT{T1}, that
$\E\xi^{2+\gd}<\infty$ for some $\gd>0$.
Then, the limit \eqref{t1} holds with all moments,
ordinary and absolute.
In other words,  if\/ $\rga>0$, then $\E|\tY(\ga)|^r<\infty$ for every
$r<\infty$; 
furthermore, for any integer $\ell\ge1$,
  \begin{equation}\label{t1mom}
 n^{-\ell(\qgay)}\E\bigsqpar{\tX_n(\ga)^\ell}
= \E\bigsqpar{\tY_n(\ga)^\ell}
\to  \gs^{-\ell}\E\bigsqpar{ \tY(\ga)^\ell},
\qquad \rea>0
,\end{equation}
and similarly for absolute moments and mixed moments of $\tX_n(\ga)$
and $\overline{\tX_n(\ga)}$.

Moreover, for each fixed $\ell$, \eqref{t1mom} 
and its analogues for absolute moments and mixed moments
hold uniformly for $\ga$ in
any fixed compact subset of $H_+$;
the limit $\E\tY(\ga)^\ell$ is an analytic
function of $\ga\in H_+$ while absolute moments and
mixed moments of $\tY(\ga)$ and $\overline{\tY(\ga)}$ 
are continuous functions of $\ga\in H_+$.
\end{theorem}

The result extends to joint moments for several $\ga\in H_+$.
The moments of 
$\tY(\ga)$ may be computed by \eqref{Y} and the recursion
formula
\eqref{kk1}--\eqref{kk2} below.
Note that $\tY(\ga)$ is centered: $\E\tY(\ga)=0$; this follows, \eg, by the case $\ell=1$ of \eqref{t1mom}.
See also 
\refR{Rcent} and \refE{EVar}.

\begin{remark}\label{RT1mom}
 We conjecture that \refT{T1mom} holds also without the extra moment  condition.
Note that even without that condition, \eqref{t1mom} holds for
$\ga\neq\frac12$ as a simple consequence of \refT{TXmom} below.
The case $\ga=\frac12$ is more complicated, but has been treated directly in
the special case $\xi\sim\Bi(2,\frac12)$ (binary trees) 
by \cite{FillK04}; that special case satisfies 
$\E \xi^r<\infty$ for every $r$, 
but it seems likely that the proof in \cite{FillK04} 
can be adapted to the general case by arguments similar to those in \refS{Smom}.
However, we have not pursued this and leave it as an open problem.
See also \cite{Caracciolo20}.
\end{remark}

Theorems \ref{T<0} and \ref{T1} are stated for the centered variables
$\tX_n(\ga)$. 
We obtain results for $X_n(\ga)$ by combining Theorems \ref{T<0}--\ref{T1}
with the  
asymptotics for the expectation $\E X_n(\ga)$ given in the next theorem, but
we first need more notation.

Let $\cT$ be the \GWt{} (without conditioning) defined by the offspring
distribution $\xi$; see \refSS{SSGW}.
It follows from \eqref{pct} that $f_\ga(\cT)=|\cT|^\ga$ has a finite
expectation if and only if $\rea<\frac12$, and we define
\begin{equation}
  \label{mu}
\mu(\ga):=\E f_\ga(\cT)=\E|\cT|^\ga
=\sumn n^\ga\P(|\cT|=n),
\qquad \rea<\tfrac12.
\end{equation}
This is an analytic function in the half-plane $\rea<\frac12$. Note that
$\mu(\ga)$ depends on the offspring distribution $\xi$, although we do not
show this in the notation.
Note also that $\mu(\ga)$ has a singularity at $\ga=\frac12$; in fact, it
is easily seen from \eqref{pct} that
\begin{equation}
  \label{aaa}
\mu(\ga)\sim \frac{(2\pi\gss)\qqw}{\frac12-\ga}, \qquad 
\text{as  $\ga\upto\tfrac12$}.
\end{equation}

\begin{remark}\label{Rmua}
In \refS{Smua} (\refT{TM}), 
we show by a rather complicated argument 
that although $\mu(\ga) \to \infty$ as $\ga \upto \frac12$ (se \eqref{aaa}),   
$\mu(\ga)$ has a continuous extension to all other points on the line 
$\rea = \frac12$.
\end{remark}

It is shown by  \citet{Aldous-fringe} 
that if we construct
a random fringe tree $\tnV$ by first choosing a random 
\cGWt{} $\cT_n$ as above, and then a random node $V$ in the tree,
then $\tnV$ converges in distribution as \ntoo{}
to the random \GWt{} $\cT$.
This was sharpened in \cite[Theorem 7.12]{SJ264}
to the corresponding 'quenched' result:
the conditional distribution of $\tnV$ given $\cT_n$ converges in
probability to the distribution of $\cT$.
As a consequence (see \refS{Sfringe}), we obtain the following
results, which show the central role of $\mu(\ga)$ in the study of
$X_n(\ga)$. 
\begin{theorem}\label{Tfringe}
  \begin{thmenumerate}
  \item \label{TfringeE}
If\/ $\rea\le0$, then as \ntoo,
\begin{equation}
  \label{tfringe}
\E X_n(\ga) = \mu(\ga) n + o(n).  
\end{equation}

  \item \label{TfringeP}
If\/ $\rea \le 0$, then $X_n(\ga)/n \pto\mu(\ga)$.
  \end{thmenumerate}
\end{theorem}

The following theorem improves and extends the estimate \eqref{tfringe}; in
particular, note that [in parts \ref{TE<0} and \ref{TE-}]
the error term in \eqref{tfringe} is improved to
$o\bigpar{n\qq}$ for $\rea<0$ and $O\bigpar{n\qq}$ for $\rea=0$.

\begin{theorem}\label{TE}
The following estimates hold as \ntoo, in all cases uniformly for $\ga$ in
compact subsets of the indicated domains.
  \begin{romenumerate}
  \item \label{TE<0}
If\/ $\rea<0$, then 
  \begin{equation}\label{te<0}
	\E X_n(\ga)
	=\mu(\ga)n+o\bigpar{n\qq}.
  \end{equation}
\item \label{TE-}
  If\/ $-\frac12<\rea<\frac12$, then
  \begin{equation}\label{te-}
	\E X_n(\ga)
	=\mu(\ga)n+\frac{1}{\sqrt{2}\gs}\frac{\gG(\ga-\frac12)}{\gG(\ga)}\,n^\qgay
+o\bigpar{n^{(\rea)_++\frac12}}.
  \end{equation}
  \item \label{TE+}
  If\/ $\rea>\frac12$, then
  \begin{equation}\label{te+}
	\E X_n(\ga)
=\frac{1}{\sqrt{2}\gs}\frac{\gG(\ga-\frac12)}{\gG(\ga)}n^\qgay 
+ o\bigpar{n^\qgay}. 
  \end{equation}
  \item \label{TE=}
  If\/ $\ga=\frac12$, then
  \begin{equation}\label{te=}
	\E X_n(1/2)
=
\frac{1}{\sqrt{2\pi\gss}} n \log n + o\bigpar{n\log n}.
  \end{equation}
  \end{romenumerate}
\end{theorem}

\begin{remark}\label{RTE1/2}
As shown in \refT{T1/2}\ref{T1/2E}, the estimate~\eqref{te-} holds also for
$\ga = \frac12 + \ii y$, $y \neq 0$, where $\mu(\ga)$ is the continuous
extension 
described in \refR{Rmua}.
\end{remark}

Theorems \ref{T<0} and \ref{TE}\ref{TE<0} together yield the following
variant of \eqref{t0}.
\begin{theorem}
  \label{TX<0}
If\/ $\rea<0$, then, as \ntoo, 
  \begin{equation}
\frac{ X_n(\ga)-n\mu(\ga)}{\sqrt n}
	 \dto \hX(\ga).
  \end{equation}
Moreover, this holds in the space $\cH(H_-)$.
\end{theorem}

Similarly, Theorems \ref{T1} and \ref{TE} [parts \ref{TE+} and \ref{TE-}]
yield the following.
We define, for $\rea>0$ and $\ga\neq\frac12$,
the complex random variable
\begin{equation}
  \label{Y}
Y(\ga):=\tY(\ga)+\frac{1}{\sqrt{2}}\frac{\gG(\ga-\frac12)}{\gG(\ga)}.
\end{equation}

\begin{theorem}
  \label{TX}
  \begin{thmenumerate}
  \item   \label{TX>}
If\/ $\rea>\frac12$, then, as \ntoo,
  \begin{equation}\label{tx1}
	  Y_n(\ga):= n^{\wgay} X_n(\ga)
	 \dto \gs\qw Y(\ga).
  \end{equation}
  \item  \label{TX<} 
If\/ $0<\rea<\frac12$, then, as \ntoo,
  \begin{equation}\label{tx<}
n^{\wgay}\bigsqpar{ X_n(\ga)-n\mu(\ga)}
	 \dto \gs\qw Y(\ga).
  \end{equation}
  \end{thmenumerate}Moreover, in both cases, 
this holds in the space $\cH(D)$ for the indicated domain $D$.
\end{theorem}

\begin{remark}\label{RT1/2X}
As shown in \refT{T1/2}\ref{T1/2X}, the limit result~\eqref{tx<} holds also
for $\ga = \frac12 + \ii y$, $y \neq 0$, where $\mu(\ga)$ is the continuous
extension of \refR{Rmua}.
\end{remark}

We can add moment convergence to \refT{TX}, too.

\begin{theorem}\label{TXmom}
The limits \eqref{tx1} and \eqref{tx<} hold with all moments,
for $\rga>\frac12$, and $0<\rga<\frac12$, respectively.
In other words, for any integer $\ell\ge1$,
if\/ $\rga>\frac12$, then
  \begin{equation}\label{mtx1}
\E X_n(\ga)^\ell
= \gs^{-\ell}\E Y(\ga)^\ell n^{\ell(\qgay)}
+ o\bigpar{n^{\ell(\qgay)}} 
,  \end{equation}
and if\/ $0<\rga<\frac12$, then
  \begin{equation}\label{mtx<}
\E\bigsqpar{ X_n(\ga)-n\mu(\ga)}^\ell
=  \gs^{-\ell} \E Y(\ga)^\ell n^{\ell(\qgay)}
+ o\bigpar{n^{\ell(\qgay)}} 
.  \end{equation}
Moreover, in both cases, the moments
$\kk_\ell=\kk_\ell(\ga):=\E Y(\ga)^\ell$ are given by the recursion formula
\begin{align}\label{kk1}
  \kk_1&=\frac{\gG(\ga-\frac12)}{\sqrt2\,\gG(\ga)},
\intertext{and, for $\ell\ge2$, with $\gax:=\ga+\frac12$,}
\kk_\ell&=
\frac{\ell\gG(\ell\gax-1)}{\sqrt2\,\gG(\ell\gax-\frac12)}\kk_{\ell-1}
+\frac{1}{4\sqrt\pi}\sum_{j=1}^{\ell-1}\binom{\ell}{j}
\frac{\gG(j\gax-\frac12)\gG((\ell-j)\gax-\frac12)}{\gG(\ell\gax-\frac12)}
\kk_j\kk_{\ell-j}.
\label{kk2}
\end{align}
\end{theorem}

The result extends to joint moments; see \refSS{SSmom-mix}. 

\begin{remark}\label{RFillK04}
For the case of random binary trees [the case $\xi\sim\Bi(2,\frac12)$]
and real $\ga$,
\refTs{TX} and \ref{TXmom} were shown already by \citet{FillK04}, by
the method used here in \refS{Smom}
to show \refT{TXmom}
(namely, singularity analysis of generating functions and the method of
moments).
Recently (and independently),
the case of uniformly random ordered trees
[$\xi\sim\Ge(\frac12)$, in connection with a study of Dyck paths] has been
shown (also by such methods) 
by \citet{Caracciolo20}, and they have extended their result to general
$\xi$, at least when $\xi$ has a finite exponential moment 
[personal communication].
\end{remark}

\begin{remark}\label{RDelmas2}
\refT{TX}\ref{TX>} has also been shown 
by \citet{Delmas18} (for $\ga>1$, or for full binary trees)
and \citet{Delmas20} (in general).
(They consider only real $\ga$, but their results extend immediately to
complex $\ga$.)
The results in these papers are more general and allow more general
toll functions,
and they show how the result can be formulated in an
interesting way as convergence of random measures defined by the trees; 
moreover, they consider also
more general \cGWt{s}, where $\Var(\xi)$ may be infinite
provided $\xi$ belongs to the domain of attraction of a stable distribution.
We do not consider such extensions here.
\end{remark}

\begin{remark}\label{Rcent}
Centered moments $\E\tY(\ga)^k$ can as always be found from the ordinary
moments given by the recursion above. 
Alternatively, \cite[Proposition 3.9]{FillK04}
gives a (more complicated) recursion formula for the centered moments that
yields them directly. [The formula there is given for real $\ga$, but it
extends to complex $\ga$ with $\rga>0$ by the same proof or by analytic
continuation.
Note also the different normalizations: $Y$ there is our $\sqrt2 Y(\ga)$.]
Another formula for centered moments is given by 
\cite[Proposition 7]{Caracciolo20}
[again with a different normalization: $x_p$ there is our $2\qqw Y(p)$].
\end{remark}

\begin{example}
  \label{EVar}
Consider for simplicity real $\ga>0$. 
It follows from
\eqref{kk1}--\eqref{kk2} that
\begin{align}\label{evar}
  \E \tY(\ga)^2&
 = \Var Y(\ga)
= \kk_2-\kk_1^2
\notag\\&
=\frac{\gG(2\ga)\gG(\ga-\frac12)}{\gG(2\ga+\frac12)\gG(\ga)}
+\frac{\gG(\ga-\frac12)^2}{4\sqrt\pi\gG(2\ga+\frac12)}
-\frac{\gG(\ga-\frac12)^2}{2\gG(\ga)^2}
,\qquad \ga\neq\tfrac12
.\end{align}
Moreover, the moments of $\tY(\ga)$ (which do not depend on $\xi$) are
continuous functions of $\ga$ by \refT{T1mom}, and thus we can obtain
the variance $\Var\tY(\frac12)$ by taking the limit of \eqref{evar} as
$\ga\to\frac12$. 
A simple calculation using Taylor and Laurent expansions of $\gG(z)$
yields, \cf{} \cite[Remark 3.6(c)(iv)]{FillK04},
\begin{align}
\E\tY(\tfrac12)^2=  \Var \tY(\tfrac12) = 
\frac{4\log 2}{\pi}-\frac{\pi}4.
\end{align}
Higher moments of $\tY(\frac12)$ can be calculated in the same way.
The moments of $\tY(\frac12)$ were originally found in 
\cite[Proposition~3.8 and Theorem~3.10(b)]{FillK04}, and given by a recursion
there.
[Note  again that~$Y$ there is our $\sqrt2 Y(\frac12)$.]
See
\cite[Proposition 7 and Table 3]{Caracciolo20}
for another formula and explicit expressions up to order 5
(again with a different normalization).
\end{example}

Theorems \ref{T<0} and \ref{T1},
or \ref{TX<0} and \ref{TX},
show that the asymptotic distribution exhibits a phase transition at
$\rea=0$. 

\begin{remark}\label{Riy}
We do not know how to bridge the gap between the two cases
$\rea<0$ and $\rea>0$. 
Moreover, we do not know the asymptotic distribution, if any, when $\rea=0$
(excepting the trivial case $\ga=0$ when $X_n(0)=n$ is deterministic),
although we note that \refT{Tfringe}\ref{TfringeP} yields a weaker result
on convergence in probability.
However,
we conjecture that $(n\log n)\qqw \tX_n(\ii t)$ converges in distribution to
a symmetric complex normal distribution, for any $t\neq0$.
\end{remark}

\begin{problem}\label{Piy}
Does $X_n(\ii t)$ have an asymptotic distribution, after suitable normalization,
for (fixed and real) $t \neq 0$? If so, what is it?
\end{problem}

\begin{remark}\label{R0}
For real $\ga\downto0$, \eqref{kk1}--\eqref{kk2} show that 
$\E Y(\ga)^2\to0$,
and thus $Y(\ga)\pto0$. [See also~\eqref{evar}.]
As remarked in \cite[Remark 3.6(e)]{FillK04}, one can use
\eqref{kk1}--\eqref{kk2} and the method of moments to show that
\begin{align}\label{r0}
\ga\qqw Y(\ga)\dto N(0,2-2\log 2),
\qquad \ga\downto0.  
\end{align}

If we consider complex $\ga$ with $\rga>0$, and let $\ga\to 0$ 
from  various different
directions, then $\ga\qqw Y(\ga)$ converges in
distribution to various different limits, each of which has a certain
complex normal distribution;
see \refApp{A0}. 

If we
instead let
$\ga\to \ii t$ with
$t\neq0$ real, then \eqref{kk1}--\eqref{kk2} imply that the (complex)
moments $\E Y(\ga)^\ell$ converge. However, the absolute moment
$\E|Y(\ga)|^2\to\infty$ by a similar calculation; see \eqref{e|2|}.  
It can be shown, again by the method of moments, that in this case, 
$(\rga)\qq Y(\ga)$ converges in distribution to a symmetric complex normal
distribution; see
\refApp{Ait}.
As a consequence, the imaginary axis
is \as{} a natural boundary for the random analytic functions $Y(\cdot)$
and $\tY(\cdot)$
\ie, they have no analytic extension to any 
larger domain; see again \refApp{Ait} for details. 
\end{remark}

Theorems \ref{TE} and \ref{TX}
show another phase transition at $\rea=\frac12$; 
this phase transition comes from the behavior of the mean $\E X_n(\ga)$,
while the fluctuations $\tX_n(\ga)$ vary analytically by \refT{T1}.
To be precise, there is a singularity at $\ga=\frac12$, as shown by 
\eqref{aaa} together with \eqref{te-} or \eqref{tx<}.
For non-real $\ga$ on the line $\rea=\frac12$, the situation is more
complicated.
As said in Remarks \ref{Rmua}, \ref{RTE1/2}, and \ref{RT1/2X}, 
the results for $\rea<\frac12$ extend continuously to 
$\Re\ga=\frac12$, $\ga\neq\frac12$.
Moreover, the next theorem (\refT{TEgd}) shows that
if we add a weak moment
assumption on $\xi$,  
then 
we can extend Theorems~\ref{TE} and~\ref{TX} \emph{analytically} 
across the line $\rea=\frac12$, 
and also refine the result at the exceptional case $\ga=\frac12$.
[The results now depend on~$\xi$ through more than just
$\gss$, see \eqref{tegd-c}.]
Hence, assuming a higher moment, there is a singularity at $\ga=\frac12$ but
no other singularities at the line $\rea=\frac12$.
However, in general (without higher moments), $\mu(\ga)$ \emph{cannot} be
extended analytically across the line $\rea = \frac12$, see \refT{Tbad}; 
hence, in general the entire line $\rea=\frac12$ is 
a singularity---in other words, a phase transition.

\begin{theorem}
  \label{TEgd}
Suppose that $\E \xi^{2+\gd}<\infty$ for some $\gd\in(0,1]$.
Then:
\begin{romenumerate}
\item \label{TEgdmu}
$\mu(\ga)$ can be analytically continued to a meromorphic function in
  $\rea<\frac12+\gdd$, with a single pole at $\ga=\frac12$ with residue
$-1/\sqrt{2\pi\gss}$.
\item \label{TEgd-}
Using this extension of $\mu(\ga)$, \eqref{te-} holds, uniformly on compact
sets, 
for $-\frac12<\rea<\frac12+\gdd$ with $\ga\neq\frac12$.
\item \label{TEgd=}
For some constant $c$ (depending on the offspring distribution),
  \begin{equation}\label{tegd=}
	\E X_n(\tfrac12)
=
\frac{1}{\sqrt{2\pi\gss}} n \log n + cn + o(n).
  \end{equation}
\end{romenumerate}
\end{theorem}

\begin{remark}\label{Rhighmom}
If $\xi$ has higher moments, then $\mu(\ga)$ can be continued even further:\ see \refT{TH}.
In particular, if $\xi$ has finite moments of all orders, then $\mu(\ga)$
can be continued to a meromorphic function in the entire complex plane
$\bbC$, with poles at $j+\frac12$, $j=0,1,2,\dots$
\hspace{-.07in}
(or possibly a subset 
thereof).
\end{remark}

\begin{theorem}
  \label{TXgd}
Suppose that $\E \xi^{2+\gd}<\infty$ for some $\gd\in(0,1]$.
Then:
\begin{romenumerate}
\item \label{TXgd-}
The limit in distribution \eqref{tx<} holds for all $\ga\in
\Dgdx:=\set{\ga\neq\frac12:0<\rea<\frac12+\gdd}$; 
moreover \eqref{tx<} holds in
$\cH(\Dgdx)$. 
\item \label{TXgd=}
For some constant $c$ (depending on the offspring distribution),
  \begin{equation}\label{txgd=}
n^{-1}\Bigsqpar{ X_n(\tfrac12)-\pigsqqw n\log n }
	 \dto \gs\qw \tY(\tfrac12)+c.
  \end{equation}
\end{romenumerate}
\end{theorem}

The constants $c$ in \eqref{tegd=} and \eqref{txgd=} are equal. 
The proof yields the formula \eqref{tegd-c}.

\begin{remark}
The phase transitions at $\rea=0$ and $\rea=\frac12$ can be explained as
follows. 
Consider for simplicity real $\ga$, when all terms in \eqref{F} are
positive.
The expected number of subtrees $\tnv$ of order $k$ is
roughly $n\P(|\cT|=k)=\Theta(n k^{-3/2})$,
by \cite[Theorem 7.12]{SJ264} (see \refS{Sfringe}) 
and \eqref{pct}.
Hence, if 
$\ga>\frac12$, $\E X_n(\ga)$ is dominated by the rather few large
$\tnv$ of size $\Theta(n)$; there are roughly $\Theta(n\qq)$ such trees,
which explains the order $n^\qgay$ of $\E X_n(\ga)$.
For $\ga<\frac12$, $\E X_n(\ga)$ is dominated by the small 
subtrees $\tnv$, of size
$O(1)$, and this yields the linear behavior of $\E X_n(\ga)$ in \refT{TE}.

For $\ga<0$, the fluctuations, too, are dominated by the small subtrees (as
shown in the proof of \cite[Theorem~1.5]{SJ285});
there are $\approx n$ of these, 
and they are only weakly dependent on each other,
and as a result $X_n(\ga)$ has an asymptotic normal distribution with the
usual scaling.

For $0<\ga<\frac12$, on the other hand, 
the mean $\E X_n(\ga)$ is dominated by the small subtrees as just said,
but 
fluctuations are dominated by the large subtrees
of order $\Theta(n)$. 
(To see this, note that for $\ga>0$ and $\eps>0$, the 
contribution to $X_n(\ga)$ from subtrees of order $\le\eps n$ has
variance $O\bigpar{\eps^{2\ga}n^{2\ga+1}}$ 
by \cite[Theorem 6.7]{SJ285}.)
Hence, we
have the same asymptotic behavior of $\tX_n(\ga)$ as for larger $\ga$. 
The large subtrees  are more
strongly dependent on each other, and lead to a non-normal limit; on the
other hand, asymptotically they do not depend on details in the offspring
distribution. 
\end{remark}

At least when $\rea>\frac12$, the limit random variable $Y(\ga)$ can be
expressed as a function of a normalized Brownian excursion 
$(\ex(t))$. [Recall
that $(\ex(t))$ is a random continuous function on $\oi$; see, \eg{}, \cite{RY}
for a definition.] 
For a function $f$ defined on an interval, define
\begin{equation}
  \label{m}
m(f;s,t):=\inf_{u\in[s,t]}f(u).
\end{equation}
The general representation formula for $\rea>\frac12$ is a little bit
complicated, and  
we give three closely related versions \eqref{wa0}--\eqref{wa2}, where the
first two are related by mirror symmetry and the third, symmetric, formula 
is the average of the two preceding. (See further the proof,
which also gives a fourth formula \eqref{richard}.
The representations \eqref{wa2} and \eqref{wb} were stated in
\cite[(4.2)--(4.3), see also Examples 4.6 and 4.7]{SJ197};
the present paper gives, after a long delay, the proof promised there.)
Note that the integrals in \eqref{wa0}--\eqref{wa2} converge (absolutely) \as{}
when $\rea>\frac12$,
since $\ex(t)$ is \as{} 
\Holder($\gam$)-continuous for every $\gam<\frac12$,
and thus, \eg, $|\ex(t)-m(\ex;s,t)|\le C(t-s)^{\gam}$ for
some random constant $C$.
(This well-known fact follows \eg{} from the corresponding fact for Brownian
motion together with the construction of $\ex$ from the excursions of the
Brownian motion, see \cite[Theorem I.(2.2) and Chapter XII.2--3]{RY}.)

We also give a simpler expression \eqref{wb} valid for $\rea>1$. [The
integral in \eqref{wb} diverges for $\rea\le1$.]
\begin{theorem}
  \label{Tbrown}
  \begin{thmenumerate}
  \item   \label{Tbrown>1/2}
If\/ $\rea>\frac12$, then, jointly for all such $\ga$, 
\begin{align}
Y(\ga)&\eqd
 2\ga \int^{1}_{0}\!t^{\ga - 1}\ex(t)\dd t
\notag
\\&\hskip4em
{}- 2\ga (\ga - 1) \iint\limits_{0<s<t<1} (t - s)^{\ga - 2}
\bigsqpar{{\ex(t)}-{m(\ex;s,t)}}\dd s \dd t 
\label{wa0}
\\
&=
 2\ga \int^{1}_{0}\!(1-t)^{\ga - 1}\ex(t)\dd t
\notag
\\& \hskip4em
{} - 2\ga (\ga - 1) \iint\limits_{0<s<t<1} (t - s)^{\ga - 2}
\bigsqpar{{\ex(s)}-{m(\ex;s,t)}}\dd s \dd t 
\label{wa1}
\\
&=
\ga \int^1_0 \left[ t^{\ga - 1} + (1 - t)^{\ga - 1} \right]\ex(t)\dd t 
\notag
\\&\hskip1em
 {} - \ga (\ga - 1) \iint\limits_{0 < s < t < 1} 
 (t - s)^{\ga - 2} \left[ \ex(s) + \ex(t) - 2 m(\ex;s,t) \right]\dd s\dd t. 
\label{wa2}
\end{align}

\item   \label{Tbrown>1}
If\/ $\rea>1$, we have also the simpler representation
\begin{equation}\label{wb}
Y(\ga)\eqd2\ga (\ga - 1) \iint\limits_{0 < s < t < 1} 
 (t - s)^{\ga - 2}  m(\ex;s,t) \dd s\dd t. 
\end{equation}
  \end{thmenumerate}
\end{theorem}

\begin{example}\label{Ega=1}
  For $\ga=1$, \eqref{wa0}--\eqref{wa2} reduce to 
  \begin{equation}\label{y1}
	Y(1)=2\intoi\!\ex(t)\dd t,
  \end{equation}
twice the \emph{Brownian excursion area}.
In fact, with $\hh{v}$ denoting the depth of a given 
node~$v$, it is easy to see that
\begin{equation}\label{xn1}
  X_n(1)=\sum_{v\in\tn}|\tnv|=\sum_{v\in\tn}(\hh{v}+1)
=n+\sum_{v\in\tn}\hh{v},
\end{equation}
\ie, $n$ plus the \emph{total pathlength}. The convergence of the
total pathlength, suitably rescaled, to the Brownian excursion area was
shown by \citet{AldousII,AldousIII}, see also \cite{SJ146}. 
The Brownian excursion area has been studied by many authors in various
contexts, for example
\cite{Lou:kac,Lou,Takacs:Bernoulli,Takacs:rooted,Takacs:binary,Spencer,FPV,FL:Airy,SJ133},
see also \cite{SJ201} and the further references there.

Furthermore, for $\ga=2$, \eqref{wb} reduces to 
\begin{equation}\label{y2}
  Y(2)\eqd 4 \iint\limits_{0 < s < t < 1}\!m(\ex;s,t) \dd s\dd t. 
\end{equation}
This too was studied in \cite{SJ146}, where $Y(2)$ was denoted $\eta$.
Moreover, the random variable 
$P(\ctn)$ there equals $X_n(1)-n$,
$Q(\ctn)$  equals $X_n(2)-n^2$,
and the Wiener index $W(\ctn)=nP(\ctn)-Q(\ctn)$ equals $nX_n(1)-X_n(2)$.
Hence, the limit theorem \cite[Theorem 3.1]{SJ146} follows from 
\refTs{TX} and \ref{Tbrown}.

Moreover, as noted by \cite{FillK04},
\refT{TXmom} yields for $\ga=1$ 
a recursion formula for the moments of the Brownian excursion area, 
which is equivalent to the formulas given by
\cite{Takacs:Bernoulli,Takacs:rooted,Takacs:binary,FPV,FL:Airy},
see also \cite[Section 2]{SJ201}.
Similarly, also noted by \cite{FillK04},
\refT{TXmom} yields for $\ga=2$ the recursion formula for moments of $Y(2)$
given in \cite{SJ146}.  More generally, the recursion in \cite{SJ146}
for mixed moments of $Y(1)$ and $Y(2)$ follows from \refT{Tmix} below.
\end{example}

\begin{remark}\label{Rdelmas}
For $\ga>\frac12$,  
a different (but equivalent) 
representation of the limit $Y(\ga)$ as a function of a
Brownian excursion~$\ex$ is given by
\citet[(1.10) and (2.6)]{Delmas18}.
That representation can also be written as a functional of the Brownian
continuum random tree; see
\citet[Theorem 1.1]{Delmas20}.
\end{remark}

\begin{remark}\label{Rrepr}
As demonstrated in \refS{Spf2}, it follows from the proof of \refT{T1} given in that section
that there exists a representation of
$Y(\ga)$ as a (measurable) functional of~$\ex$ also for $0<\rga\le\frac12$.
However, this is only an  existence statement, and we do not know
any explicit representation.
More precisely, there exists a measurable function 
$\Psi: H_+\times\coi \to\bbC$ such that 
\begin{align}
 \label{rrepr}
Y(\ga) = \Psi(\ga,\ex),
\qquad \rga >0,
\end{align}
where $\ex$ as above is a Brownian excursion.
Moreover, $\Psi(\ga,f)$ is an analytic function of
$\ga\in H_+$ for every $f\in\coi$.
For $\rga>\frac12$, $\Psi(\ga,\ex)$ is \as{} given by the formulas
\eqref{wa0}--\eqref{wa2}, and for $\rga>1$ also by \eqref{wb}.
Hence, in principle, $\Psi(\ga,\ex)$ is given by an analytic extension of 
\eqref{wb} to all $\ga\in H_+$, and such an extension (necessarily unique)
 exists a.s.
(Note that for $\rga<1$, the double integrals in \eqref{wa0}--\eqref{wa2} do
not  converge for every function $\ex\in\coi$, so we can
only claim existence of the extension a.s.)

We  concede that the existence of an analytic extension $\Psi(\ga,\ex)$ gives a
``representation'' of $Y(\ga)$ only in a rather abstract sense.
\end{remark}

\begin{problem}
Find an explicit representation for $Y(\ga)$ as a function of~$\ex$
for $0<\rga<\frac12$, or even for $0<\ga<\frac12$.
\end{problem}

Finally, 
we consider real $\ga$ and let $\ga\to\infty$. 
We show the following asymptotic
result yielding a limit of the limit in \refT{TX};
this improves a result in \cite{FillK04} which shows
the existence of such a limit together with \eqref{EZk}.
Let $B(t)$, $t\ge0$, be a standard Brownian motion, and let
\begin{equation}\label{S}
S(t):=\sup_{s\in[0,t]} B(s)  
\end{equation}
be the corresponding supremum process.

\begin{theorem}
  \label{Tlol}
As $\ga\to+\infty$
along the real axis, 
we  have 
$\ga\qq Y(\ga)\dto \Yoo$, where $\Yoo$ is a random variable with 
the representation
\begin{equation}\label{Z}
  \Yoo = \intoo e^{-t} S(t)\dd t.
\end{equation}
and moments
\begin{equation}\label{EZk}
  \E \Yoo^k = 2^{-k/2}\sqrt{k!}, 
\qquad k\ge0,
\end{equation}
and more generally, for real or complex $r$,
\begin{equation}\label{EZr}
  \E \Yoo^r = 2^{-r/2}\sqrt{\gG(r+1)}, 
\qquad \Re r>-1.
\end{equation}
\end{theorem}

Further representations of $\Yoo$ are given in \eqref{yoo} and \eqref{ytau}.

\begin{remark}\label{R'}
  Since convergence in the space $\cH(D)$ (for a domain $D\subseteq\bbC$)
of a sequence of analytic functions implies convergence of their
derivatives, the results above imply corresponding results for $X_n'(\ga)$
and $Y_n'(\ga)$ (and also for higher derivatives).
Note that $X_n'(\ga)$ is the additive functional given by the toll function
$\frac{\ddx}{\ddx\ga}f_\ga(T)=|T|^\ga\log|T|$. In particular, we have
\begin{align}\label{r'}
  X_n'(0) = \sum_{v\in \cT_n}\log|\cT_{n,v}|
=\log\prod_{v\in \cT_n}|\cT_{n,v}|,
\end{align}
which is known as the \emph{shape functional}, see \eg{}
\cite{Fill96,MM98}.
Unfortunately, because of the phase transition at $\rea=0$,
most of our results do not include $0$ in their domains.
The exception is \refT{TE}\ref{TE-}, which implies
\begin{align}\label{x'0}
  \E X_n'(0) = \mu'(0) n + o\bigpar{n\qq\log n},
\end{align}
where the error term is obtained from \eqref{te-} and Cauchy's estimates 
using the circle $|z|=1/\log n$. 
More precise estimates of $\E X_n'(0)$ have been proved by
\cite{Fill96,FillK04,FillFK} 
[random binary trees, the case $\xi\sim\Bi(2,\frac12)$], 
and
\cite{MM98} (general $\xi$ with an exponential moment);
furthermore, these papers also give results for the variance 
(which is of order $n\log n$). 
Moreover, asymptotic normality of $X_n'(0)$ has been shown in special cases 
by
\citet{Pittel} [random labelled trees, the case $\xi\sim\Po(1)$], 
\citet{FillK04} [random binary trees, the case $\xi\sim\Bi(2,\half)$],
and
\citet{Caracciolo20} [random ordered trees, the case $\xi\sim\Ge(\half)$]. 
We have been able to extend this to general $\xi$, assuming
$\E \xi^{2+\gd}<\infty$ for some $\gd>0$,  by suitable modifications
of the arguments in \refS{Smom} (we might provide details in future work).
It seems to be an open problem to show asymptotic normality 
of $X_n'(0)$
for arbitrary $\xi$ with $0<\Var\xi<\infty$ (and $\E\xi=1$, as always).

Note that although the asymptotic normality of $X_n'(0)$ does not follow
from the results in the present paper, 
it fits well together with \refT{T<0} which shows that
$X_n(\ga)$ is asymptotically normal for every $\ga<0$.
\end{remark}

The contents of the paper are as follows.
\refS{Sprel} contains some preliminaries. 
\refS{Sfringe} gives the simple proof of \refT{Tfringe}.
\refS{Stight} shows two lemmas on tightness,
and \refS{S<0} then gives a short proof of \refT{T<0}.
\refS{SE} is a detailed study of the expectation $\E X_n(\ga)$.
\refS{Sbrown} treats
convergence to Brownian excursion and functions thereof.
\refS{Spf2} gives some remaining proofs.
\refS{Slim} discusses the limit as real $\ga\to+\infty$.
\refSs{Smua} and \ref{S:bad} give proofs and a counterexample, respectively,
for the case $\rea=\frac12$.
\refS{Smom} studies moments and gives proofs of \refTs{T1mom} and \ref{TXmom}.
This section uses a method different from that of the previous sections; the two
methods complement each other and combine in the proof of \refT{T1mom}.
Finally,
\refApp{Amuexamples} 
discusses calculation of $\mu(\ga)$ and gives some examples of it;
\refApp{Apoly} gives a proof of a technical lemma in \refS{Smom},
together with some background on polylogarithms used in the proof;
\refApps{A0} and \ref{Ait}
give proofs of the additional results claimed in \refR{R0}.

\begin{acks}
We are grateful to Nevin Kapur for his contributions to \refS{Smom}; Kapur
also coauthored the related unpublished manuscript~\cite{FillK03}.
The present paper was originally conceived as a joint work including him.

We are also grateful to Lennart Bondeson for helpful comments on the topic of~\refR{Rid}.
\end{acks}

\section{Preliminaries and notation}\label{Sprel}

\subsection{Conditioned Galton--Watson trees}\label{SSGW}

Given a non-negative integer-valued random variable $\xi$, 
with distribution $\cL(\xi)$,
the \emph{\GWt}
$\cT$
with offspring distribution $\cL(\xi)$ is constructed recursively by
starting with a root and 
giving each node a number of children that is a new copy of $\xi$,
independent of the numbers of children of the other nodes.
Obviously, 
only the distribution $\cL(\xi)$
of $\xi$ matters; we abuse language and say also
that $\cT$ has offspring distribution $\xi$.
Furthermore, let $\tn$ be $\cT$ conditioned on having
exactly $n$ nodes; this is called a \emph{\cGWt}.
(We consider only $n$ such that $\P(|\cT|=n)>0$.)


We assume that $\P(\xi=0)>0$, since otherwise the tree $\cT$ is \as{} infinite.
In fact, we consider here only the \emph{critical} case $\E\xi=1$; 
in this case
$\cT$ is \as{} finite (provided $\P(\xi\neq1)>0$).
It is well known that in most cases, but not all,
a \cGWt{} with an offspring distribution $\xi'$ with an expectation
$\E\xi'\neq1$ is equivalent to a \cGWt{} with another offspring distribution
$\xi$ satisfying $\E\xi=1$, so this is only a minor restriction. 
See \eg{}
\cite[Section~4]{SJ264} for details.

We also assume $0<\Var\xi<\infty$ (but usually no higher moment assumptions).

\begin{remark}
  More generally, 
a \emph{simply generated random tree} $T_n$ defined by a given 
sequence of non-negative weights
$(\phi_k)_0^\infty$ is a random ordered tree with $n$ nodes such
that for every ordered tree $T$ with $|T|=n$, the probability $\P(\tn=T)$ is
proportional to $\prod_{v\in T} \phi_{\out(v)}$, 
where $\out(v)$ denotes the outdegree of $v$, see \eg{} \cite{MM} or
\cite[Section 1.2.7]{Drmota}. Every \cGWt{} is a \sgrt, and the converse
holds under a weak condition. In particular, if the generating function
$\Phi(z):=\sumko\phi_k z^k$ has a positive radius of convergence $R$ and
there exists $\tau$ with $0<\tau<R$ and $\tau\Phi'(\tau)=\Phi(\tau)$
(which is a common assumption in studies of \sgrt{s}), then the \sgrt{}
$T_n$ equals a \cGWt{} defined by a suitable $\xi$ with $\E\xi=1$;
furthermore, this $\xi$ has finite moment generating function 
$\E e^{t\xi} < \infty$ at some $t > 0$,
and thus
finite moments of all orders. Again, see \eg{}
\cite[Section~4]{SJ264} for details.
\end{remark}

Let $\xi_1,\xi_2,\dots$ be independent  copies of $\xi$ and define
\begin{equation}
  \label{Sn}
S_n:=\sumin \xi_i.
\end{equation}
It is well known (see \citet{Otter}, or \cite[Theorem 15.5]{SJ264} and the
further references given there)
that for any $n\ge1$,
\begin{equation}\label{ptk}
{\P(|\cT|=n)}=\frac{1}n\P(S_n=n-1).  
\end{equation}
In particular, \eqref{mu} can be written
\begin{equation}
  \label{mua}
\mu(\ga)=\sumn n^{\ga-1}\P(S_n=n-1), 
\qquad \rea<\tfrac12.
\end{equation}

For some examples where exact (and in one case rational) 
values of $\mu(\ga)$ can be computed when~$\ga$ is a negative integer, see
\refApp{Amuexamples}.

Recall that the 
\emph{span} of an integer-valued random variable $\xi$,
denoted $\spann\xi$, 
is the largest integer $h$ such that $\xi\in a+h\bbZ$ 
\as{} for some $a\in\bbZ$; we  consider only $\xi$ with $\P(\xi=0)>0$
and then the span is the largest integer $h$ such that $\xi/h\in\bbZ$ a.s.,
\ie, the greatest common divisor of \set{n:\P(\xi=n)>0}.
(Typically, $h=1$, but we have for example $h=2$ in the case of full binary
trees, when $\xi\in\set{0,2}$.)
The local limit theorem for discrete random variables
can in our setting can be stated as follows;
see, \eg{}, 
\cite[Theorem 1.4.2]{Kolchin} or 
\cite[Theorem VII.1]{Petrov}.

\begin{lemma}[Local limit theorem]\label{LLT}
  Suppose that $\xi$ is an integer-valued random variable with 
$\P(\xi=0)>0$, 
$\E\xi=1$, 
$0<\gss:=\Var\xi<\infty$, and span $h$.
Then, as \ntoo, uniformly in all $m\in h\bbZ$,
\begin{equation}\label{llt}
  \P(S_n=m)=\frac{h}{\sqrt{2\pi\gss n}} \Bigsqpar{e^{-(m-n)^2/(2n\gss)}+o(1)}.
\end{equation}
\nopf
\end{lemma}
In particular, for any fixed $\ell\in\bbZ$, as \ntoo{} with
$n\equiv \ell \pmod h$,
\begin{equation}\label{snn}
  \P(S_n=n-\ell)\sim\frac{h}{\sqrt{2\pi\gss }}\,n\qqw.
\end{equation}

Combining  \eqref{ptk} and \eqref{snn} we see that
\begin{equation}\label{pct}
  \P(|\cT|=n)\sim\frac{h}{\sqrt{2\pi\gss }}\,n^{-3/2}
\end{equation}
as \ntoo{} with $n\equiv1\pmod h$. 
[The probability is 0 when $n\not\equiv1\pmod h$.]

We will for simplicity assume
in some proofs below 
that the span of $\xi$ equals 1;
then \eqref{pct} is valid as $\ntoo$ without restriction.
However, this is just for convenience, and the results hold also for $h>1$,
using standard modifications of the arguments. (We leave these to the
reader, but give sometimes a hint.)

\subsection{Random analytic functions}\label{SSanalytic}

For a domain (non-empty open connected set) $D\subseteq\bbC$, let $\cH(D)$
denote the 
space of all analytic functions on $D$, equipped with the usual topology of
uniform convergence on compact sets; this is a topological vector space
with the topology given by the seminorms $p_K(f):=\sup_{z\in K}|f(z)|$,
with $K$ ranging over all compact subsets of $D$.
The space $\cH(D)$ is a Fr\'echet space, \ie{}, a locally convex space with a
  topology that can be defined by a complete translation-invariant metric,
and it has (by Montel's theorem on normal families) the property that
every closed bounded subset is compact,
see \eg{} \cite[\S1.45]{Rudin-FA}
or \cite[Example 10.II and Theorem 14.6]{Treves}.
Furthermore, $\cH(D)$ is separable.
$\cH(D)$ is thus a Polish space (\ie, a complete separable metric space).
We equip $\cH(D)$ with its  Borel $\gs$-field, and note that this is
generated by the point evaluations $f\mapsto f(z)$, $z\in D$.
[This can be seen by choosing 
an increasing sequence $(K_i)$ of compact sets with $D=\bigcup_i K_i$,
and  a countable dense subset $(f_j)$ of $\cH(D)$, and noting that
then the sets
$U_{i,j,n}:=\set{f:p_{K_i}(f-f_j)<1/n}$ form a countable basis of the
topology of $\cH(D)$; furthermore, each $U_{i,j,n}$ belongs to the
$\gs$-field generated by the point evaluations. We omit the standard details.]
It follows from this and the monotone class theorem that the distribution of
a random function $f$ in $\cH(D)$ is determined by its finite-dimensional
distributions (\ie, the distributions of finite sets of point evaluations).

We can use the general theory in
\eg{} Billingsley \cite{Billingsley} or Kallenberg \cite{Kallenberg}
for convergence in distribution of random
functions in $\cH(D)$.
In particular, recall that a sequence $(X_n)$ of random variables in a
metric space $\cS$ is \emph{tight} if for every $\eps>0$, there exists
a compact subset $K\subseteq\cS$ such that $\P(X_n\in K)>1-\eps$ for
every $n$.
Prohorov's theorem \cite[Theorems 6.1--6.2]{Billingsley},
\cite[Theorem 16.3]{Kallenberg} says that 
in a Polish space, a sequence $X_n$ is tight if and only if
the corresponding sequence of distributions $\cL(X_n)$ is relatively
compact, \ie, each subsequence has a subsubsequence that converges in
distribution.

It is easy to characterize tightness in $\cH(D)$ in terms of tightness of
real-valued random variables.

\begin{lemma} \label{Lhd}
Let $D$ be a domain in $\bbC$, and let $(X_n(z))$ be a sequence 
of random analytic functions on $D$. Then the following are equivalent.
\begin{romenumerate}
\item \label{Lhdt}
The sequence $(X_n(z))$  is tight in $\cH(D)$.
\item \label{LhdK}
The sequence
$(\sup_{z\in K}|X_n(z)|)$ is tight for every compact $K\subset D$.
\item \label{LhdB}
The sequence
$(\sup_{z\in B}|X_n(z)|)$ is tight for every closed disc $B\subset D$.
\end{romenumerate}
\end{lemma}

\begin{proof}
This proof is an easy exercise that we include for completeness.

\ref{Lhdt}$\implies$\ref{LhdK}$\implies$\ref{LhdB} is trivial.

\ref{LhdB}$\implies$\ref{Lhdt}. Assume 
that \ref{LhdB} holds and choose a
sequence of closed discs $B_j\subset D$, $j\ge1$, 
such that the interiors $B_j^\circ$
cover $D$. 
Let $\eps>0$. Then, by \ref{LhdB}, for each $j$ there exists
$M_j<\infty$ such that $\P(\sup_{z\in B_j}|X_n(z)|>M_j)<2^{-j}\eps$.
Let $L:=\set{f\in\cH(D):\sup_{z\in B_j}|f(z)|\le M_j\textrm{\ for all~$j$}}$. 
Each compact subset $K$ of $D$ is covered by a finite collection of open
discs $B_j^\circ$, and it follows that there exists $M_K<\infty$ such that
if $f\in L$, then $p_K(f):=\sup_{z\in K}|f(z)|\le M_K$. 
In other words, $\sup_{f\in L}p_K(f)<\infty$ for each compact $K\subset D$,
which 
says that $L$ is bounded in $\cH(D)$, because the topology is defined
by the seminorms $p_K$
\cite[Proposition 14.5]{Treves}.
Moreover, $L$ is a closed set in $\cH(D)$, and thus
$L$ is compact in $\cH(D)$ by the Montel property mentioned above.
Furthermore, $\P(X_n\notin L)< \sumj 2^{-j}\eps=\eps$.
\end{proof}

This leads to the following simple sufficient condition.

\begin{lemma}
  \label{LL1}
Let $D$ be a domain in $\bbC$ and let $(X_n(z))$ be a sequence of random
analytic functions in $\cH(D)$.
Suppose that there exists a function $\gam:D\to(0,\infty)$, bounded on
each compact subset of $D$, such that 
$\E|X_n(z)|\le \gam(z)$
for every $z\in D$. Then the sequence $(X_n)$ is tight in $\cH(D)$.
\end{lemma}

\begin{proof}
  Let $B\subset D$ be a closed disc. There exists a circle $\gG\subset D$
  such that $B$ lies in the interior of $\gG$. 
If $f\in\cH(D)$, then the value $f(z)$ at a 
point inside $\gG$ can be
expressed by a 
Poisson integral 
$\int_\gG P(z,w) f(w)|\ddx w|$ 
over the circle $\gG$, where $P$ is the 
Poisson kernel.
(This is because analytic functions are harmonic.
See \eg{} \cite[11.4, 11.12, and 11.13]{Rudin-RealAndComplex}.)
Furthermore, the Poisson kernel is continuous, and
thus bounded by some constant $\CC$ for all $z\in B$ and $w\in \gG$.
Consequently, for every $f\in\cH(D)$ we have
\begin{equation}
\sup_{z\in B}  |f(z)|\le \CCx \int_\gG |f(w)|\,|\ddx w|.
\end{equation}
Applying this to $X_n(z)$ and taking the expectation, we obtain
\begin{equation}
  \begin{split}
\E \sup_{z\in B} |X_n(z)|
&
\le \CCx \E \int_\gG |X_n(w)|\,|\ddx w|
= \CCx \int_\gG \E|X_n(w)|\,|\ddx w|
\\&
\le \CCx \int_\gG \gam(w)\,|\ddx w|<\infty.	
  \end{split}
\end{equation}
  Hence the sequence $(X_n)$ satisfies \refL{Lhd}\ref{LhdB} (by Markov's
  inequality), and the
  conclusion follows by \refL{Lhd}. 
\end{proof}

We shall also use the following, which again uses properties of analytic
functions. 

\begin{lemma}\label{Lsub}
Let $D$ be a domain in $\bbC$ and let $E$ be a subset of $D$ that has a
limit point in $D$. (I.e., there exists a sequence $z_n\in E$ of distinct
points and $z_\infty\in D$ such that $z_n\to z_\infty$.)
Suppose that $(X_n)$ is a tight sequence of random elements of $\cH(D)$ and
that there exists 
a family of random variables $\set{Y_z:z\in E}$ such that for each $z\in E$,
$X_n(z)\dto Y_z$ and, moreover, this holds jointly for any finite set of
$z\in E$. 
Then $X_n\dto Y$ in $\cH(D)$, for some random function
$Y(z)\in\cH(D)$. 
Furthermore,
$Y(z)\eqd Y_z$, jointly for any finite set of $z\in E$.
That is, $Y$ restricted to $E$ and $(Y_z)$ have the same finite-dimensional
distributions, and thus have the same distribution as random elements of
$\bbC^E$. 
\end{lemma}

\begin{proof}
  It suffices to consider the case when $E=\set{z_1,z_2,\dots}$ with $z_n\to
  z_\infty\in D$.
The result then is a special case of 
\citet[Lemma 7.1]{SJ185}; in the notation there we take $\cS_1=\cH(D)$,
$\cS_2=\bbC^E$ and let $\phi$ be the obvious restriction map 
$f(z)\mapsto (f(z_i))_{i=1}^{\infty}$; note that $\phi$ is injective by the
standard uniqueness for analytic functions. The assumption of joint
convergence $X_n(z)\dto Y_z$ for any finite subset of $E$ is equivalent to 
the convergence $\phi(X_n)\dto (Y_{z_i})$ in $\bbC^E$, since this space has
the product topology
\cite[p.~19]{Billingsley}. 
The conclusion follows from
\cite[Lemma 7.1]{SJ185}.
\end{proof}

\begin{remark}
  \refL{Lsub} may fail 
if we do not assume joint convergence; \ie, if only
  $X_n(z)\dto Y_z$ for each $z\in E$ separately. For a counterexample, 
let $D=\bbC$ and $E=\set{z:|z|=1}$; further, let
  $U$ be uniformly distributed on the unit circle \set{z:|z|=1},
let $X_{2n}(z):=U$ (a constant function) and $X_{2n+1}(z):=Uz$.
Then $X_n(z)\dto U$ for each fixed $z\in E$, 
and 
$(X_n)$ is tight in $\cH(D)$ by \refL{LL1} with $\gamma(z) := \max\{1,|z|\}$,
but $X_n$ does not converge in
$\cH(\bbC)$; for example, $X_n(0)$ does not converge in distribution.

We do not know whether it would be sufficient to assume $X_n(z)\dto Y_z$ for
each $z\in E$ separately in the case when $E$ contains a non-empty open set. 
\end{remark}

\subsection{Dominated convergence}\label{SSdom}

To show uniformity in $\ga$ of various estimates, we use the following
simple, but perhaps not so well known, version of Lebesgue's dominated
convergence theorem.

\begin{lemma}
  \label{Ldom}
Let $\cA$ be an arbitrary index set.
Suppose that, for $\ga\in\cA$ and $n\ge1$, $f_{\ga,n}(x)$ are measurable
functions on a measure space $(\cS,\cF,\mu)$, and that 
for \aex{} fixed $x\in\cS$, we have
$f_{\ga,n}(x)\to g_\ga(x)$ as \ntoo, uniformly in $\ga\in\cA$. 
Suppose furthermore that
$h(x)$ is an integrable function on $\cS$, such that $|f_{\ga,n}(x)|\le
h(x)$ \aex{} for each $\ga$ and $n$. Then
$\int_{\cS} f_{\ga,n}(x)\dd\mu(x) \to \int_{\cS} g_{\ga}(x)\dd \mu(x)$ as
\ntoo, uniformly in $\ga\in\cA$. 
\end{lemma}

\begin{proof}
Note first that the assumptions imply $|g_{\ga}(x)|\le h(x)$ \aex{} for
each $\ga$;
hence,
$\bigabs{f_{\ga,n}(x)-g_{\ga}(x)}\le 2h(x)$ a.e.
  Let $\ga_n$ be an arbitrary sequence of elements of $\cA$. Then
$\int_{\cS} \bigpar{f_{\ga_n,n}(x)-g_{\ga_n}(x)}\dd\mu(x) \to 0$ as \ntoo{} by
  the standard dominated convergence theorem. The result follows.
\end{proof}

\begin{remark}\label{Rdom}
Suppose that the assumptions of \refL{Ldom} hold, and furthermore 
that $\cA$ is an open
  set in the complex plane and that $g_\ga(x)$ is 
an analytic function of $\ga$
  for every $x\in\cS$, and jointly measurable in $\ga$ and $x$.
Then the limit $G(\ga):=\int_{\cS} g_\ga(x)\dd\mu(x)$
  is an   analytic function of $\ga\in\cA$. 
To see this, note again
that the assumptions imply $|g_\ga(x)|\le h(x)$ \aex{} for
each $\ga$. It follows by dominated convergence that $G(\ga)$ is a
continuous function of $\ga$, and by Fubini's theorem that the line integral
of $G(\ga)$ around the boundary of any closed triangle inside $\cA$ is 0;
hence $G(\ga)$ is analytic by Morera's theorem.
\end{remark}

\subsection{Further notation}

We denote the distance between two nodes $v$ and $w$ in a tree by $d(v,w)$.
Furthermore, we let $\hh{v}:=d(v,o)$ denote the distance from $v$ to the
root $o$; this is usually called the \emph{depth} of $v$.

For two nodes $v,w$ of a rooted tree $T$, $v\prec w$ means that $w$ is a
descendant of $v$. Thus, $w\in T_v\iff w\succeq v$.
Furthermore, $v\land w$ denotes the last common ancestor of $v$ and $w$.
Thus, 
\begin{equation}
  \label{min}
u\preceq v\land w \iff (u\preceq v) \land  (u\preceq w).
\end{equation}
For real numbers $x$ and $y$, $x\land y$ is another notation for 
$\min\xpar{x,y}$.
Furthermore, $x_+:=\max(x,0)$ and $x_-:=-\min(x,0)$.
 
Unspecified limits are as \ntoo.


$C, C_1,\dots$ and $c,c_1,\dots$
denote positive constants (typically with large and small values, respectively), not necessarily the same at
different places. The constants may depend on the offspring distribution
$\xi$; they may also depend on other parameters that are indicated 
as arguments.

\section{The case $\rea\le0$, convergence in probability}
\label{Sfringe}
 
\begin{proof}[Proof of \refT{Tfringe}]
  By \eqref{Xn},
recalling that $V$ is a random node in $\tn$,
  \begin{equation}\label{teffa}
\E\bigpar{f_\ga(\tnV)\mid \tn} 
=\frac{1}n \sum_{v\in\tn}|\tnv|^\ga
= \frac{1}{n} X_n(\ga),
  \end{equation}
and  consequently
  \begin{equation}\label{teffb}
\E f_\ga(\tnV) 
= \frac{1}{n} \E X_n(\ga).
  \end{equation}

The random trees defined in \refS{S:intro} may be regarded as random
elements of the countable discrete set $\xT$ of finite ordered rooted trees.
As noted just before the statement of~\refT{Tfringe} in \refS{S:intro}, 
\citet{Aldous-fringe} shows that $\tnV\dto\cT$, as
random elements of $\xT$.
If $\rea\le0$, then $f_\ga$ is a bounded function on $\xT$, trivially
continuous since $\xT$ is discrete. Hence, it follows from \eqref{teffb} that
\begin{equation}
\frac{1}n X_n(\ga)=\E f_\ga(\tnV) \to \E f_\ga(\cT)=\mu(\ga),
\end{equation}
showing \eqref{tfringe}.

Similarly, by \cite[Theorem 7.12]{SJ264}, 
the conditional distribution of $\tnV$ given $\cT_n$ converges 
(as a random element of the space of probability distributions on $\xT$)
in probability to the distribution of $\cT$, which by \eqref{teffa} yields 
part \ref{TfringeP} of \refT{Tfringe}.
\end{proof}

\section{Tightness}\label{Stight}

Recall the notation at~\eqref{Xn}--\eqref{tX} and \eqref{Yn}--\eqref{tY}.

\begin{lemma}
  \label{LL2}
  \begin{thmenumerate}
  \item \label{LL2<0}
For $\rea<0$ and all $n\ge1$,
$\E|\tX_n(\ga)|^2\le C(\ga) n$,
for some constant $C(\ga)=O\bigpar{1+|\rea|\qww}$; thus $C(\ga)$ is 
bounded on each proper half-space \set{\ga:\Re\ga<-\eps<0}.
  \item \label{LL2>0}
For $\rea>0$ and all $n\ge1$,
$\E|\tX_n(\ga)|^2\le C(\ga) n^{2\rea+1}$
and thus
$\E|\tY_n(\ga)|^2\le C(\ga)$,
for some constant $C(\ga)=O\bigpar{1+(\rea)\qww}$; thus $C(\ga)$ is
bounded on each proper half-space \set{\ga:\Re\ga>\eps>0}.
  \end{thmenumerate}
\end{lemma}

\begin{proof} \CCreset
Recall the notation $f_{\alpha}(T) := |T|^{\alpha}$.
We apply \cite[Theorem 6.7]{SJ285} to 
(the real and imaginary parts of)
the functional $f(T):=f_\ga(T)\cdot\ett{|T|\le n}$.
Since 
$|f(\cT_k)|=|f_\ga(\cT_k)|=|k^\ga|=k^{\rea}$ for $k\le n$,
and $f(\cT_k)=0$ for $k>n$,
this yields
\begin{equation*}
  \begin{split}
  \bigpar{\E|\tX_n(\ga)|^2}^{1/2}
&\le \CC n\qq \Bigpar{\sup_{k\le n}k^{\rea} + \sumkn k^{\rea-1}}
\\&
\le 
\begin{cases}
\CC(\ga) n^{\xfrac12}, & \rea<0,	\CCdef{\CCneg}
\\  
\CC(\ga) n^{\rea+\frac12},	  &\rea>0,
\end{cases}
  \end{split}
\end{equation*}
with 
$\CCneg(\ga)=O\bigpar{1+|\rea|\qw}$ and
$\CCx(\ga)=O\bigpar{1+|\rea|\qw}$.
\end{proof}

\begin{lemma}
  \label{LcH}
  \begin{thmenumerate}
  \item 	\label{LcH<0}
The family of random functions $n\qqw\tX_n(\ga)$ is tight in the space
$\cH(H_-)$. 
  \item 	\label{LcH>0}
The family of random functions $\tY_n(\ga):=n^{-\ga-\frac12} \tX_n(\ga)$ is tight in the space $\cH(H_+)$.
  \end{thmenumerate}
\end{lemma}

\begin{proof}
This is an immediate consequence of Lemmas \ref{LL1} and \ref{LL2}
(and the \CSineq).
\end{proof}

\section{The case $\rea<0$}\label{S<0}

\begin{proof}[Proof of \refT{T<0}]
  For a fixed real $\ga<0$, 
\cite[Theorem 1.5]{SJ285} yields \eqref{t0} with
$\hX(\ga)\sim N\bigpar{0,\gam^2(\ga)}$ for some $\gam^2(\ga)\ge0$.
Furthermore, as remarked in \cite{SJ285},
\cite[Theorem 1.5]{SJ285} extends, by the Cram\'er--Wold device,
to joint convergence for several functionals.
By considering $\Re f_\ga$ and $\Im f_\ga$, we thus obtain \eqref{t0} for
complex $\ga\in H_-$; furthermore, we obtain joint convergence for any
finite set of (real or complex) such $\ga$.
The convergence in $\cH(H_-)$ now follows from Lemmas \ref{Lsub} and
\ref{LcH}\ref{LcH<0}. 

The symmetry \eqref{t0symm} is now obvious, since the corresponding formula for
$X_n(\ga)$ follows trivially from the definition \eqref{Xn}. 
Finally, \eqref{t0cov} follows from \cite[(1.16)]{SJ285} and polarization
(\ie, considering linear combinations).
\end{proof}

\begin{remark}\label{R0cov}
Furthermore, \cite[(1.17)]{SJ285} and polarization yields a formula
for the covariance function, for $\rea,\Re\gb<0$:
{\multlinegap=0pt
\begin{multline}\label{r0cov}
\E\bigpar{\hX(\ga)\hX(\gb)}=
\E\bigpar{f_\ga(\cT)\bigpar{F_\gb(\cT)-|\cT|\mu(\gb)}}
+
\E\bigpar{f_\gb(\cT)\bigpar{F_\ga(\cT)-|\cT|\mu(\ga)}}
\\
-\mu(\ga+\gb)+\bigpar{1-\gs\qww}\mu(\ga)\mu(\gb).
\end{multline}}
\end{remark}

\section{The mean}\label{SE}

\begin{lemma}
  \label{LE}
For any complex $\ga$,
\begin{equation}
  \label{le}
\E X_n(\ga)=n\sumkn \frac{\P(S_{n-k}=n-k)\P(S_k=k-1)}{\P(S_n=n-1)}k^{\ga-1}.
\end{equation}
\end{lemma}
\begin{proof}
By  \cite[Lemma 5.1]{SJ285}, summing over $k$,
\begin{equation}
\E X_n(\ga) = \E F_\ga(\tn)=\sumkn n \frac{\P(S_{n-k}=n-k)}{\P(S_n=n-1)}\E f_{\ga,k}(\cT)
\end{equation}
where $f_{\ga,k}(T):=f_\ga(T)\ett{|T|=k}=k^\ga\ett{|T|=k}$
and thus, using \eqref{ptk},
\begin{equation}
\E f_{\ga,k}(\cT)=k^{\ga}\P(|\cT|=k)
=k^{\ga-1}\P(S_k=k-1).  
\end{equation}
The result follows.
\end{proof}

We now prove \refT{TE}. We begin with part~\ref{TE<0}, which follows from
\cite{SJ285}, and part~\ref{TE+}, which is rather easy.

\begin{proof}[Proof of \refT{TE}\ref{TE<0}]
The estimate \eqref{te<0} is an instance of \cite[(1.13)]{SJ285}, and the
proof in \cite{SJ285} shows that the estimate holds uniformly in each
half-space $\rea<-\eps<0$.	
  \end{proof}

\begin{proof}[Proof of \refT{TE}\ref{TE+}]
We write \eqref{le} as
$\E X_n(\ga)=n^{\ga-\frac12}\sumkn \gna(k)$ where 
\begin{equation}\label{gn}
\gna(k):=\frac{\P(S_{n-k}=n-k)\P(S_k=k-1)}{\P(S_n=n-1)}k^{\ga-1}n^{-\ga+\frac32}.
\end{equation}
Thus, converting the sum in \eqref{le} to an integral by letting 
$k:=\ceil{x n}$, 
\begin{equation}\label{pyret}
n^\wgay \E X_n(\ga)
=n\qw\sumkn \gna(k)
=\intoi \gna(\ceil{xn})\dd x.
\end{equation}
Assume for simplicity $\spann\xi=1$.
[Otherwise, replace $\ceil{xn}$ by $xn$ rounded upwards to the nearest
integer $k\equiv1\pmod {\spann\xi}$, and make minor modifications.]
For any fixed $x\in(0,1)$, it then follows from \eqref{snn}
that as \ntoo,
for any fixed $\ga$ and uniformly for $\ga$ in a compact set,
\begin{equation}\label{glim}
  \gna(\ceil{xn})
\sim \frac{(n-nx)\qqw (nx)\qqw}{{\sqrt{2\pi\gss}}\,n\qqw}
(nx)^{\ga-1}n^{-\ga+\frac32}
= \frac{1}{\sqrt{2\pi\gss}}(1-x)\qqw x^{\ga-\frac{3}2}.
\end{equation}
Furthermore, \eqref{snn} similarly also implies that, for $n$ so large that
$\P(S_n=n-1)>0$, 
\begin{equation}\label{gb}
\abs{  \gna(\ceil{xn})}
\le C (1-x)\qqw x^{-(\rea-\frac{3}2)_-}
\end{equation}
for some constant $C$ 
(depending on the offspring distribution, but not on $\ga$).
Since we assume $\rea>\frac12$,
the \rhs{} of \eqref{gb} is integrable, and thus dominated convergence and
\eqref{glim} yield,  evaluating a beta integral,
\begin{align}
\intoi \gna(\ceil{xn})\dd x
&\to
\frac{1}{\sqrt{2\pi\gss}}
\intoi
x^{\ga-3/2}(1-x)\qqw\dd x
\notag\\&
= \frac{1}{\sqrt{2\pi\gss}}B\bigpar{\ga-1/2,1/2}
= \frac{1}{\sqrt{2\pi\gss}}\frac{\gG(\ga-1/2)\gG(1/2)}{\gG(\ga)}
\notag\\&
= \frac{1}{\sqrt{2}\gs}\frac{\gG(\ga-1/2)}{\gG(\ga)}
.
\end{align}
Moreover, 
using \refL{Ldom},
this holds 
uniformly for $\ga$ in each compact subset of $\set{\ga:\rea>\frac12}$.
The result follows by \eqref{pyret}.
\end{proof}

Before completing the proof of \refT{TE}, we give another lemma with a
related estimate for $\E X_n(\ga)$.
We define, compare \eqref{mu} and \eqref{mua}, for any complex $\ga$,
\begin{equation}
  \label{mun}
\mu_n(\ga):=
\E\bigpar{|\cT|^\ga\ett{|\cT|\le n}}
=
\sumkn k^\ga\P(|\cT|=k)
=\sumkn k^{\ga-1}\P(S_k=k-1).
\end{equation}

\begin{lemma}\label{LEn}
  If\/ $\rea>-\frac12$,
then, as \ntoo,
    \begin{equation}\label{len}
	\E X_n(\ga)
	=n \mu_n(\ga)
+\frac{1}{\sqrt{2}\gs}\left[\frac{\gG(\ga-\frac12)}{\gG(\ga)}
-\frac{\pi\qqw}{\ga-\frac12}\right]
n^\qgay
+o\bigpar{n^{(\rea)_+ +\frac12}}.
  \end{equation}
Moreover, this holds uniformly 
for any compact set of $\ga$ with $\rea>-\frac12$.
\end{lemma}

\begin{remark}\label{Rat1/2}
  For $\ga=\frac12$, the square bracket in \eqref{len} is interpreted by
  continuity.
With $\psi(x):=\gG'(x)/\gG(x)$,
the value at $\frac12$ is easily found to be
$\pi\qqw\bigpar{\psi(1)-\psi(\frac12)}=(2\log2)\pi\qqw$,
using \cite[5.4.12--13]{NIST}.
\end{remark}

\begin{proof}
  This time we use \eqref{le} and \eqref{mun} to obtain, with $\gna(k)$ as in
  \eqref{gn}, \cf{}~\eqref{pyret},
  \begin{equation}\label{winston}
	\begin{split}
  \E X_n(\ga)-n\mu_n(\ga)
& =n^{\ga-\frac12}\sumkn
 \bigsqpar{\gna(k)-n^{\frac{3}2-\ga}k^{\ga-1}\P(S_k=k-1)},
\\&
=n^{\ga-\frac12}\sumkn \hna(k),	  
	\end{split}
  \end{equation}
where, see \eqref{gn},
\begin{equation}\label{emm}
  \begin{split}
\hna(k)&:=\gna(k)-n^{\frac{3}2-\ga}k^{\ga-1}\P(S_k=k-1)
\\&\phantom:
=n^{\frac{3}2-\ga}k^{\ga-1}\P(S_k=k-1)
\lrsqpar{\frac{\P(S_{n-k}=n-k)}{\P(S_n=n-1)}-1}.
  \end{split}
\end{equation}

We use once more \eqref{snn} and see that, 
assuming for simplicity that $\xi$ has span 1,
for any fixed $x\in(0,1)$,
for any fixed $\ga$ and uniformly for $\ga$ in a compact set,
\begin{equation}\label{hlim}
  \hna(\ceil{xn})\to
\frac{1}{\sqrt{2\pi\gss}} x^{\ga-\frac{3}2}
\lrsqpar{(1-x)\qqw-1}.
\end{equation}
Furthermore, 
by \eqref{snn},
for all $n$, $k$, and $\ga$,
\begin{equation}\label{emm1}
  n^{\frac{3}2-\ga}k^{\ga-1}\P(S_k=k-1)
=O\Bigpar{\Bigparfrac{k}{n}^{\rea-\frac{3}2}}.
\end{equation}
If $1\le k\le n/2$, then by \cite[Lemma 5.2(i)]{SJ285},
\begin{equation}\label{emm2}
  \begin{split}
{\frac{\P(S_{n-k}=n-k)}{\P(S_n=n-1)}-1}
=O\parfrac{k}{n}+o\bigpar{n\qqw}
,	
  \end{split}
\end{equation}
and if $n/2<k\le n$, 
then by \cite[Lemma 5.2(ii)]{SJ285},
\begin{equation}\label{emm3}
  \begin{split}
{\frac{\P(S_{n-k}=n-k)}{\P(S_n=n-1)}-1}
=O\lrpar{\frac{n\qq}{(n-k+1)\qq}}.
  \end{split}
\end{equation}

For $k\ge n\qq$, the bound in \eqref{emm2} is $O(k/n)$.
Let $\hnax(k):=\hna(k)\ett{k\ge n\qq}$,
and fix $\ga$ with $\Re\ga>-\frac12$.
Then,
combining \eqref{emm} and  \eqref{emm1}--\eqref{emm3},
for all $n$ and $x\in(0,1)$,
\begin{equation}\label{emm5}
  \hnax(\ceil{xn})  
=
O\bigpar{x^{-(\rea-\frac12)_-}+(1-x)\qqw}.
\end{equation}
This bound is integrable, and thus dominated convergence and \eqref{hlim} yield
\begin{equation}\label{emma}
  \begin{split}
n\qw\sum_{n^{1/2} \leq k \leq n}\hna(k)
&=\intoi \hnax(\ceil{xn})\dd x
\\&
\to \frac{1}{\sqrt{2\pi\gss}} 
\intoi 
x^{\ga-\frac{3}2}\lrsqpar{(1-x)\qqw-1} \dd x.
  \end{split}
\end{equation}
The 
integral on the \rhs{} of
\eqref{emma} converges for any $\ga$ with $\rea>-\frac12$, and defines an
analytic function in that region. If $\rea>\frac12$, we have
\begin{equation}\label{sofie}
  \begin{split}
\intoi 
x^{\ga-\frac{3}2}\lrsqpar{(1-x)\qqw-1}\dd x
&=
\intoi x^{\ga-\frac{3}2}(1-x)\qqw \dd x
-
\intoi x^{\ga-\frac{3}2} \dd x
\\&
=B\bigpar{\ga-\tfrac12,\tfrac12}-\frac{1}{\ga-\frac12}
\\&
=\frac{\gG(\ga-\tfrac12)\gG(\tfrac12)}{\gG(\ga)}-\frac{1}{\ga-\frac12}.
  \end{split}
\raisetag\baselineskip
\end{equation}
  The \rhs{} in \eqref{sofie} is analytic for $\rea>-\frac12$ (with a removable
  singularity at $\ga=\frac12$), and thus by analytic continuation,
\eqref{sofie} holds as soon as
  $\rea>-\frac12$.

By combining \eqref{winston},  \eqref{emma}, and
\eqref{sofie}, we obtain the main terms in \eqref{len}.
However, it remains to show that the terms with $k<n\qq$ in \eqref{winston}
are negligible. 
For this we use again \eqref{emm1} and \eqref{emm2} and obtain
\begin{equation}\label{anna}
  \begin{split}
\sum_{k<\sqrt n} \hna(k)
&
\le C \sum_{k<\sqrt n} \Bigparfrac{k}{n}^{\rea-\frac12}	
+o\bigpar{n\qqw}
\sum_{k<\sqrt n} \Bigparfrac{k}{n}^{-(\rea)_--\frac32}	
\\&
\le C_1(\ga) n^{\frac12(\rea+\frac12)+\frac12-\rea}
+o\bigpar{n^{1+(\rea)_-}}
\\&
=o\bigpar{n^{1+(\rea)_-}}.
  \end{split}
\end{equation}
This shows that the contribution to \eqref{winston} for $k<n\qq$ is
$o\bigpar{n^{\rea+\frac12+(\rea)_-}}=o\bigpar{n^{(\rea)_+ +\frac12}}$,
which completes the proof of \eqref{len}.

Moreover, the estimates \eqref{hlim},
\eqref{emm5}, and~\eqref{anna} hold uniformly in
any compact subset of  $\set{\ga:\rea>-\frac12}$, which 
using \refL{Ldom} gives the uniformity in \eqref{len}.
\end{proof}

\begin{proof}[Proof of \refT{TE}\ref{TE-}]
Assume again for simplicity that $\xi$ has span 1. Then
\eqref{snn} yields
\begin{equation}\label{skk}
\P(S_k=k-1)=\frac{1}{\sqrt{2\pi\gss}} k^{-\frac12}(1+\eps_k),   
\end{equation}
with $\eps_k\to0$ as \ktoo,
and thus, using dominated convergence, for \mbox{$\rea<\frac12$},
\begin{equation}\label{magnus}
  \begin{split}
n^{\frac12-\ga} \bigsqpar{\mu(\ga)-\mu_n(\ga)}
&=
n^{\frac12-\ga}  \sum_{k=n+1}^\infty k^{\ga-1}\P(S_k=k-1)
\\&
=
\frac{1}{\sqrt{2\pi\gss}}
\int_1^\infty \parfrac{\ceil{xn}}{n}^{\ga-\frac{3}2}(1+\eps_{\ceil{xn}})\dd x
\\
\\&
\to	
\frac{1}{\sqrt{2\pi\gss}}
\int_1^\infty x^{\ga-\frac{3}2}\dd x
=
\frac{1}{\sqrt{2\pi\gss}(\frac12-\ga)}.
  \end{split}
\end{equation}
Moreover, by \refL{Ldom}, \eqref{magnus} holds uniformly in every half-plane 
$\rea<b<\frac12$.
The result follows by combining \eqref{len} and \eqref{magnus}.
  \end{proof}

\begin{proof}[Proof of \refT{TE}\ref{TE=}]
By \eqref{mun} and \eqref{skk}, again assuming $\spann\xi=1$,
\begin{equation}\label{aannc}
  \mu_n(\tfrac12)
=\sumkn \frac{1}{\sqrt{2\pi\gss}} k^{-1}(1+\eps_k)
= \frac{1}{\sqrt{2\pi\gss}} \log n+o(\log n),
\end{equation}
and the result follows from \refL{LEn}.
  \end{proof}

\begin{proof}[Proof of \refT{TX<0}]
This follows, as said in the introduction,
immediately from Theorems~\ref{T<0}
and  \ref{TE}\ref{TE<0}.
  \end{proof}

\subsection{Extensions assuming higher moments} 
We first prove \refT{TEgd} where we assume $\E\xi^{2+\gd}<\infty$ for some
$\gd\in(0,1]$. 
For an example (without higher moments) where $\mu(\ga)$ \emph{cannot} be
extended analytically across the line $\ga = \frac12$, see \refT{Tbad} in
\refS{S:bad}. 

\begin{proof}[Proof of \refT{TEgd}]
Assume again for simplicity that $\spann\xi=1$.
Then 
  the assumption $\E\xi^{2+\gd}<\infty$ implies that
\eqref{snn} can be improved to 
  \begin{equation}\label{rrn}
\P(S_n=n-1)=\frac{1}{\sqrt{2\pi\gss}} n\qqw+ r(n),	
  \end{equation}
with
\begin{equation}\label{rn}
  r(n)=O\bigpar{n^{-\frac12-\frac{\gd}{2}}},
\end{equation}
see \cite[Theorem 6.1]{Ibragimov66}, 
\cite[Theorem 4.5.3 and 4.5.4]{IbragimovLinnik}.

\pfitemref{TEgdmu}
Consequently, with~$\zeta(\cdot)$ denoting the Riemann zeta function,
\eqref{mua} yields
\begin{equation}\label{puh}
  \mu(\ga)=\frac{1}{\sqrt{2\pi\gss}}\sumn n^{\ga-\frac32}+ \sumn n^{\ga-1}r(n)
= \frac{\zeta(\frac32-\ga)}{\sqrt{2\pi\gss}}+ \sumn n^{\ga-1}r(n),
\end{equation}
where the final sum by \eqref{rn} converges and is analytic 
in $\ga$
for
\mbox{$\rea<\frac12+\gdd$}. 
It is well known that the Riemann
zeta function can be extended to a meromorphic function in the complex plane,
with a single pole at~$1$ with residue~$1$.
The result follows.
[If $\spann\xi>1$, we use the Hurwitz zeta function
\cite[\S25.11]{NIST} instead of the Riemann zeta function.]

\pfitemref{TEgd-}
Let $\DEgdx:=\set{\ga\neq\frac12:-\frac12<\rea<\frac12+\gdd}$,
$\DEgdx_-:=\set{\ga\in \DEgdx:\rea<\frac12}$, and 
$\DEgdx_+:=\set{\ga\in \DEgdx:\rea>\frac12}$.
Furthermore, fix a compact subset $K$ of $\DEgdx$.
Define, for $k\ge2$ and for $k=1$ and $\rea>\frac12$,
\begin{align}
\label{aka}
  a(k,\ga)&:=k^{\ga-\frac32}
-(\ga-\tfrac12)\qw\bigsqpar{k^{\ga-\frac12}-(k-1)^{\ga-\frac12}},
\\
 b(k,\ga)&:=\pigsqqw a(k,\ga)+k^{\ga-1} r(k).\label{bka}
\end{align}
 Note that
for $k\ge2$ and $\ga\in K$,
by a Taylor expansion,
\begin{align}\label{akao}
  a(k,\ga)&=O\bigpar{k^{\rea-\frac52}},
\end{align}
where the implied constant depends only on~$K$,
and thus, using also \eqref{rn},
\begin{equation}\label{bkb}
 b(k,\ga) = O\bigpar{k^{\rea-\frac32-\gdd}}.
\end{equation}
By \eqref{rrn}, \eqref{aka}, and \eqref{bka}, 
\begin{equation}\label{sb}
k^{\ga-1}\P(S_k=k-1)=\gapigsqqw\bigsqpar{k^{\ga-\frac12}-(k-1)^{\ga-\frac12}}
+b(k,\ga), 
\end{equation}
where either $k\ge2$ or $k\ge1$ and $\ga\in \DEgdx_+$.

It follows from \eqref{akao} that $\sum_{k=2}^\infty a(k,\ga)$ converges for
$\ga\in \DEgdx$ and defines an analytic function there. Furthermore, if $\ga\in
\DEgdx_-$, then, summing the telescoping sum,
\begin{equation}
  \sum_{k=2}^\infty a(k,\ga)
=\zeta\bigpar{\tfrac32-\ga}-1+\bigpar{\ga-\tfrac12}\qw,
\end{equation}
and consequently, by \eqref{puh},
\begin{equation}\label{mux}
  \begin{split}
\mu(\ga)=\pigsqqw\lrpar{\sum_{k=2}^\infty a(k,\ga)+1-\bigpar{\ga-\tfrac12}\qw}
+\sumk k^{\ga-1}r(k).
  \end{split}
\end{equation}
Both sides of \eqref{mux} are analytic in $\DEgdx$, so by analytic continuation,
\eqref{mux} holds for all $\ga\in \DEgdx$ (and also for $\rea\le-\frac12$).
In particular, for $\ga\in \DEgdx_+$, where $a(1,\ga)$ is defined,
\begin{equation}\label{mux+}
  \begin{split}
\mu(\ga)=\pigsqqw{\sum_{k=1}^\infty a(k,\ga)}
+\sumk k^{\ga-1}r(k)
=\sum_{k=1}^\infty b(k,\ga).
  \end{split}
\end{equation}

We now analyze $\mu_n(\ga)$ further. 
First, for $\ga\in K_-:=K\cap \DEgdx_-$, 
using \eqref{sb} and \eqref{bkb},
\begin{equation}\label{elf-}
  \begin{split}
\mu(\ga)-\mu_n(\ga)
&
=\sum_{k=n+1}^\infty k^{\ga-1}\P(S_k=k-1)
\\&
=\sum_{k=n+1}^\infty \biggpar{
\frac{(\ga-\frac12)\qw}{\pigsqq}
\bigsqpar{k^{\ga-\frac12}-(k-1)^{\ga-\frac12}}
+b(k,\ga)}
\\&
=-\frac{(\ga-\frac12)\qw}{\pigsqq} n^{\ga-\frac12}
+O\bigpar{n^{\rea-\frac{1}2-\gdd}}.
  \end{split}
\end{equation}

Next, consider $\ga\in \DEgdx_+$. By \eqref{sb} and \eqref{mux+}, for $\ga\in
K_+:=K\cap \DEgdx_+$,
\begin{equation}\label{elf+}
  \begin{split}
\mu_n(\ga)-\mu(\ga)
&=\sum_{k=1}^n \biggpar{
\frac{(\ga-\frac12)\qw}{\pigsqq}\bigsqpar{k^{\ga-\frac12}-(k-1)^{\ga-\frac12}}
+b(k,\ga)}
-\sumk b(k,\ga)
\\&
=\gapigsqqw n^{\ga-\frac12} -\sum_{k=n+1}^\infty b(k,\ga)
\\&
=\frac{(\ga-\frac12)\qw}{\pigsqq} n^{\ga-\frac12}
+O\bigpar{n^{\rea-\frac{1}2-\gdd}}.
  \end{split}
\raisetag{\baselineskip}
\end{equation}

We have obtained the same estimate for the two ranges in \eqref{elf-} and
\eqref{elf+}, and can combine them to obtain, 
for $\ga\in K_-\cup K_+$,
\begin{equation}\label{elf}
  \begin{split}
\mu_n(\ga)-\mu(\ga)
=\frac{(\ga-\frac12)\qw}{\pigsqq} n^{\ga-\frac12}
+O\bigpar{n^{\rea-\frac{1}2-\gdd}}.
  \end{split}
\end{equation}

Furthermore, for each $n$, $\mu_n(\ga)-\mu(\ga)$ is a continuous function in
$\DEgdx$, and thus \eqref{elf} holds for $\ga\in \overline{K_-\cup K_+}$ by
continuity. 
If $K$ is a closed disc, then $K=\overline{K_-\cup K_+}$, and thus
\eqref{elf} hold  for $\ga\in K$.
In general, any compact $K\subset \DEgdx$ can be covered by a finite union of
closed discs $K_i\subset \DEgdx$,
and it follows that \eqref{elf}
holds uniformly in $\ga\in K$ for each compact $K\subset \DEgdx$. 

Combining \eqref{len} and \eqref{elf}, we obtain
\eqref{te-}, uniformly on each compact $K\subset \DEgdx$.

\pfitemref{TEgd=}
By \eqref{rrn}--\eqref{rn}, \eqref{skk} holds with $\eps_k=O(k^{-\gd/2})$.
Consequently, \eqref{aannc} is improved to 
\begin{equation}\label{annac}
  \mu_n(\tfrac12)
=\sumkn \frac{1}{\sqrt{2\pi\gss}} k^{-1}(1+\eps_k)
= \frac{1}{\sqrt{2\pi\gss}} \log n+c_1+o(1),
\end{equation}
with $c_1=(2\pi\gss)\qqw\bigpar{\gamma+\sumk \eps_k/k}$.
The result \eqref{tegd=} follows from \eqref{len}.
\end{proof}

\begin{remark}
The proof shows, using \refR{Rat1/2}, that the constant $c$ 
in \refT{TEgd}\ref{TEgd=} is given by
\begin{equation}    \label{tegd-c}
c = \sum_{k = 1}^{\infty}\frac{1}{k} 
\bigsqpar{k^{1/2} \P(S_k = k - 1) - \frac{1}{\sqrt{2 \pi \gss}}}
- \frac{1}{\sqrt{2 \pi \gss}} \psi\Bigpar{\frac12},     
\end{equation}
where $\psi(\frac12) = - (2 \log 2 + \gam)$
\cite[5.4.13]{NIST}.
\end{remark}


We next show that \refT{TEgd}\ref{TEgdmu} extends 
to the case $\gd>1$, at least if $\gd$ is an integer.

\begin{theorem}
\label{TH}
If\/ $\E\xi^k<\infty$ for an integer $k\ge3$, then
$\mu(\ga)$ can be continued as a meromorphic function in 
$\rea<\frac{k-1}2$ with 
simple poles at
$\set{\frac{1}2,\frac{3}2,\frac{5}2,\dots}$
(or possibly a subset of these points)
and no other poles.
\end{theorem}

Typically, all these points $\ell-\frac12$ 
(with $1\le \ell<k/2$) are poles; however
in special cases, $\mu(\ga)$ might be regular at some
of these points, see \refE{Eresidue0}.

\begin{proof}
Assume for simplicity that $\spann\xi=1$. 
  In this case, 
see \cite[Theorem VII.13]{Petrov} (with slightly different notation),
\eqref{rrn} can be refined to
\begin{equation}\label{h1}
  \P(S_n=n-1)=\frac{e^{-x^2/2}}{\sqrt{2\pi \gss n}}
\biggsqpar{1+\sum_{\nu=1}^{k-2}\tq_\nu(x) n^{-\nu/2}} 
+o\bigpar{n^{-(k-1)/2}}
\end{equation}
where $x=-1/(\gs\sqrt n)$ and $\tq_\nu$ is a polynomial (independent of $n$)
whose coefficients depend on the cumulants of $\xi$
of order up to $\nu+2$, 
see
\cite[VI.(1.14)]{Petrov} for details. 
The polynomial $\tq_\nu$ is  odd
if $\nu$ is odd, and is even if $\nu$ is even;
hence the term $\tq_\nu(x)n^{-\nu/2}$ is a polynomial in $n\qw$
for every $\nu$, and expanding $e^{-x^2/2}=e^{-1/(2\gss n)}$ 
into its Taylor series
and
rearranging,
we obtain from \eqref{h1}
\begin{equation}\label{h2}
  \P(S_n=n-1)=
\sum_{j=0}^{\floor{k/2}-1} a_j n^{-j-\frac12}
+ r(n)
\end{equation}
with $r(n)=o\bigpar{n^{-(k-1)/2}}$,
for some coefficients $a_j$.
Consequently \eqref{mua} yields, \cf{} \eqref{puh},
\begin{equation}
  \mu(\ga)=
\sum_{j=0}^{\floor{k/2}-1} a_j \zeta\bigpar{\tfrac{3}2+j-\ga}
+\sumn n^{\ga-1}r(n),
\end{equation}
where the final sum is analytic in $\rea<(k-1)/2$, which proves the result.
\end{proof}

\begin{remark}
  The proof of \refT{TH} shows that the residue of $\mu(\ga)$ at
  $\ga=j+\frac12$
(assumed to be less than $\frac{k-1}2$)
  is $-a_j$, where $a_j$ is the coefficient in the expansion \eqref{h2}
and can be calculated from the cumulants
$\gk_2=\gss,\gk_3\dots,\gk_{2j+2}$ of $\xi$.  
For example, see \refT{TEgd}\ref{TEgdmu},
the residue at $\frac12$ is
$-a_0=-1/\sqrt{2\pi\gss}$.
As another example,
a calculation (which we omit) shows that if $k>4$, the residue
at $\frac{3}2$ is
\begin{equation}\label{noso}
  -a_1=
\frac{1}{\sqrt{2\pi\gss}}\Bigpar{
\frac{1}{2\gss} - \frac{\gk_3}{2\gs^{4}}
- \frac{\gk_4}{8\gs^{4}}
+ \frac{5\gk_3^2}{24\gs^6}}.
\end{equation}
\end{remark}

\begin{example}
  Consider the case of uniformly random labelled trees, which is given by
  $\xi\sim\Po(1)$.
In this case, 
\begin{equation}
  \begin{split}
	\P(S_n=n-1)=\P(\Po(n)=n-1)
=\frac{n^{n-1}}{(n-1)!}e^{-n}
  \end{split}
\end{equation}
which by Stirling's formula, see \eg{} \cite[5.11.1]{NIST}, has a
(divergent)
asymptotic expansion that can be written
\begin{equation}\label{bern}
  \begin{split}
\sim (2\pi n)\qqw \exp\Bigpar{-\sumk \frac{B_{2k}}{2k(2k-1)n^{2k-1}}}	
  \end{split}
\end{equation}
where $B_{2k}$ are the Bernoulli numbers.
Expanding the exponential in \eqref{bern}
(as a formal power series), 
we obtain coefficients 
$a_k$ such that for any integer~$J$ we have
\begin{equation}
  \P(S_n=n-1)=\sum_{j=0}^J a_j n^{-j-\frac12} +o\bigpar{n^{-J-\frac12}},
\end{equation}
which is the same as \eqref{h2},
and it follows by the argument above that $\mu(\ga)$ has residue
$-a_j$ at $j+\frac12$.

For example, $a_0=(2\pi)\qqw$, see \eqref{rrn} and \eqref{puh}, 
and $a_1=-\frac{1}{12}(2\pi)\qqw$, showing that $\mu(\ga)$ has a pole with
residue $\frac{1}{12}(2\pi)\qqw$ at $\frac{3}2$.
(This agrees with \eqref{noso} since $\gk_k=1$ for every $k\ge1$.)
\end{example}

\begin{example}\label{Eresidue0}
  We construct an example where $\xi$ is bounded, so \refT{TH} applies for
  every $k$ and $\mu(\ga)$ is meromorphic in the entire complex plane,
and furthermore $\mu(\ga)$ is regular at $\ga=\frac{3}{2}$.

We use three parameters $m$, $s$, and $A$, where $m\ge10$ is a fixed integer
(we may take $m=10$), $s\in[0,m)$, and $A$ is a large integer.
Let $\xi=\xi_{m,s,A}$ take the values $0,1,A,mA$ with the probabilities
\begin{align}
  \P(\xi=mA)&=\frac{s}{2m^2A},\\
  \P(\xi=A)&=\frac{1}{2A},\\
\P(\xi=1)&=\frac12-\frac{s}{2m},\\
\P(\xi=0)&=1-\P(\xi=1)-\P(\xi=A)-\P(\xi=mA).
\end{align}
Then $\E\xi=1$ and $\spann\xi=1$.
Keep $m$ and $s$ fixed, and let $A\to\infty$; then
\begin{align}
 \gss&\sim\E\xi^2\sim \frac{sA}2+\frac{A}2 = \frac{1+s}{2} A,\label{xi2}\\
 \gk_3&\sim\E\xi^3\sim \frac{smA^2}2+\frac{A^2}2 = \frac{1+sm}{2} A^2,
  \label{xi3}\\
 \gk_4&\sim\E\xi^4\sim \frac{sm^2A^3}2+\frac{A^3}2 = \frac{1+sm^2}{2} A^3.
\label{xi4}
\end{align}
Denote the parenthesized factor in \eqref{noso} by $f(m,\ga,A)$.
It follows from \eqref{xi2}--\eqref{xi4} that as $A\to\infty$ with fixed $m$
and $s$, 
\begin{align}\label{fmsa}
  f(m,s,A) = -\frac{1+sm^2}{4(1+s)^2}A + \frac{5(1+sm)^2}{12(1+s)^3}A+o(A)
=\bigpar{g(m,s)+o(1)}A,
\end{align}
where
\begin{align}\label{gms}
  g(m,s):= -\frac{1+sm^2}{4(1+s)^2} + \frac{5(1+sm)^2}{12(1+s)^3}
=
\frac{5(1+sm)^2-3(1+s)(1+sm^2)}{12(1+s)^3}.
\end{align}
For $s=0$, the final numerator in \eqref{gms} is $2>0$, and thus $g(m,0)>0$.
For $s=1$, the final numerator is $5(1+m)^2-6-6m^2<0$, and thus
$g(m,1)<0$. Hence, by \eqref{fmsa}, we may choose a large $A$ such that
$f(m,0,A)>0$ and $f(m,1,A)<0$.
Then, by continuity, these exists $s\in(0,1)$ such that $f(m,s,A)=0$,
and  \eqref{noso} shows that for the corresponding $\xi$, we have 
the residue 0 at $\frac{3}2$, \ie, there is no pole there and
$\mu(\ga)$ is regular at $\frac{3}2$.
\end{example}

\section{Brownian representations}\label{Sbrown}

We use the well-known result by \citet{AldousII,AldousIII} 
that represents a \cGWt{} asymptotically by a Brownian excursion $(\ex(t))$
in the
following way (under the conditions $\E\xi=1$ and $\gss:=\Var\xi<\infty$
that also we assume).
(See also \citet{LeGall} and \citet[Chapter 4.1]{Drmota}.)

Consider the \emph{depth-first walk} on the tree $\tn$; this is a walk
$v(1),\dots,\allowbreak
v(2n-1)$ on the nodes of $\tn$, where $v(1)=v(2n-1)$ is the
root $o$, and each time we come to a node, we proceed to the first unvisited
child of the node, if there is any, and otherwise to the parent.
For convenience, we also define $v(0)=v(2n)=o$.
We define $W_n(i):=\hh{v(i)}$, and extend $W_n$ to the interval $[0,2n]$ by
linear interpolation between the integers. Furthermore, we scale $W_n$ to a
function on $\oi$ by
\begin{equation}\label{WW}
  \WW_n(t):=\gs n\qqw W_n(2nt).
\end{equation}
Then $\WW_n$ is a random continuous funtion on \oi, and is thus a random
element of the Banach space $\coi$. One of the main results of 
\citet[Theorem 23 with Remark 2]{AldousIII}
is that, as random elements of $\coi$,
\begin{equation}\label{WW2e}
  (\WW_n(t))\dto (2\ex(t)).
\end{equation}

We can think of $W_n(t)$ as the position of a worm that crawls on the edges
of the tree, visiting each edge twice (once in each direction).

We define $v(x)$ also for non-integer $x\in[0,2n]$ as 
either $v(\floor x)$ or $v(\ceil x)$, choosing between these two
the node more distant from the root. 
Thus,
\begin{equation}\label{hhvx}
  \hh{v(x)}=\ceil{W_n(x)}.
\end{equation}

For a node $v$, let $i_v':=\min\set{i\ge1:v(i)=v}$ and 
$i_v'':=\max\set{i\le2n-1:v(i)=v}$, \ie, the first and last times that $v$ is
visited (with $i_o'=1$ and $i_o''=2n-1$). Then the subtree 
$\tnv$ is visited during the interval $[i_v',i_v'']$, and
$i_v''-i'_v=2(|\tnv|-1)$.
Let 
\begin{equation}\label{jv}
  J_v:=\set{x\in(0,2n):v(x)\succeq v}.
\end{equation}
Then $J_v=(i_v'-1,i_v''+1)$, and thus $J_v$ is an interval of length
\begin{equation}
  \label{jvt}
|J_v|=i''_v-i'_v+2
=2|\tnv|.
\end{equation}

We can now prove \refT{Tbrown}. 
When $\rea>1$, all four expressions \eqref{wa0}--\eqref{wb} are equivalent
by elementary calculus, so part \ref{Tbrown>1} follows from part \ref{Tbrown>1/2}.
Nevertheless, we begin with a straightforward proof of the
simpler part \ref{Tbrown>1}, and then show how part \ref{Tbrown>1/2} can be
proved by a similar, but more complicated, argument.
Since we have not yet proved convergence of $Y_n(\ga)$, 
we state the result as the following two lemmas.

\begin{lemma}\label{LB1}
  If\/ $\rea>1$, then $Y_n(\ga)\dto \gs\qw Y(\ga)$ as \ntoo, 
with $Y(\ga)$ given
  by \eqref{wb}. 
Moreover, this holds jointly for any finite set of such $\ga$.
\end{lemma}

\begin{proof}
We assume $\rea>1$, and then \eqref{jvt} implies
\begin{equation}
(2|\tnv|)^\ga 
= \iint\limits_{\substack{x,y\in J_v \\ x<y}} \ga(\ga-1)(y-x)^{\ga-2}\dd x \dd y
.
\end{equation}
Hence,
\begin{equation}\label{hamlet}
\begin{split}
2^\ga X_n(\ga)=
\sum_{v\in\tn} (2|\tnv|)^\ga & = 
\iint\limits_{0<x<y<2n} \ga(\ga-1)(y-x)^{\ga-2}
\sum_{v\in\tn} \ett{x,y\in J_v}\dd x\dd y.
\end{split}
\end{equation}
Now, by \eqref{jv} and \eqref{min},
$x,y\in J_v\iff v \preceq v(x)\land v(y)$, and thus
\begin{equation}\label{ofelia}
\sum_{v\in\tn} \ett{x,y\in J_v}
=\#\set{v:v\preceq v(x)\land v(y)}
=\hh{v(x)\land v(y)}+1.
\end{equation}
Furthermore, from the construction of the depth-first walk,
\begin{equation}\label{hhvxy}
\hh{v(x)\land v(y)}=\ceil{m(W_n;x,y)}.  
\end{equation}
recalling the notation \eqref{m}. 
[Actually, $m(W_n;x,y)$ is an integer except when 
$v(x)$ is an ancestor of $v(y)$
or conversely.]
Combining \eqref{hamlet}--\eqref{hhvxy} and \eqref{WW} yield
\begin{align}
2^\ga X_n(\ga)
& = \iint\limits_{0<x<y<2n} \ga(\ga-1)(y-x)^{\ga-2} \bigsqpar{ 
\hh{v(x) \land  v(y)}+1} \dd x\dd y
\notag\\
& = \iint\limits_{0<x<y<2n} \ga(\ga-1)(y-x)^{\ga-2} \bigsqpar{m(W_n;x,y)+O(1)}
\dd x\dd y
\notag\\
& = \ga(\ga-1) (2n)^\ga\iint\limits_{0<s<t<1} (t-s)^{\ga-2} 
\bigsqpar{m(n^{1/2}\gs\qw \WW_n;s,t)+O(1)}\dd s\dd t.
\end{align}
Since $\WW_n\dto 2\ex$ in $C[0,1]$ by \eqref{WW2e},
and the integral below defines a
continuous functional on $\coi$ because $\iint(t-s)^{\ga-2} \dd s \dd t$ converges (absolutely),
it follows that
\begin{align}
\gs n^{-\ga-\frac12}X_n(\ga)
&=
 \ga(\ga-1) \iint\limits_{0<s<t<1} (t-s)^{\ga-2} m(\WW_n;s,t)\dd s\dd t
+O\bigpar{n\qqw}
\notag\\&
\dto
 \ga(\ga-1) \iint\limits_{0<s<t<1} (t-s)^{\ga-2} 
m(2\ex;s,t)\dd s\dd t
=Y(\ga).
\end{align}
In other words, recalling \eqref{Yn},
$\sigma Y_n(\ga)\dto  Y(\ga)$.

Joint convergence for several $\ga$ follows by the same argument.
\end{proof}

\begin{lemma}\label{LB1/2}
  If\/ $\rea>\frac12$, then $Y_n(\ga)\dto \gs\qw Y(\ga)$ as \ntoo, 
with $Y(\ga)$ given
  by \eqref{wa0}--\eqref{wa2}. 
Moreover, this holds jointly for any finite set of such $\ga$.
\end{lemma}

\begin{proof}
Fix $\ga$ with $\rea>\frac12$.
We begin with a calculus fact (assuming only that $\rea > 0$).  
For any $0 < a < b < \infty$,
\begin{equation}
(b - a)^{\ga} = \ga \int^b_a\!x^{\ga - 1}\dd x 
- \ga (\ga - 1) \iint\limits_{0 < x < a < y < b} (y - x)^{\ga - 2}\dd x\dd y.  
\end{equation}
We apply this to the interval $(a, b) = J_v$ in \eqref{jv} and obtain, using
\eqref{jvt},
\begin{equation*}
  \begin{split}
	(2|\tnv|)^\ga &
= \ga \int_{x \in J_v}\!x^{\ga - 1}\dd x
- \ga (\ga - 1) \iint\limits_{0<x<y,\,x\notin J_v,\,y\in J_v} 
(y - x)^{\ga - 2}\dd x\dd y 
  \end{split}
\end{equation*}
and thus, summing over all nodes $v$ of $\tn$,
\begin{multline}
\label{1}
  2^\ga\sum_v |\tnv|^\ga
 = \ga \int^{2 n}_{0}\!x^{\ga - 1}\sum_{v\in \tn}\ett{x\in J_v}\dd x
\\
{} - \ga (\ga - 1) \iint\limits_{0 < x < y < 2 n} (y - x)^{\ga - 2}
\sum_{v\in \tn}\ett{x\notin J_v,\, y\in J_v}\dd x \dd y .
\end{multline}
Now, using \eqref{jv} and \eqref{hhvx},
\begin{equation}
  \begin{split}
\sum_{v\in \tn}\ett{x\in J_v}
&=
\#\set{v: v(x) \succeq v}
= \hh{v(x)}+1
= \ceil{W_n(x)}+1
  \end{split}
\end{equation}
and similarly, using also \eqref{min} and \eqref{hhvxy},
\begin{equation}
  \begin{split}
\sum_{v\in \tn}\ett{x\notin J_v,\, y\in J_v}
&=
\#\set{v: v(x) \not\succeq v\text{ and }v(y) \succeq v}
\\
&=	
\#\set{v: v(y) \succeq v}
-
\#\set{v: v(x)\land v(y) \succeq v}
\\&
= \hh{v(y)}-\hh{v(x)\land v(y)}
\\&
= \ceil{W_n(y)}-\ceil{m(W_n;x,y)}.  
  \end{split}
\end{equation}
Consequently, recalling the definitions \eqref{Xn} and \eqref{Yn}
of $X_n(\ga)$ and $Y_n(\ga)$,
\begin{multline}
\label{polonius}
2^\ga X_n(\ga)
 = \ga \int^{2 n}_{0}\!x^{\ga - 1}\bigpar{\ceil{W_n(x)}+1}\dd x
\\
{} - \ga (\ga - 1) \iint\limits_{0 < x < y < 2 n} (y - x)^{\ga - 2}
\bigpar{\ceil{W_n(y)}-\ceil{m(W_n;x,y)}}\dd x \dd y 
\end{multline}
and thus
\begin{multline}
\label{laertes}
Y_n(\ga)
 = \ga \int^{1}_{0}\!t^{\ga - 1}n\qqw \bigpar{\ceil{W_n(2nt)}+1}\dd t
\\
{} - \ga (\ga - 1) \iint\limits_{0 < s < t < 1} (t - s)^{\ga - 2}
n\qqw \bigpar{\ceil{W_n(2nt)}-\ceil{m(W_n;2ns,2nt)}}\dd s \dd t .
\end{multline}
The first integral in \eqref{laertes} is no problem; it converges (in
distribution) by \eqref{WW} and \eqref{WW2e},
just as the  integral at the end of the proof of \refL{LB1},
because $\int\!t^{\ga-1} \dd t$ converges (absolutely).

The second integral, however, is more difficult, since $\iint(t-s)^{\ga-2} \dd s \dd t$
diverges if $\rea\le 1$. We therefore use a truncation argument.
For $0 < \eps < 1$ we split $Y_n(\ga)=\zne(\ga) + \zne'(\ga)$, where
\begin{multline}
\zne(\ga)
:= \ga \int^{1}_{0}\!t^{\ga - 1}n\qqw \bigpar{\ceil{W_n(2nt)}+1}\dd t
\\
{} - \ga (\ga - 1) \iint\limits_{t-s>\eps} (t - s)^{\ga - 2}
n\qqw \bigpar{\ceil{W_n(2nt)}-\ceil{m(W_n;2ns,2nt)}}\dd s \dd t
\end{multline}
and
\begin{multline}\label{zne'}
  \zne'(\ga):= \\
{} - \ga (\ga - 1) \iint\limits_{ 0<t-s<\eps}\!(t - s)^{\ga - 2}
n\qqw \bigpar{\ceil{W_n(2nt)}-\ceil{m(W_n;2ns,2nt)}}\dd s \dd t .
\end{multline}
For each fixed $\ga$ with $\rea>0$ and each fixed $0 < \eps < 1$, 
\begin{multline}
\gs\zne(\ga)
= \ga \int^{1}_{0}\!t^{\ga - 1}\WW_n(t)\dd t
\\
{} - \ga (\ga - 1) \iint\limits_{t-s>\eps} (t - s)^{\ga - 2}
\bigpar{{\WW_n(t)}-{m(\WW_n;s,t)}}\dd s \dd t 
+ O\bigpar{n\qqw}
\end{multline}
and thus, by \eqref{WW2e} and the continuous mapping theorem,
\begin{multline}\label{falstaff}
\gs\zne(\ga)
\dto
\ze(\ga):=
 2\ga \int^{1}_{0}\!t^{\ga - 1}\ex(t)\dd t
\\
{} - 2\ga (\ga - 1) \iint\limits_{t-s>\eps} (t - s)^{\ga - 2}
\bigpar{{\ex(t)}-{m(\ex;s,t)}}\dd s \dd t 
.
\end{multline}

We now use the assumption $\rea>\frac12$. We define $Y(\ga)$ by \eqref{wa0},
noting that the integrals converge, as said in \refS{S:intro}, because
$\ex(t)$ is \Holder($\gam$)-continuous for every $\gam<\frac12$.
This shows that
as $\eps\to0$, 
\begin{equation}\label{york}
\ze(\ga)\to Y(\ga)  
\end{equation}
\as{} (and thus in distribution).
Furthermore, let $\gb$ be real with $\frac12<\gb<(\rea\land 1)$.
It follows from \eqref{zne'} that
\begin{equation}\label{henryv}
  |\zne'(\ga)|
\le \frac{|\ga(\ga-1)|}{|\gb(\gb-1)|}\,\eps^{\rea-\gb}\zne'(\gb).
\end{equation}
Furthermore, by \eqref{laertes}, $\zne'(\gb)\le Y_n(\gb)$, and by 
\refT{TE}\ref{TE+} we have $\E Y_n(\gb)=O(1)$.
Consequently, \eqref{henryv} implies
\begin{equation}\label{crispin}
\E|Y_n(\ga)-\zne(\ga)|=\E  |\zne'(\ga)| = O\bigpar{\eps^{\rea-\gb}}.
\end{equation}
Consequently, $\zne(\ga)\pto Y_n(\ga)$ as $\eps\to0$ uniformly in $n$, 
\ie, for any $\gd>0$,
$\sup_n\P(|Y_n(\ga)-\zne(\ga)|>\gd)\to0$.
This together with the facts \eqref{falstaff} and \eqref{york} imply
the result $\gs Y_n(\ga)\dto Y(\ga)$, 
see \eg{} \cite[Theorem 4.2]{Billingsley} 
or \cite[Theorem 4.28]{Kallenberg}.
Joint convergence for several $\ga$ follows by the same argument.

It remains to show that \eqref{wa1}--\eqref{wa2} are equal to $Y(\ga)$.
Let us temporarily denote these expressions by $Y^{(1)}(\ga)$ and
$Y^{(2)}(\ga)$. 

Note that, a.s.,
\begin{align}
(\ga - 1) \iint\limits_{t-s>\eps} (t - s)^{\ga - 2}\ex(t)\dd s \dd t 
&=
(\ga - 1) \int_\eps^1 \ex(t)\int_0^{t-\eps} (t - s)^{\ga - 2}\dd s \dd t 
\notag\\& \hskip-2em
=
 \int_\eps^1 \ex(t) \bigpar{t^{\ga - 1}-\eps^{\ga-1}} \dd t 
\notag\\& \hskip-2em
=
\int_0^1 \ex(t) \bigpar{t^{\ga - 1}-\eps^{\ga-1}} \dd t +O\bigpar{\eps^{\rea}}
\end{align}
and hence 
\begin{equation}
  \ze(\ga)=
2\ga\eps^{\ga-1}\int_0^1 \ex(t)  \dd t 
+ 2\ga (\ga - 1) \iint\limits_{t-s>\eps} (t - s)^{\ga - 2}
{m(\ex;s,t)}\dd s \dd t 
+O\bigpar{\eps^{\rea}}.
\end{equation}
Consequently, \eqref{york} yields the formula
\begin{equation}\label{richard}
Y(\ga)=
\lim_{\eps\to0}
\lrpar{
2\ga\eps^{\ga-1}\int_0^1 \ex(t)  \dd t 
+ 2\ga (\ga - 1) \iint\limits_{t-s>\eps} (t - s)^{\ga - 2}
{m(\ex;s,t)}\dd s \dd t }.
\end{equation}
If we replace $\ex(t)$ by the reflected $\ex(1-t)$, the \rhs{} of
\eqref{richard} is unchanged, while $Y(\ga)$ defined by \eqref{wa0} becomes
$Y^{(1)}(\ga)$ defined by \eqref{wa1}. Consequently,
$Y^{(1)}(\ga)=Y(\ga)$ a.s.
Furthermore, $Y^{(2)}(\ga)=[Y(\ga)+Y^{(1)}(\ga)]/2$, and thus also
$Y^{(2)}(\ga)=Y(\ga)$ a.s.
\end{proof}

\begin{remark}
  For $\ga=k\ge2$ integer, an alternative argument uses the following identity,
  obtained by extending \eqref{ofelia} to several nodes: 
  \begin{equation}
	\begin{split}
\sum_{v\in\tn}|\tnv|^k
&=\sum_{v,v_1,\dots,v_k\in\tn}\ett{v_1,\dots,v_k\succeq v}
\\&
=2^{-k}\int_0^{2n}\dotsi\int_0^{2n}\bigpar{d(v(x_1)\land\dotsm\land v(x_k))+1} \dd x_1 \dotsm \dd x_k
\\&
=2^{-k}k!\idotsint\limits_{0<x_1<\dots< x_k<2n}\ceil{m(W_n;x_1,x_k)}\dd
x_1\dotsm \dd x_k
+n^k
\\&
=n^{k}k(k-1)\iint\limits_{0<t_1<
  t_k<1}\ceil{m(W_n;2nt_1,2nt_k)}(t_k-t_1)^{k-2}\dd t_1 \dd t_k
+n^k.
\end{split}
  \end{equation}
This easily shows \refL{LB1} with \eqref{wb} in this case. 
A similar, but simpler, argument yields \eqref{y1} for $k=1$, see \eqref{xn1}.
\end{remark}

\section{Proofs of \refT{T1} and remaining limit theorems}
\label{Spf2}

\begin{proof}[Proof of \refT{T1}]
\refT{T1} now follows from \refL{Lsub}, with $D=H_+$ and
$E=\set{\ga:\rea>1}$, using Lemmas \ref{LcH}\ref{LcH>0} and 
\ref{LB1}. 
\end{proof}

\begin{proof}[Proof of \refR{Rrepr}.]
This is implicit in the proof above, but we add some details.
Let $D$ and $E$ be as in the proof of \refT{T1}.
(Alternatively, take $E:=[2,3]$.)
Let $\gf:\cH(D)\to C(E)$ be
the restriction mapping $f\mapsto f|_E$,
and let $\psi:\coi\to C(E)$ be the mapping
taking $\ex \in C[0, 1]$ to the element of $C(E)$ that maps $\alpha \in E$ to
the \rhs{} of~\eqref{wb};
both~$\gf$ and~$\psi$ are continuous and thus measurable.
Let also $Y$ denote the random function $Y(\ga)\in \cH(D)$.
The proof above (in particular, \refL{Lsub})
shows that $\gf(Y)\eqd \psi(\ex)$,
and thus we may assume
\begin{align}\label{pippi}
\gf(Y) = \psi(\ex) 
\qquad\text{a.s.}   
\end{align}
(The skeptical 
reader might apply \cite[Corollary 6.11]{Kallenberg} for the
last step.)
Furthermore, $\gf$ is injective, and both $\cH(D)$ and $C(E)$ are Polish
spaces;
thus the range $R:=\gf(\cH(D))$ is a Borel set in $C(E)$, and the inverse
function $\gf\qw:R\to \cH(D)$ is measurable,
see \eg{} \cite[Theorem 8.3.7 and Proposition 8.3.5]{Cohn}.
By \eqref{pippi}, we have
$Y = \gf\qw\bigpar{\psi(\ex)}$ a.s. 
Consequently,
\eqref{rrepr} holds with
\begin{align}
  \label{emil}
\Psi(\ga,f):=
  \begin{cases}
\gf\qw\bigpar{\psi(f)}(\ga) , & \psi(f)\in R,
\\
0,& \text{otherwise}.   
  \end{cases}
\end{align}
\end{proof}

\begin{proof}[Proofs of Theorems \ref{TX} and \ref{TXgd}] 
These results
follow immediately from Theorem~\ref{T1} 
and the estimates of $\E X_n(\ga)$ in
Theorems \ref{TE} and \ref{TEgd}.
\end{proof}

\begin{proof}[Proof of \refT{Tbrown}]
\refT{Tbrown} follows from \refT{TX}\ref{TX>} and 
Lemmas \ref{LB1}--\ref{LB1/2}, comparing the limits.
More precisely, this yields equality in distribution jointly for any finite
number of $\ga$, 
which implies equality jointly
for all $\ga$ since
the distribution of $Y(\ga)$ in $\cH(H^+)$ 
is determined by the finite-dimensional distributions, 
see \refSS{SSanalytic}. 
\end{proof}

\section{The limit as $\ga\to\infty$}\label{Slim}

We introduce more notation.
As above, $\ex(t)$, $t\in\oi$, is a normalized Brownian excursion, and
$m(\ex;s,t)$ is defined by \eqref{m}. We further define
\begin{align}\label{mm'}
  m(s):=m\bigpar{\ex;s;\tfrac12},&&&
  m'(s):=m\bigpar{\ex;\tfrac12,1-s}
\end{align}
for $0\le s\le\frac12$; for convenience we extend $m$ and $m'$ to continuous
functions on $\ooo$ by defining $m(s)=m'(s):=m(\frac12)=\ex(\frac12)$ for
$s>\frac12$. 
Furthermore,
\begin{itemize}
\item
$(B(t))$ is a standard Brownian motion on $[0,\infty)$.
\item 
$(S(t):=\sup_{s\in[0,t]}B(s))$ is the corresponding supremum process. 
\item 
$(\tau(a):=\min\set{t: B(t)=a}, a\ge0)$ is the corresponding family of
hitting times.
\item
$(\bes(t))$ is a three-dimensional Bessel process on $[0,\infty)$,
i.e.,\ $(\bes(t))\eqd(|\bbb(t)|)$, where $(\bbb(t)=\bigpar{B_1(t),B_2(t),B_3(t)})$
is a three-dimensional Brownian motion (so $B_1$, $B_2$, $B_3$ are
three independent copies of $B$).
It is well known that a.s.\ $\bes(0)=0$, $\bes(s)>0$ for all $s>0$ and
$\bes(s)\to\infty$ as $s\to\infty$ \cite[\S VI.3]{RY}.
\item
$J(t):=\inf_{s\ge t} \bes(s)$, $t\ge0$, is the future minimum of $\bes$.
By Pitman's theorem \cite[VI.(3.5)]{RY}, as stochastic processes in $\coo$ we have
\begin{equation}
  \label{pitman}
  (J(t)) \eqd (S(t)).
\end{equation}
\item 
$(J'(t), t\ge0)$ is an independent copy of the stochastic process $(J(t))$.
Similarly, $(S'(t))$ is an independent copy of $(S(t))$ and $(\tau'(a))$ is an
independent copy of $(\tau(a))$.
\end{itemize}

For notational convenience, we also define, 
using \eqref{wb},
for $r>-1$, 
\begin{equation}\label{WY}
  W_r:=\iint_{0<s<t<1}(t-s)^rm(\ex;s,t)\dd s\dd t =\frac1{2(r+1)(r+2)}Y(r+2).
\end{equation}
The assertion
$\ga\qq Y(\ga)\dto \Yoo$ in \refT{Tlol} 
is thus equivalent to $r^{5/2}W_r\dto\frac12\Yoo$ as \rtoo.

\begin{lemma}\label{Lmm}
As $r\to\infty$ we have jointly (\ie, bivariately for sequences of processes)
$\rr m(x/r)\dto J(x)$ and	
$\rr m'(x/r)\dto J'(x)$ in $\coT$, for any $T<\infty$.
\end{lemma}
\begin{remark}
  Convergence  in $\coT$ for every fixed $T$ is equivalent to
  convergence in $\coo$, 
see \eg{} \cite[Proposition 16.6]{Kallenberg},
so the conclusion may as well be stated as
  joint convergence in distribution in $\coo$.
\end{remark}

\begin{proof} 
Let us first consider $m$.
We use the representation, see \eg{} \cite[II.(1.5)]{Blum},
\begin{equation}
  \label{e}
\ex(t)\eqd (1-t)\bes\Bigpar{\frac{t}{1-t}}
\end{equation}
as processes on $[0,1)$. Hence, using Brownian scaling, 
for $x\in[0,r)$ we have,
as processes,
\begin{equation}\label{ex}
\rr \ex(x/r)\eqd \Bigpar{1-\frac xr} \rr \bes\Bigpar{\frac{x/r}{1-(x/r)}}
\eqd \Bigpar{1-\frac xr} \bes\Bigpar{\frac{x}{1-(x/r)}}
\end{equation}
and thus, for $x\in[0,r/2]$,
\begin{equation}\label{mr}
\rr m(x/r)
\eqd \min_{x\le t\le r/2} \Bigpar{1-\frac tr} \bes\Bigpar{\frac{t}{1-(t/r)}}.
\end{equation}
Recall that \as{} $R(t)\to\infty$ as \ttoo. Hence, 
given $T$, we can choose a (random) $T_1\ge T$ 
such that $\bes(t)\ge 2\sup_{u\in[T,2T]} \bes(u)$
for all $t\ge T_1$. It follows that if $T_1\le t\le r/2$, then
\begin{equation}
\Bigpar{1-\frac tr} \bes\Bigpar{\frac{t}{1-(t/r)}}
\ge \tfrac12 \cdot2\sup_{u\in[T,2T]} \bes(u)
\ge 
\bes\Bigpar{\frac{T}{1-(T/r)}}.  
\end{equation}
Hence, if $x\le T$ and $r\ge 2T_1$,
the minimum in \eqref{mr} equals the minimum over $x\le t\le T_1$. 
Furthermore, 
as \rtoo,
since $\bes$ is continuous,
\begin{equation}\label{zmr}
\min_{x\le t\le T_1} \Bigpar{1-\frac tr} \bes\Bigpar{\frac{t}{1-(t/r)}}
\to
\min_{x\le t\le T_1} \bes\xpar{t}
=\min_{x\le t<\infty}\bes(t)
=J(x)
\end{equation}
uniformly for  $x\in [0,T]$, \ie{} in $\coT$.
Consequently,
\begin{equation}
\min_{x\le t\le r/2} \Bigpar{1-\frac tr} \bes\Bigpar{\frac{t}{1-(t/r)}}
\asto
J(x)
\end{equation}
in \coT,
and
\eqref{mr} implies
\begin{equation}\label{mx}
  \rr m(x/r) \dto J(x)
\qquad \text{in }
\coT,
\end{equation}
which proves the assertion about $m$.
By symmetry also
\begin{equation}\label{m'x}
  \rr m'(x/r) \dto J(x)\eqd J'(x)
\qquad \text{in }
\coT,
\end{equation}
since $(\ex(1-t)) \eqd (\ex(t))$ and
thus
$(m'(t)) \eqd (m(t))$ (as random functions in $\coo$).

It remains to prove joint convergence to independent limits.
Let
\begin{align}\label{hm}
 \hm(s):=m(\ex; s,r^{-2/3}), &&&
\hm'(s):=m(\ex; 1-r^{-2/3},1-s) 
\end{align}
(for $r$ with $s\le r^{-2/3}\le\frac12$).
We may assume that the left and right sides of \eqref{ex} are equal, 
and
then $m(x/r)=\hm(x/r)$ whenever the minimum in \eqref{mr} equals the minimum
over $t\in[x,r\qqq]$; in particular, this holds if $x\le T$ and $r\qqq\ge
T_1$ defined above. (This implies $r\ge 2r\qqq\ge 2T_1$.) Consequently,
\begin{equation}
  \P\bigpar{m(x/r)=\hm(x/r)
\text{ for all }x\in[0,T]}
\ge \P\bigpar{T_1\le r\qqq} \to1
\end{equation}
as \rtoo. By symmetry, also
\begin{equation}\label{macbeth}
  \P\bigpar{m'(x/r)=\hm'(x/r) \text{ for all }x\in[0,T]}
  \to1.
\end{equation}

Next, we may assume $\bes(t)=|\bbb(t)|$ and that equality holds 
in \eqref{e}. 
Define the modification
$\tbes(t):=|\bbb(t)-\bbb(1)|$ and the corresponding
$\tex(t):=(1-t)\tbes\bigpar{t/(1-t)}$
and 
$\thm'(s):=m(\tex; 1-r^{-2/3},1-s)$.
Then $|\tbes(t)-\bes(t)|\le |\bbb(1)|$ for all $t$, and thus
$|\tex(t)-\ex(t)|\le(1-t)|\bbb(1)|$
and
$  |\thm'(s)-\hm'(s)|\le r^{-2/3}|\bbb(1)|.$
Consequently, assuming $r\qqq\ge T$,
\begin{equation}\label{banquo}
 \sup_{x\le T} |\rr\thm'(x/r)-\rr\hm'(x/r)|\le r^{-1/6}|\bbb(1)| \pto0.
\end{equation}
Let $\rho$ denote the metric in \coT. By \eqref{macbeth} and \eqref{banquo},
\begin{equation}\label{cawdor}
  \rho\bigpar{\rr\thm'(x/r),\rr m'(x/r)}\pto0
\end{equation}
as \rtoo. Thus by \eqref{m'x}, 
\begin{equation}\label{thm'x}
  \rr \thm'(x/r) \dto J(x)
\qquad \text{in }
\coT,
\end{equation}

Now, for $x\le T$ and large $r$, $\thm'(x/r)$ depends only on $\tex(t)$
for $t\ge\frac12$, and thus on $\tbes(t)$ for $t\ge1$. However,
$\bigpar{\tbes(t)=|\bbb(t)-\bbb(1)|,\, t\ge1}$ is independent of 
$\bigpar{\bes(t)=|\bbb(t)|,\,t\le1}$, 
and thus of 
$\bigpar{\ex(t),\,t\le\frac12}$ and of 
$\bigpar{m(s),\, s\le\frac12}$.
Consequently, we can combine \eqref{mx} and \eqref{thm'x} to 
\begin{equation*}
\bigpar{\rr m(x/r),  \rr \thm'(x/r)} \dto \bigpar{J(x),J'(x)}
\qquad \text{in } \coT\times\coT,
\end{equation*}
with independent limits $(J(x))$ and $(J'(x))$.
Finally, the result follows by using \eqref{cawdor} again.
\end{proof}

\begin{lemma}\label{LW}
As $r\to\infty$,
\begin{equation}
  \label{a}
r^{5/2}W_r\dto \Woo:=\iint_{x,y>0} e^{-x-y}\bigpar{J(x)\bmin J'(y)}\dd x\dd y.
\end{equation}
\end{lemma}

\begin{proof}
Note first that for some constant $c$  (in fact, $c=\E|\bbb(1)|=\sqrt{8/\pi}$),
$\E \bes(t) = ct\qq$. Hence, $\E J(x)\le\E \bes(x)= cx\qq$ and
\begin{equation*}
\E \iint_{x,y>0} e^{-x-y}\bigpar{J(x)\bmin J'(y)}\dd x\dd y
\le \iint_{x,y>0} e^{-x-y}cx\qq\dd x\dd y
<\infty.
\end{equation*}
Consequently, the double integral in \eqref{a} converges a.s.

If $s\le \frac12\le t$, then by \eqref{mm'}, $m(\ex;s,t)=m(s)\land m'(1-t)$.
Noting this, we define a truncated version of $W_r$ by,
for $r\ge 2T$ and substituting $s=x/r$ and $t=1-y/r$,
\begin{equation}\label{wrt}
  \begin{split}
W_r^T&:=\iint_{\substack{0<s<T/r\\ 1-T/r<t<1}}	(t-s)^rm(\ex; s,t)\dd s\dd t
\\
&\phantom: =r^{-2}\int_0^T\!\int_0^T\!
\Bigpar{1-\frac{y}{r}-\frac{x}{r}}^r\bigpar{m(x/r)\bmin m'(y/r)}
\dd x\dd y.
  \end{split}
\end{equation}
Since $\bigpar{1-\frac{y}{r}-\frac{x}{r}}^r\to e^{-y-x}$ uniformly for
$x,y\in[0,T]$ as \rtoo, 
it follows from \refL{Lmm} and the continuous mapping theorem that for each
  fixed $T<\infty$, as \rtoo\ we have
\begin{equation}\label{wtoo}
  \begin{split}
r^{5/2}W_r^T&
\dto \Woo^T :=\int_0^T\int_0^T
e^{-x-y}\bigpar{J(x)\bmin J'(y)} \dd x\dd y.
  \end{split}
\end{equation}
Furthermore, $\Woo^T\to \Woo$ \as{} as $T\to\infty$.

Moreover, by \eqref{e},
$\E\ex(t)=ct\qq(1-t)\qq$ and thus $\E m(\ex;s,t)\le \E \ex(s) \le c s^{1/2}$.
Hence, for $r\ge 2T>0$, and again with the substitutions $s=x/r$ and $t=1-(y/r)$, we have
\begin{equation}\label{wwrt}
  \begin{split}
\E (W_r-W_r^T)&=\iint_{\set{T/r<s<t<1}\cup\set{0<s<t<1-T/r}}	(t-s)^r
\E m(\ex;s,t)\dd s\dd t
\\
&\le r^{-2}\iint_{[0,r)^2\setminus[0,T]^2}	
\Bigpar{1-\frac{y}{r}-\frac{x}{r}}_+^r c\parfrac{x}{r}^{1/2}
\dd x\dd y
\\
&\le cr^{-5/2}\iint_{[0,\infty)^2\setminus[0,T]^2}	
e^{-x-y} x^{1/2} 
\dd x\dd y.
  \end{split}
\end{equation}
Hence, 
\begin{equation}
  \limsup_{\rtoo}\E|r^{5/2}W_r-r^{5/2}W_r^T | \to 0
\end{equation}
as $T\to\infty$. This shows, by \cite[Theorem 4.2]{Billingsley} 
or \cite[Theorem 4.28]{Kallenberg} again, 
that we can let $T\to\infty$ inside \eqref{wtoo} and obtain the conclusion
\eqref{a}. 
\end{proof}

\begin{proof}[Proof of \refT{Tlol}]
By \eqref{WY}, \refL{LW} can be written
\begin{equation}
  \ga\qq Y(\ga)\dto \Yoo:=2\Woo
\end{equation}
as \gatoo. We now give some equivalent expressions for the limit.
First, by \eqref{pitman},
\begin{equation}\label{yoo}
  \begin{split}
\Yoo\eqd 2\intoo\!\intoo\!
e^{-x-y}\bigpar{S(x)\land S'(y)} \dd x\dd y.
  \end{split}
\end{equation}

Secondly, note that $\tau(a)\le x\iff S(x)\ge a$; thus $\tau$ and $S$ are
inverses of each other. 
Similarly, we may assume that $\tau'$ is the  inverse of $S'$.
By Fubini's theorem,
\begin{align}
\Yoo
&\eqd2 \int_0^\infty\!\int_0^\infty\! e^{-x-y} \bigpar{S(x)\bmin S'(y)} \dd x\dd y
\notag\\
&= 2\iiint_{0\le s\le S(x) \land S'(y)} e^{-x-y} \dd s \dd x \dd y
\notag\\
&= 2\iiint_{ \tau(s)\le x,\,  \tau'(s)\le y} e^{-x-y} \dd x \dd y \dd s
\notag\\
&= 2\int_0^\infty e^{-\tau(s)-\tau'(s)} \dd s.
\end{align}
However, $(\tau(s))$ and $(\tau'(s))$ are independent  processes with independent
increments, and thus $(\tau(s)+\tau'(s))$ has independent increments.
Furthermore, for each fixed $s$,
$\tau(2s)-\tau(s)\eqd\tau(s)$ and is independent of $\tau(s)$, and
hence $\tau(s)+\tau'(s)\eqd\tau(2s)$.
It follows that the stochastic process $(\tau(s)+\tau'(s))$ 
equals in distribution $(\tau(2s))$.
Hence, we also have the representation
\begin{equation}\label{ytau}
  \Yoo
\eqd 2\int_0^\infty e^{-\tau(2s)} \dd s
= \int_0^\infty e^{-\tau(s)} \dd s.
\end{equation}

The same Fubini argument in the opposite direction now gives
\begin{equation}
  \begin{split}
\Yoo
&\eqd
  \int_0^\infty e^{-\tau(s)} \dd s
 =\iint_{ x \ge \tau(s)} e^{-x} \dd x \dd s
\\&
= \iint_{0\le s\le S(x)} e^{-x} \dd s \dd x
= \intoo e^{-x} S(x) \dd x.	
  \end{split}
\end{equation}
This shows \eqref{Z}.

It remains to calculate the moments of $\Yoo$. 
For integer moments we  use \eqref{ytau}. 
Recall,
see \eg{} \cite[Proposition II.3.7 and Sections III.3--4]{RY}, 
that $\tau(s)$ is a stable process with stationary independent
increments and
\begin{equation}\label{est}
\E \expx{-s \tau(t)} = \expx{-t\sqrt{2s}}, \qquad s,t\ge0.  
\end{equation}
Define $\gdtau(s,s'):=\tau(s')-\tau(s)$.
Then, by symmetry and the change of variables
$t_1=s_1$, $t_2=s_2-s_1$, \dots, $t_k=s_k-s_{k-1}$,
noting that the increments $\gdtau(s_{i-1},s_i)$ are independent
and $\gdtau(s_{i-1},s_i)\eqd\tau(t_i)$ (with $s_0=0$), we have
\begin{align}\label{zzbx}
\E \Yoo^k
&= k! \int_{0<s_1<s_2<\dots< s_k} \E e^{-\tau(s_1)-\dots-\tau(s_k)} 
\dd s_1
\dotsm \dd s_k 
\notag\\
&= k! \int_{0<s_1<s_2<\dots< s_k} 
\E e^{-k\gdtau(0,s_1)-(k-1)\gdtau(s_1,s_2)- \dots-\gdtau(s_{k-1},s_k)} 
\dd s_1 \dotsm \dd s_k 
\notag\\
\notag\\
&= k! \int_{t_1,\dots,t_k>0} 
\E e^{-k\tau(t_1)}\E e^{-(k-1)\tau(t_2)}\dotsm \E e^{-\tau(t_k)} 
\dd t_1\dotsm \dd t_k 
\notag\\
&= k! \int_{t_1,\dots,t_k>0} 
  e^{-t_1\sqrt{2k}-t_2\sqrt{2(k-1)}-\dots-t_k\sqrt2} 
\dd t_1\dotsm \dd t_k 
\notag\\
&= k! \prod_1^k \frac{1}{\sqrt{2j}}
= 2^{-k/2} k!^{1/2},
\end{align}
which is \eqref{EZk}. 

In order to extend this to non-integer moments, 
let 
\begin{equation}
  \label{ZZ}
Z:=\log\Yoo+\frac12\log2, 
\end{equation}
and let $Z'$ be an independent copy of $Z$.
Then, for integer $k\ge1$,
\begin{equation}
  \begin{split}
\E \bigpar{\bigpar{e^{Z+Z'}}^k}
=\E{e^{k(Z+Z')}}
=\bigpar{\E e^{kZ}}^2
= \bigpar{2^{k/2}\E\Yoo^k}^2=k!,
  \end{split}
\end{equation}
and thus $\EE:=e^{Z+Z'}\sim\Exp(1)$, since an exponential distribution is
determined by its moments. Hence, for any real $r>-1$,
\begin{equation}\label{EIIr}
  \begin{split}
	\bigpar{\E e^{rZ}}^2
=\E e^{{rZ+rZ'}}
=
\E \EE^r
=\intoo x^r e^{-x}\dd x=\gG(r+1),
  \end{split}
\end{equation}
and thus
$\E e^{rZ}=\sqrt{\gG(r+1)}$.
Since $e^{rZ}=2^{r/2}\Yoo^r$, \eqref{EZr} follows, for real~$r$.
Finally, \eqref{EZr} is extended to complex $r$ by analytic continuation,
or by \eqref{EIIr} again, now knowing that the expectations exist.
\end{proof}

\begin{remark}
The characteristic function $\gf_Z(t)$
of the random variable $Z$ in \eqref{ZZ}
is thus $\gG(1+\ii t)\qq$, which decreases exponentially as $t\to\pm\infty$;
hence 
$Z$ has by Fourier inversion a continuous density 
\begin{align}
  f_Z(x)=\frac{1}{2\pi}\intoooo e^{-\ii tx}\gf_Z(t)\dd t
=\frac{1}{2\pi}\intoooo e^{-\ii tx}\gG(1+\ii t)\qq\dd t
,\end{align}
see \eg{}
\cite[Theorem XV.3.3]{FellerII}; furthermore, by a standard argument, 
we may differentiate repeatedly under the integral sign, and thus
the density function  $f_Z(x)$ is infinitely differentiable.
(In fact, it follows from Stirling's formula that $\gf_Z(t)=\gG(1+\ii t)\qq$
belongs to the Schwartz class $\cS(\bbR)$
of infinitely differentiable functions such
that every derivative decreases faster than $|x|^{-k}$ for any $k<\infty$;
hence $f_Z\in\cS(\bbR)$, see \cite[Theorem 25.1]{Treves}.)

Consequently, also $\Yoo$ is absolutely continuous, with a density $f_Y(x)$ 
that is
infinitely differentiable on $(0,\infty)$. 
Results on the asymptotics of the density function $f_Y(x)$ of $\Yoo$ as
  $x\to0$ and $x\to\infty$ are given in \cite{FillK04}.
\end{remark}

\begin{remark}\label{Rid}
$2\qq\Yoo$ has moments $\sqrt{k!}$, and it follows that if $\Yoo'$ is
an independent copy of $\Yoo$, then $2\Yoo\Yoo'$ has moments $k!$ and
$2\Yoo\Yoo'\sim\Exp(1)$. Hence, the distribution of 
$2\qq\Yoo$ is a ``square root'' of $\Exp(1)$, in the sense of taking
products of independent variables.

Moreover, if we let $(\tau(s))$ be another stable subordinator, with
$\E e^{-s\tau(t)}=e^{-ts^\gam}$ ($0<\gam<1$) instead of \eqref{est}, 
then \eqref{ytau} defines by the same calculations a random
variable $\ygam$ with 
\begin{equation}
\E \ygam^k=(k!)^{1-\gam}.  
\end{equation}
In particular, choosing $\gam=1-(1/m)$, we obtain an $m^{\rm th}$ root of the
exponential distribution $\Exp(1)$.

Recalling that $V \sim \Exp(1)$ and taking logarithms, this shows that $\log
\EE$ is infinitely divisible, and thus the same holds for $-\log\EE$, which
has a Gumbel distribution.
This has been known for a long time, and a calculation shows that $-\log\EE$
has a \Levy{} measure with a density $\sumj e^{-jx}/x=x\qw(e^x-1)\qw$, 
$x>0$; see, \eg{}, 
\cite[Examples 11.1 and  11.10]{Steutels}. 
See also \cite[Example 7.2.3]{Bondis}.
\end{remark}

\section{Extensions to $\rea=\frac12$} \label{Smua}
\CCreset

In this section, we show the extensions to $\rea=\frac12$ claimed in
Remarks \ref{Rmua}, \ref{RTE1/2}, and \ref{RT1/2X}. These require 
different methods from the ones used above.

Let $\gf(t):=\E e^{\ii t\xi}$ be the \chf{} of the offspring distribution
$\xi$. Furthermore, let $\txi:=\xi-\E\xi=\xi-1$, and denote its \chf{} by
\begin{equation}\label{tgf}
  \tgf(t):=\E e^{\ii t\txi}=\emit\gf(t),
\qquad t\in\bbR.
\end{equation}
Since $\E\txi=0$ and $\E\txi^2=\gss$, we have
$\tgf(t)=1-\frac{\gss}2t^2+o(t^2)$; hence
\begin{equation}
  \label{pyr}
\tgf(t)=1-\frac{\gss}2t^2\bigsqpar{1+\gam(t)},
\qquad t\in\bbR,
\end{equation}
for some continuous 
function $\gam(t)$ on $\bbR$ such that $\gam(0)=0$.

We also let
\begin{equation}\label{rho}
  \rho(t):=1-\tgf(t)=\frac{\gss}2t^2\bigsqpar{1+\gam(t)}.
\end{equation}

Since $\xi$ is integer-valued, $\gf$ and $\tgf$ are $2\pi$-periodic.
 Note that 
$\P(\txi=-1)>0$ and thus
$\tgf(t)\neq1$ if $0<|t|\le\pi$ [also when
$\spann\xi>1$]; hence \eqref{pyr} and continuity imply
\begin{equation}\label{fitr}
  \Re\rho(t)=1-\Re\tgf(t)\ge\cc t^2, \ccdef\ccft
\qquad 0\le|t|\le\pi,
\end{equation}
for some $\ccx>0$.
Furthermore, if $\spann\xi=h\ge1$, then $\gf(\pm2\pi/h)=1$ but
$|\gf(t)|<1$ for $0<|t|<2\pi/h$, and it follows
similarly from \eqref{pyr} and continuity that
\begin{equation}\label{fir}
  |\tgf(t)|=|\gf(t)|\le 1-\cc t^2 \le e^{-\ccx t^2}, \ccdef\ccf
\qquad 0\le|t|\le\pi/h.
\end{equation}

\begin{lemma}
  \label{Lpyret}
If\/ $\rea<\frac12$, then
\begin{equation}\label{pyr1}
  \mu(\ga) = \frac{1}{2\pi\gG(1-\ga)}\intpi\intoo 
  x^{-\ga}\frac{\gf(t)}{e^x-\tgf(t)}\dd x\dd t,
\end{equation}
where the double integral is absolutely convergent.
\end{lemma}

\begin{proof}
Let $\rea<\frac12$.
  Fourier inversion and \eqref{mua} yield
  \begin{equation}\label{pyr2}
	\mu(\ga) 
= \sumn n^{\ga-1} \intT e^{-\ii(n-1)t}\gf(t)^n\dd t
= \sumn n^{\ga-1} \intT e^{\ii t}\tgf(t)^n\dd t.
  \end{equation}
Let $\spann\xi=h\ge1$.
It follows from the estimate \eqref{fir} that
\begin{equation}\label{pq}
  \intpi|\tgf(t)|^n\dd t 
={h} \int_{-\pi/h}^{\pi/h}|\tgf(t)|^n\dd t 
\le {h} \intoooo e^{-\ccf n t^2}\dd t 
= \CC n\qqw.\CCdef\CCpq
\end{equation}
Hence, 
\begin{equation}\label{pq1}
\sumn   \intpi \bigabs{n^{\ga-1}\eit \tgf(t)^n}\dd t 
\le\CCpq\sumn n^{\rea-\frac32}<\infty.
\end{equation}
Thus we may interchange the order of summation and integration in \eqref{pyr2}
and obtain
  \begin{equation}\label{pyr3}
	\mu(\ga) 
= \intT\eit\sumn n^{\ga-1} \tgf(t)^n\dd t.
  \end{equation}
The sum $\sumn n^{\ga-1} \tgf(t)^n$ is known as the polylogarithm
$\Li_{1-\ga}(\tgf(t))$ \cite[\S25.12(ii)]{NIST}.
It can be expressed as an
integral  \cite[25.12.11]{NIST} 
by a standard argument, which we adapt as follows:
Since $\rea<\frac12<1$,
we have
$ 
  n^{\ga-1}\gG(1-\ga)
= \intoo x^{-\ga} e^{-nx}\dd x
$ 
and thus \eqref{pyr3} yields
\begin{equation}
  \gG(1-\ga)\mu(\ga)
= \intT\sumn\intoo x^{-\ga}e^{-nx} \tgf(t)^n\eit \dd x\dd t.
\end{equation}
Again, this expression is absolutely convergent as a consequence of
\eqref{pq} and \eqref{pq1},
and thus we may again interchange the order of summation and integration
and obtain
\begin{equation}\label{pyr18}
  \begin{split}
  \gG(1-\ga)\mu(\ga)
= \intT\intoo x^{-\ga}\sumn e^{-nx} \tgf(t)^n\eit \dd x\dd t
\\
= \intT\intoo x^{-\ga}\frac{e^{-x}\tgf(t)}{1- e^{-x} \tgf(t)}\eit \dd x\dd t.	
  \end{split}
\end{equation}
This yields \eqref{pyr1}, with absolute convergence.
\end{proof}

We next modify \eqref{pyr1} by ignoring terms that are analytic at
$\rea=\frac12$; more precisely, we ignore terms that are analytic in
$\doi:=\set{\ga:0<\Re \ga<1}$.

\begin{lemma}
  \label{LP2}
There exists a function $h(\ga)\in\cH(\doi)$ such that 
if $0<\rea<\frac12$, then
\begin{equation}\label{pyrr}
  \mu(\ga)=
\frac{\gG(\ga)}{2\pi}\intpi \rho(t)^{-\ga}\dd t
+h(\ga).
\end{equation}
\end{lemma}

\begin{remark}\label{RP2} 
  Since $\rho(t)\neq0$ for $0<|t|\le\pi$, the integral 
$\int_{t_0\le |t|\le\pi}\rho(t)^{-\ga}\dd t$ is an entire function of $\ga$
for any
  $t_0\in(0,\pi]$, and thus the integral in \eqref{pyrr} can be replaced by
the integral over $|t|\le t_0$ for any such $t_0$.
\end{remark}

\begin{proof}[Proof of \refL{LP2}]
  First, for $x\ge1$ and $\rea>0$, the integrand in \eqref{pyr1} is
  $O(e^{-x})$ so the double integral over $\set{x\ge1,\,t\in(-\pi,\pi)}$
converges and defines an analytic function $h_1\in\cHoi$.
We may thus consider the integral for $0<x<1$ only.

Next, using \eqref{rho} and \eqref{fitr}, for $x>0$ we have 
  \begin{equation}\label{pl}
	\bigabs{e^x-\tgf(t)}
\ge \Re \bigpar{e^x-\tgf(t)}
=e^x-1+\Re\rho(t)
\ge x+\ccft t^2.
  \end{equation}
Hence, using $|\gf(t)-1|\le \CC t$ 
(since $\E\xi<\infty$),
\begin{equation}\label{pk}
\intpi\intoi
\Bigabs{ x^{-\ga}\frac{\gf(t)-1}{e^x-\tgf(t)}}\dd x\dd t
\le
\CC
\intpi\intoi
 x^{-\rea}\frac{|t|}{x+t^2}\dd x\dd t.
\end{equation}
Now, for $0<x<1$,
\begin{equation}\label{pkb}
  \int_0^\pi \frac{t}{x+t^2}\dd t
\le \int_0^{\sqrt x}\frac{t}x\dd t +\int_{\sqrt x}^\pi\frac{t}{t^2}\dd t
=\frac12+\log\pi-\log\sqrt x =O(1+|\log x|)
\end{equation}
and thus \eqref{pk} converges for $\rea<1$.
It follows that if we replace the numerator $\gf(t)$ by 1 in \eqref{pyr1}
(with $x<1$ only), then the difference is in $\cHoi$.

Similarly, for $0<x<1$ and $|t|\le\pi$,
\begin{equation}\label{fita}
  \Bigabs{\frac{1}{e^x-\tgf(t)}-\frac{1}{x+1-\tgf(t)}}
\le \frac{e^x-1-x}{(x+\ccft t^2)^2}
\le 1,
\end{equation}
and we may thus also replace the denominator $e^x-\tgf(t)$ by 
$x+1-\tgf(t)=x+\rho(t)$.

This yields
\begin{equation}
  \mu(\ga) = \frac{1}{2\pi\gG(1-\ga)}\intpi\intoi
  x^{-\ga}\frac{1}{x+\rho(t)}\dd x\dd t+h_2(\ga).
\end{equation}
with $h_2\in\cHoi$.
We now reintroduce $x\ge1$, noting that $\Re\rho(t)\ge0$ and thus,
for $\rea>0$,
\begin{equation}
  \intpi\int_1^\infty \Bigabs{\frac{x^{-\ga}}{x+\rho(t)}}\dd x\dd t
\le 2\pi \int_1^\infty x^{-\rea-1}\dd x<\infty.
\end{equation}
Hence, for $\ga\in\doi$,
\begin{equation}
  \mu(\ga) = \frac{1}{2\pi\gG(1-\ga)}\intpi\intoo
 \frac{x^{-\ga}}{x+\rho(t)}\dd x\dd t+h(\ga),
\end{equation}
with $h\in\cHoi$, and \eqref{pyrr} follows by a standard beta integral:\ for 
$0<\rea<1$ and $\rho\notin(-\infty,0]$ we have
\begin{equation}
  \begin{split}
  \intoo \frac{x^{-\ga}}{x+\rho}\dd x
&=\rho^{-\ga}  \intoo \frac{x^{-\ga}}{x+1}\dd x
=\rho^{-\ga}B(1-\ga,\ga)
\\&
=\rho^{-\ga}\gG(1-\ga)\gG(\ga),	
  \end{split}
\end{equation}
where the first equality holds for all $\rho>0$ by a change of variables and
therefore for all $\rho\notin(-\infty,0]$ by analytic continuation.
\end{proof}

Recall the function $\gam(\ga)$ defined by \eqref{pyr}.

\begin{lemma}
  \label{LG}
For any $a>0$ we have
\begin{equation}
  \intoo \frac{|\gam(at)-\gam(t)|}t\dd t<\infty.
\end{equation}
\end{lemma}

\begin{proof}
  By \eqref{pyr}, recalling $\E\txi=0$ and $\E\txi^2=\gss$, we have
  \begin{equation}\label{fitb}
	-\frac{\gss t^2}2\gam(t) = \tgf(t)-1+\frac{\gss t^2}2
=\E e^{\ii t\txi}-1-\E(\ii t\txi)-\frac12\E\xpar{\ii t\txi}^2.
  \end{equation}
Define
\begin{align}
\psi_1(x)&:=e^{\ii x}-1-\ii x, \label{psi1}
\\
\psi_2(x)&:=e^{\ii x}-1-\ii x-\tfrac12(\ii x)^2.\label{psi2}
\end{align}
Then \eqref{fitb} implies
\begin{equation}
  \gam(t)=-\frac{2}{\gss t^2}\E \psi_2(t\txi)
\end{equation}
and thus
\begin{align}
\gam(at)- \gam(t)
=\frac{2}{\gss t^2}\E \Bigsqpar{\psi_2(t\txi)-\frac{1}{a^2}\psi_2(at\txi)}.
\label{win1}
\end{align}

Fix $a>0$. Taylor's formula yields the standard estimate
$|\psi_2(x)|\le |x|^3$, and thus
\begin{equation}
|\psi_2(x)-a\qww\psi_2(ax)|\le C|x|^3.  
\end{equation}
Furthermore, 
$\psi_2(x)-a\qww\psi_2(ax)=\psi_1(x)-a\qww\psi_1(ax)$
by cancellation, and
$|\psi_1(x)|\le2|x|$ and thus
\begin{equation}
|\psi_1(x)-a\qww\psi_1(ax)|\le C|x|.  
\end{equation}
Consequently,
\begin{equation}\label{win2}
|\psi_2(x)-a\qww\psi_2(ax)|\le C\bigpar{|x|\land|x|^3}.  
\end{equation}

Combining \eqref{win1} and \eqref{win2} we obtain, 
for $t\neq0$,
\begin{equation}
|\gam(at)-\gam(t)|\le Ct^{-2} \E\bigpar{|t\txi|\land|t\txi|^3}.  
\end{equation}
Hence,
\begin{equation*}
  \begin{split}
\intoo\frac{|\gam(at)-\gam(t)|}t\dd t
&\le C \intoo \E\bigpar{|t\qww\txi|\land|\txi|^3}\dd t
\\&
= C \biggpar{\int_0^{|\txi|\qw}|\txi|^3\dd t
+ \int_{|\txi|\qw}^\infty t^{-2}|\txi|\dd t}
\\&
= C\E\bigpar{|\txi|^2+|\txi|^2} = 2 C \gss <\infty.	
  \end{split}
\qedhere
\end{equation*}
\end{proof}

\begin{remark}
  \refL{LG} and its proof hold
  with $\txi$ replaced by  
  any random variable $X$ with $\E X=0$ and
  $\E X^2<\infty$.
\end{remark}

\begin{remark}
  Note, in contrast, that the integral $\intoi |\gam(t)|t\qw \dd t$ may
  diverge; hence some cancellation is essential in \refL{LG}.
In fact, it is not difficult to show, using similar arguments, that 
$\intoi |\gam(t)|t\qw \dd t<\infty$ if and only if $\E \txi^2 {\log|\txi|}<\infty$.
(Since $\gam(t)\to-1$ as $t\to\infty$, we cannot here integrate to $\infty$.)
\end{remark}

The function $\mu(\ga)$ is defined by \eqref{mu} for $\rea<\frac12$. As
noted at~\eqref{aaa},
$\mu(\ga)\to\infty$ as $\ga\upto\frac12$. However, $\mu(\ga)$ has a
continuous extension to all other points on the line $\rea=\frac12$.

\begin{theorem}
  \label{TM}
The function $\mu(\ga)$ has a continuous extension to the set
$\set{\ga:\rea\le\frac12}\setminus\set{\frac12}$. 
\end{theorem}

\begin{proof}
  For $0<s\le\pi$ and $\rea<\frac12$, let
  \begin{equation}\label{f}
	f_s(\ga):=\Bigpar{\frac{\gss}2}^\ga\int_{-s}^s\rho(t)^{-\ga}\dd t
=\int_{-s}^s t^{-2\ga}\bigsqpar{1+\gam(t)}^{-\ga}\dd t.
  \end{equation}

Let $a>0$ and let $s_0:=\pi/(1\vee a)$.
Then, for $0<s\le s_0$, we have
\begin{equation}\label{fa}
  f_{a s}(\ga) 
= \int_{-as}^{as} t^{-2\ga}\bigsqpar{1+\gam(t)}^{-\ga}\dd t
= a^{1-2\ga}\int_{-s}^{s} t^{-2\ga}\bigsqpar{1+\gam(at)}^{-\ga}\dd t.
\end{equation}
Fix $B<\infty$ and let $\DB:=\set{\ga:0\le\rea<\frac12,\,|\Im\ga|\le B}$.
By \eqref{f} and \eqref{fa}, 
uniformly for $\ga\in \DB$, noting that $1+\gam(t)\neq0$ for $0 < |t|\le\pi$ by
\eqref{pyr}, we have 
\begin{equation}\label{winn}
  \begin{split}
\bigabs{a^{2\ga-1}f_{as}(\ga)-f_s(\ga)}
&\le
\int_{-s}^s\bigabs{(1+\gam(at))^{-\ga}-(1+\gam(t))^{-\ga}}	|t^{-2\ga}|\dd t
\\&
\le C\int_{-s}^s|\gam(at)-\gam(t)| t^{-2\rea}\dd t
\\&
\le C\int_{0}^s|\gam(at)-\gam(t)| t^{-1}\dd t,
  \end{split}
\end{equation}
which tends to 0 as $s\to0$ by \refL{LG}.

Let
\begin{equation}
  F_s(\ga):=a^{2\ga-1}\bigpar{f_\pi(\ga)-f_{as}(\ga)}
-\bigpar{f_\pi(\ga)-f_s(\ga)}.
\end{equation}
We have just shown in \eqref{winn} that as $s\to0$ we have
\begin{equation}\label{Fs}
  F_s(\ga)\to\bigpar{a^{2\ga-1}-1}f_\pi(\ga)
\end{equation}
uniformly in $\DB$.
For $s\in (0,s_0]$, $F_s(\ga)$ is an entire function, see \refR{RP2}, and
  in particular continuous on $\bDB$.
Hence, the sequence $F_{1/n}(\ga)$, which is uniformly convergent on $\DB$
by \eqref{Fs}, is a Cauchy sequence in $C(\bDB)$, and thus converges
uniformly on $\bDB$ to some continuous limit.
Together with \eqref{Fs} again, this shows that 
$\xpar{a^{2\ga-1}-1}f_\pi(\ga)$ has a continuous extension to $\bDB$.

This holds for any $a>0$. We now choose $a=e^{1/B}$; then $a^{2\ga-1}\neq1$
in $\bDB\xq$, and thus $f_{\pi}(\ga)$ has a continuous extension to $\bDB\xq$.
Since $B$ is arbitrary, this shows that
$f_\pi(\ga)$ has a continuous extension to
  $\set{\ga:0\le\rea\le\frac12}\xq$.

Finally, the definition \eqref{f} shows that the same holds for 
$\intpi \rho(t)^{-\ga}\dd t$, and the result follows by
\refL{LP2}.
\end{proof}

In the sequel, $\mu(\ga)$ is defined for $\rea=\frac12$, $\ga\neq\frac12$, as
this continuous extension.

\begin{theorem}
  \label{T1/2}
  \begin{thmenumerate}
  \item \label{T1/2E}
The estimate \eqref{te-} in \refT{TE}\ref{TE-} holds also for
$\ga=\frac12+\ii y$, $y\neq0$.
Moreover, \eqref{te-} holds uniformly on compact subsets of 
$\set{\ga:-\frac12<\rea\le\frac12}\xq$.
  \item \label{T1/2X}
The limit result \eqref{tx<} in \refT{TX}\ref{TX<} holds also for
$\ga=\frac12+\ii y$, $y\neq0$.
Moreover, \eqref{tx<} holds 
in the space $C(\dbmx)$ of continuous functions on
the set 
$\dbmx:=\set{\ga:0<\rea\le\frac12}\xq$.
  \end{thmenumerate}
\end{theorem}

The topology in $C(\dbmx)$ is defined by uniform convergence on compact
subsets of $\dbmx$. 

\begin{proof}
  Part \ref{T1/2X} follows by \refT{T1} and \ref{T1/2E}, so it suffices to
  prove \ref{T1/2E}.

In this proof, let
$\db:=\set{\ga:0<\rea<\frac{3}4}$,
$\dbm:=\set{\ga:0<\rea<\frac12}$,
and, for $B>0$,
$\dbmB:=\set{\ga\in\dbm:|\Im\ga|\le B}$,
$\dbmxB:=\set{\ga\in\dbmx:|\Im\ga|\le B}$.

By \eqref{mua} and \eqref{mun}, for $\rea<\frac12$ we have
\begin{equation}
  \mu(\ga)-\mu_n(\ga)=\sum_{k=n+1}^\infty k^{\ga-1}\P(S_k=k-1).
\end{equation}
Imitating the proof of \refL{Lpyret} we obtain, \cf{} \eqref{pyr18}, for
$\ga\in\dbm$,
\begin{equation}\label{pyrx0}
  \begin{split}
\gG&(1-\ga)\bigpar{\mu(\ga)-\mu_n(\ga)}
= \intT\intoo x^{-\ga}\sum_{k=n+1}^\infty e^{-kx} \tgf(t)^k\eit \dd x\dd t
\\&
= \intT\intoo x^{-\ga}\frac{e^{-(n+1)x}\tgf(t)^{n+1}}{1- e^{-x} \tgf(t)}\eit 
\dd x\dd t
\\&
= \intT\intoo x^{-\ga}\frac{e^{-nx}\tgf(t)^{n}}{e^{x}- \tgf(t)}\gf(t)
\dd x\dd t
  \end{split}
\raisetag\baselineskip
\end{equation}
and thus, by the change of variables $x\mapsto x/n$, $t\mapsto\tqn$, we have
\begin{align}\label{pyrx1}
f_n(\ga)&:=
n^{\frac12-\ga}
\gG(1-\ga)\bigpar{\mu(\ga)-\mu_n(\ga)}
\\&\phantom:
=\frac{1 }{2\pi}
\intpm{\pi\sqrt n}\intoo x^{-\ga}
 \frac{e^{-x}\tgf(t/\sqrt n)^{n}}{n[e^{x/n}- \tgf(\tqn)]} \gf(\tqn)\dd x\dd t
\label{pyrx2}
\end{align}
Denote the  integrand in \eqref{pyrx2} by $\gnaxt$, and let
this define $\gnaxt$ for any $\ga\in\db$.
Note that for any fixed $\ga\in\db$, $x>0$, and $t\in\bbR$,
by \eqref{pyr},
\begin{equation}\label{kdu}
\gnaxt\to 
 x^{-\ga} \frac{e^{-x-\gssx t^2}}{x+\gssx t^2}
=:\gaxt.
\end{equation}
Furthermore, \eqref{kdu} trivially holds uniformly for $\ga\in \db$.
Note also that, by \eqref{fitr},
\begin{equation}\label{fitt}
  \begin{split}
  \bigabs{n\bigpar{e^{x/n}-\tgf(\tqn)}}
&\ge\Re \bigpar{n\bigpar{e^{x/n}-\tgf(\tqn)}}
\ge x+n\Re\bigpar{1-\tgf(\tqn)}
\\&\ge x+\ccft t^2.	
  \end{split}
\raisetag\baselineskip
\end{equation}

Let $h:=\spann\xi$.
If $h>1$, 
consider first $t$ with $\pi\sqrt n/h <|t|\le\pi\sqrt n$.
For such $t$, \eqref{fitt} implies $\bigabs{e^{x/n}-\tgf(\tqn)}\ge c$,
and thus $|\gnaxt|\le C n\qw x^{-\rea}e^{-x}$.
Hence, the integral \eqref{pyrx2} restricted to $|t|>\pi\sqrt n/h$ is
$O\bigpar{n\qqw}$, uniformly in $\db$.

Next (for any $h$),
for $\ga\in\db$ and $|t|\le\pi\sqrt n/h$, \eqref{fir} and \eqref{fitt} yield
\begin{equation}
  |\gnaxt|\le x^{-\rea}\frac{e^{-x-ct^2}}{x+ct^2}
\le \bigpar{1+x^{-3/4}}\frac{e^{-x-ct^2}}{x+ct^2}.
\end{equation}
The \rhs{} is integrable over
$(x,t)\in(0,\infty)\times((-\infty,-1)\cup(1,\infty))$; hence 
the integral \eqref{pyrx2} restricted to 
$1<|t|\le\pi\sqrt n/h$ converges by \refL{Ldom} uniformly on $\db$ to the
corresponding integral of $\gaxt$, which is an analytic function
$h_1(\ga)\in\cH(\db)$ by \refR{Rdom}.

Similarly, for $x\ge1$, using \eqref{fitt} again,
\begin{equation}
  |\gnaxt|
\le x^{-\rea}\frac{e^{-x}}{x + c_1 t^2}
\le 
e^{-x}
\end{equation}
and it follows by \refL{Ldom} and \refR{Rdom} that the integral \eqref{pyrx2}
restricted to $(x,t)\in(1,\infty)\times(-1,1)$ converges uniformly to an
analytic function $h_2(\ga)\in\cH(\db)$.

It remains to consider the integral in \eqref{pyrx2} over
$(x,t) \in Q:=(0, 1) \times(-1,1)$.
We modify this integral in several steps.

We first replace $e^{-x}$ by 1 in the numerator of $\gnaxt$; the absolute value
of the difference is bounded, using \eqref{fitt} again, by
\begin{equation}
  x^{-\rea}\frac{1-e^{-x}}{x+ c_1 t^2}\le x^{-\rea}\le 1+x^{-3/4}
\end{equation}
and thus \refL{Ldom} and \refR{Rdom} show that the integral of the
difference over $(x,t)\in Q$ converges uniformly to an 
analytic function $h_3(\ga)\in\cH(\db)$.

Similarly, we then replace $\tgf(\tqn)^n$ by 1
in the resulting integral; 
the difference is by \eqref{fitt} and
\eqref{pyr}, using $|1-\tgf(\tqn)^n|\le n|1-\tgf(\tqn)|$,
bounded by
\begin{equation}
  x^{-\rea}\frac{Ct^2}{x+c_1 t^2}\le Cx^{-\rea}\le C(1+x^{-3/4})
\end{equation}
and again the integral of the
difference over $Q$ converges uniformly to an 
analytic function $h_4(\ga)\in\cH(\db)$.

Next, we replace in the denominator 
$e^{x/n}-\tgf(\tqn)$ by $(x/n)+\rho(\tqn)$.
The resulting error is by \eqref{fita} bounded by
$
x^{-\rea}\frac{1}n  
$ 
so the error in the integral 
over $Q$
is $O(n\qw)$, uniformly in $\ga\in\db$.

Similarly, $\gf(\tqn)=1+O(\tqn)$, so replacing the factor $\gf(\tqn)$ by 1
yields an error in the integral over $Q$ that is bounded, for $\ga\in\db$, by
\begin{equation}
  \frac{C}{\sqrt n}\intii\intoi x^{-3/4}\frac{|t|}{x+t^2}\dd x\dd t
=O\bigpar{n\qqw},
\end{equation}
since the integral converges 
by \eqref{pkb}.

Summarizing the development so far, we have shown that
\begin{equation}\label{pyrz}
  \begin{split}
f_n(\ga)&
=\frac{1 }{2\pi}
\intii\intoi x^{-\ga}
 \frac{1}{x+n\rho(\tqn)}\dd x\dd t +h_5(\ga)+o(1),
 \end{split}
\end{equation}
uniformly in $\dbm$, for some $h_5(\ga)\in\cH(\db)$.

Define, for $a>0$ and $\ga\in\dbm$,
\begin{equation}\label{pyrw}
  \begin{split}
F_{n,a}(\ga)
&:=
\intpm{1}\int_0^{1} 
 \frac{x^{-\ga}}{x+na\qww\rho(a\tqn)}\dd x\dd t 
\\&\phantom:
=\intpm{1}\int_0^{1} 
\frac{x^{-\ga}}{x+\gssx t^2\bigsqpar{1+\gam(a\tqn)}}\dd x\dd t ,
 \end{split}
\end{equation}
noting that the integrals converge by \eqref{fitr} and the fact that
\begin{align}
  \intpm{1}\int_0^{1} \frac{|x^{-\ga}|}{x+ t^2}\dd x\dd t
\le \pi \int_0^{1}\!x^{-\rea-\frac12}\dd x<\infty.
\end{align}
Thus, \eqref{pyrz} can be written,  uniformly in $\dbm$,
\begin{equation}\label{pyro}
f_n(\ga)
=\frac{1 }{2\pi}F_{n,1}(\ga) +h_5(\ga)+o(1).
\end{equation}

Fix $a>1$.
Then,
for  $\ga\in\dbm$ (and $n\ge a^2$, say),
using \refL{LG} we have
\begin{equation}
  \begin{split}
\bigabs{F_{n,a}(\ga)-F_{n,1}(\ga)}
&
\le
\intpm{1}\intoi|x^{-\ga}|\frac{Ct^2\bigabs{\gam(a\tqn)-\gam(\tqn)}} 
{(x+ct^2)^2}
\ddx x\dd t 
\\&\le
C
\intpm{1}\bigabs{\gam(a\tqn)-\gam(\tqn)}\int_0^{1} x^{-1/2}
 \frac{\ddx x}{x+t^2}\dd t 
\\&
\le
C
\intpm{1}\frac{\bigabs{\gam(a\tqn)-\gam(\tqn)}}{|t|}\dd t 
\\&
=
C
\intpm{1/\sqrt n}\frac{\bigabs{\gam(at)-\gam(t)}}{|t|}\dd t 
\to0,
 \end{split}
\raisetag{1.2\baselineskip}
\end{equation}
as \ntoo.
Moreover, by the change of variables 
$x\mapsto a\qww x$, $t\mapsto a\qw t$,
\begin{equation}
  \begin{split}
F_{n,a}(\ga)
=
a^{2\ga-1}
\intpm{a}\int_0^{a^2} 
 \frac{x^{-\ga}}{x+n\rho(\tqn)}\dd x\dd t ,
 \end{split}
\end{equation}
which differs from $a^{2\ga-1}F_{n,1}(\ga)$ by an integral which,
using \refL{Ldom} and \refR{Rdom} again, converges uniformly to some
function $h_6(\ga)\in\cH(\db)$.

It follows that, uniformly for $\ga\in\dbm$,
\begin{equation}
\bigpar{a^{2\ga-1}-1}F_{n,1}(\ga)
=F_{n,a}(\ga)-F_{n,1}(\ga) -\bigpar{F_{n,a}(\ga)-a^{2\ga-1}F_{n,1}(\ga)}
\to -h_6(\ga)
.
\end{equation}
Consequently, \eqref{pyro} shows that
$\bigpar{a^{2\ga-1}-1}f_{n}(\ga)$ converges uniformly in $\dbm$
to some function
$h_7(\ga)\in\cH(\db)$, which, recalling the definition \eqref{pyrx1} of
$f_n(\ga)$, 
shows that
\begin{equation}\label{fitw}
  \bigpar{a^{2\ga-1}-1}n^{\frac12-\ga}\bigpar{\mu(\ga)-\mu_n(\ga)}
=h_8(\ga)+o(1),
\end{equation}
uniformly in $\dbmB$, for some function $h_8(\ga)\in\cH(\db)$ 
and every $B>0$.
By \eqref{magnus}, 
\begin{equation}\label{fitz}
  h_8(\ga)=\bigpar{a^{2\ga-1}-1}\frac{1}{\sqrt{2\pi\gss}(\frac12-\ga)}
\end{equation}
for $\ga\in\dbm$, and thus 
by analytic continuation
for $\ga\in \db\xq$.

By \refT{TM}, $\mu(\ga)$ is continuous on $\dbmx$, and so are $\mu_n(\ga)$
(which is an entire function) and $h_8(\ga)$.
Hence, by continuity, \eqref{fitw} holds
uniformly in every $\dbmxB$.

Finally, for any compact set $K\subset\dbmx$, we can choose $a>1$ such
that $a^{2\ga-1}\neq1$ on $K$, and then \eqref{fitw} and \eqref{fitz} show
that, uniformly for $\ga\in K$,
\begin{equation}
  \label{magnusx}
n^{\frac12-\ga}\bigsqpar{\mu(\ga)-\mu_n(\ga)}=
\frac{1}{\sqrt{2\pi\gss}(\frac12-\ga)}+o(1)
\end{equation}
as \ntoo.
The result \eqref{te-}, uniformly on $K$, follows from \eqref{magnusx} and
\refL{LEn}.

This shows that \eqref{te-} holds uniformly on any compact subset of
$\dbmx$, and in particular on any compact subset of 
$\set{\ga:\frac{1}{4}\le\rea\le\frac12}\xq$. Since \refT{TE}\ref{TE-} implies that
\eqref{te-} holds uniformly on any compact subset of
$\set{\ga:-\frac{1}{2}<\rea\le\frac14}$, it follows that it holds uniformly
on any compact subset of
$\set{\ga:-\frac{1}{2}<\rea\le\frac12}\xq$.
\end{proof}

\section{An example where $\mu(\ga)$ has no analytic extension} \label{S:bad}

\refT{TM} shows that $\mu(\ga)$ has a continuous extension to the line
$\rea=\frac12$, except at $\ga=\frac12$.
However, in general, $\mu$ cannot be extended analytically across this line;
in fact the derivative $\mu'(\ga)$
may diverge as $\ga$ approaches this line.
In particular, \refT{TEgd}\ref{TEgdmu} does not hold (in general) without
the extra moment assumption there.

\begin{theorem}
  \label{Tbad}
There exists $\xi$ with $\E\xi=1$ and $0<\Var\xi<\infty$ such that for any
$\gao$ with $\Re\gao=\frac12$,
$
\limsup_{\ga\to\gao,\,\rea<\frac12}|\mu'(\ga)|=\infty$.
In particular, $\mu(\ga)$ has no analytic extension in a neighborhood of
any such $\gao$. In other words, the line $\rea=\frac12$ is a natural
boundary for $\mu(\ga)$.
\end{theorem}

We shall first prove three
lemmas.
Instead of working with $\mu(\ga)$ directly, we shall use \refL{LP2} (and,
for convenience, \refR{RP2}).
We define, for any function $\rho(t)$ and a complex $\ga$,
\begin{equation}\label{FF}
  F(\rho;\ga):=\intii\rho(t)^{-\ga}\dd t.
\end{equation}
Note that if $\rho(t)\ge ct^2$ (as will be the case below), then this
integral is 
finite for $\rea<\frac12$, at least, and defines an analytic function
there. 
If $F(\rho;\ga)$ extends analytically to a larger domain, we will use the
same notation for the extension (even if the integral \eqref{FF} diverges).

If $\rho(t)=1-\E e^{\ii t\txi}$ as in \eqref{rho}, we also write $F(\xi;\ga)$.

By \refL{LP2} and \refR{RP2}, \refT{Tbad} follows if we prove the statement
with $\mu(\ga)$ replaced by $F(\xi;\ga)$.

We define in this section the domains 
$\Dbad:=\set{\ga:\frac{1}4<\rea<\frac{3}4}$, 
$\Dbadm:=\set{\ga:\frac{1}4<\rea<\frac{1}2}$
and
$\Dbadx:=\Dbad\xq$. 
(These choices are partly for convenience; we could take $\Dbad$ larger.)

%
If $(g_N)$ is a sequence of functions in a domain~$D$,
we write $O_{\cH(D)}(g_N(\ga))$ for any sequence of functions 
$f_N\in\cH(D)$ such that
$f_N(\ga)/g_N(\ga)$ is bounded on each compact $K\subset D$, uniformly in $N$.
(Often, $g_N(\ga)$ will not depend on $\ga$.)
We extend the definition to functions $g_{N,t}(\ga)$ and $f_{N,t}(\ga)$
depending also on an additional parameter~$t$, 
requiring uniformity also in $t$.

It will be convenient to work with a restricted set of offspring
distributions $\xi$. 
Let $\ppp$ be the set of all probability distributions $(p_k)_0^\infty$
on \set{0,1,2,\dots} such that
$p_0,p_1,p_2>0.1$, and if $\xi$ has the distribution $(p_k)_0^\infty$, then
$\E\xi=\sum_k kp_k=1$,
$\Var\xi=\sum_k (k-1)^2p_k=2$ and
$\E\xi^3=\sum_k k^3p_k<\infty$.
(The set $\ppp$ is clearly non-empty. A concrete example is
$(0.52,0.2,0.2,0,0,0.08)$.)
We write $\xi\in\ppp$ for $\cL(\xi)\in\ppp$.

If $\xi\in\ppp$, then $\gss=2$ and $\E\xi^3<\infty$, and thus
$\tgf(t)=1-t^2+O(t^3)$; hence $\rho(t)=t^2+O(t^3)$.
Moreover, since $\P(\txi=j)>0.1$ for $j = \pm 1$, we have
\begin{equation}\label{rer}
  \Re\rho(t)=\Re\bigpar{1-\E e^{\ii t\txi}}
=\E\bigpar{1-\cos t\txi}
\ge 0.2(1-\cos t) \ge ct^2,
\end{equation}
for $|t|\le\pi$,
uniformly for all $\xi\in\ppp$.

\begin{lemma}
  \label{Ldx}
 If $\xi\in\ppp$, then 
$F(\xi;\ga)$ extends to a function in $\cH(\Dbadx)$.
\end{lemma}

\begin{proof}
  $\mu(\ga)\in\cH(\Dbadx)$ by
\refT{TEgd}\ref{TEgdmu} (or \refT{TH}), and the result follows
 by \refL{LP2} and \refR{RP2}.
\end{proof}

\begin{lemma}
  \label{LY}
If $\xi_N\in\ppp$ for $N \ge 1$ and $\xi_N\dto\xi$, then $F(\xi_N;\ga)\to
F(\xi;\ga)$ in $\cH(\Dbadm)$.
\end{lemma}

Note that we do not assume $\xi\in\ppp$. 
(In fact, it is easy to see that the lemma extends to arbitrary $\xi_N$ and
$\xi$ with expectation 1 and finite, non-zero variance.)

\begin{proof}
  Let $\rho(t):=1-\E e^{\ii t \txi}$ and $\rho_N(t):=1-\E e^{\ii t \txi_N}$,
where as usual $\txi:=\xi-1$ and $\txi_N:=\xi_N-1$.
Since $\xi_N\dto\xi$, $\rho_N(t)\to\rho(t)$ for every $t$.
\refL{Ldom} together with the estimate \eqref{rer} show that
$F(\xi_N;\ga)=F(\rho_N;\ga)\to F(\xi;\ga)$ uniformly on every compact subset
of $\Dbadm$.
\end{proof}

\begin{lemma}\label{LD}
  If $\xi\in\ppp$ and $y\in\bbR\xo$, 
then there exists a sequence $\xi_N\in\ppp$, $N\ge1$, 
such that, as \Ntoo,
$\xi_N\dto\xi$
and
$\bigabs{\frac{\ddx}{\ddx \ga}F(\xi_N;\ga)}_{\ga=\frac12+\ii y}\to\infty$
for any fixed real $y\neq0$.
\end{lemma}

\begin{proof}
  Let $a_N:=(\log N)\qqw$ and let $\xi_N$ have the distribution
  \begin{equation}
	\cL(\xi_N)=\cL(\xi)
+a_N\Bigsqpar{\frac{2}{N^2}\bigpar{\gd_N-N\gd_1+(N-1)\gd_0}
-\frac{N-1}N\bigpar{\gd_2-2\gd_1+\gd_0}}
  \end{equation}
where $\gd_j$ is unit mass at $j$. Since $a_N\to0$, and 
$\P(\xi=j)>0.1>0$ for $j=0,1,2$,
this is clearly a probability distribution if $N$ is large enough.
Furthermore,  $\xi_N\dto\xi$ as \Ntoo, and 
$\E\xi_N=\E\xi$, $\E\xi_N^2=\E\xi^2$ and $\E\xi_N^3<\infty$, 
and thus $\xi_N\in\ppp$, provided $N$ is large enough.
(We assume in the rest of this proof that $N$ is large enough whenever
necessary, without further mention. 
We can define $\xi_N$ arbitrarily for small $N$.)

Let $\gf_N(t):=\E e^{\ii t\xi_N}$
and, recalling \eqref{psi1}--\eqref{psi2},
\begin{equation}\label{gD1}
  \begin{split}
\gD_N(t)
&
:=\gf_N(t)-\gf(t)
=\E e^{\ii t\xi_N} -\E e^{\ii t\xi}
\\&
=a_N\Bigsqpar{\frac{2}{N^2}\bigpar{e^{\ii Nt}-N\eit+N-1}
-\frac{N-1}N\bigpar{e^{2\ii t}-2\eit+1}}
\\&
=a_N\Bigsqpar{\frac{2}{N^2}\bigpar{\psi_1(Nt)-N\psi_1(t)}
-\frac{N-1}N\bigpar{(\ii t)^2+O(t^3)}}
\\&
=2a_NN\qww\bigpar{\psi_2(Nt)-N\psi_2(t)}
+O(a_Nt^3)
\\&
=2a_NN\qww\psi_2(Nt)+O(a_Nt^3),
  \end{split}
\end{equation}
since $\psi_2(x)=O(x^3)$.
We further define
\begin{equation}\label{tpsi}
  \tpsi(t):=2\frac{\psi_2(t)}{t^2}
=\frac{2e^{\ii t}-2-2\ii t+t^2}{t^2}
=2\frac{\psi_1(t)}{t^2}+1.
\end{equation}
Then $\tpsi$ is bounded and continuous on $\bbR$, $\tpsi(t)=O(t)$ and
$\tpsi(t)=1+O(t\qw)$. Furthermore, \eqref{gD1} yields
\begin{equation}\label{gD2}
  \gD_N(t)=a_Nt^2\tpsi(Nt)+O(a_Nt^3).
\end{equation}
In particular, $\gD_N(t)=O(a_Nt^2)$ for $|t|\le\pi$.

We further let $\tgf_N(t):=\E e^{\ii t\txi_N}$,
$\rho_N(t):=1-\tgf_N(t)$
and, using \eqref{gD2},
\begin{equation}\label{tgD}
  \begin{split}
\tgD_N(t):=\rho_N(t)-\rho(t)=-\emit\gD_N(t)
=
-a_Nt^2\tpsi(Nt)+O(a_Nt^3).
  \end{split}
\end{equation}
In particular, 
\begin{equation}
  \label{tgD2}
\tgD_N(t)=O(a_Nt^2),
\qquad |t|\le\pi.
\end{equation}

Let $\rho_0(t):=t^2$, and let $\gd(t):=\rho(t)-\rho_0(t)$.
Then $\gd(t)=O(t^3)$, since $\Var\xi=2$ and $\E\xi^3<\infty$.
The general formula, for any twice continuously differentiable function $f$,
\begin{equation*}
f(x+y+z)-f(x+y)-f(x+z)+f(x)
=
yz\intoi\intoi f''(x+sy+tz)\dd s\dd t  
\end{equation*}
implies together with \eqref{rer} and \eqref{tgD2},
for $\ga\in \Dbad$,
\begin{equation}
  \begin{split}
  \bigabs{(\rho(t)+\tgD_N(t))^{-\ga}-\rho(t)^{-\ga}
&-\bigpar{(\rho_0(t)+\tgD_N(t))^{-\ga}-\rho_0(t)^{-\ga}}
}
\\&
\le
C|\tgD_N(t)|\, |\gd(t)|\, |\ga|\, |\ga+1|\, |t|^{-2(\rea+2)}	
\\&
=O\bigpar{a_N |\ga|^2|t|^{1-2\rea}}.
  \end{split}
\end{equation}
Hence, integrating over $t$ and recalling \eqref{FF},
\begin{equation}\label{DF1}
  F(\rho+\tgD_N;\ga)-  F(\rho;\ga)
-\bigpar{  F(\rho_0+\tgD_N;\ga)-  F(\rho_0;\ga)}
=\OHD(a_N).
\end{equation}

Next, let $\xgdn(t):=-a_Nt^2\tpsi(Nt)$. Then $\tgdn(t)-\xgdn(t)=O(a_Nt^3)$ by
\eqref{tgD}, and thus, by the mean value theorem and \eqref{tgD2},
for $|t|\le\pi$,
\begin{equation*}
  \begin{split}
  \bigabs{
(\rho_0(t)+\tgD_N(t))^{-\ga}-(\rho_0(t)+\xgdn(t))^{-\ga}
}
&\le
C |\tgD_N(t)-\xgdn(t)|\, |\ga|\, |t|^{-2(\rea+1)}	
\\&
=O\bigpar{a_N |\ga||t|^{1-2\rea}}.
  \end{split}
\end{equation*}
Hence, by an integration,
\begin{equation}\label{DF2}
F(\rho_0+\tgD_N;\ga)-  F(\rho_0+\xgdn;\ga)
=\OHD(a_N).
\end{equation}

Now consider
$F(\rho_0+\xgdn;\ga)-  F(\rho_0;\ga)$.
Let
$\chi(t):=\ett{|t|>1}$.
Then, considering first $t>0$,
for $\ga\in \Dbadm$,
\begin{equation}\label{eleonora}
  \begin{split}
\intoi
&\bigsqpar{(\rho_0(t)+\xgdn(t))^{-\ga}-\rho_0(t)^{-\ga}}\dd t
=
\intoi t^{-2\ga}\bigsqpar{\xpar{1-a_N\tpsi(Nt)}^{-\ga}-1}\dd t
\\&
=
\intoi t^{-2\ga}\bigsqpar{(1-a_N\tpsi(Nt))^{-\ga}-(1-a_N\chi(Nt))^{-\ga}}\dd
t
\\&\qquad\qquad{}
+\bigsqpar{(1-a_N)^{-\ga}-1}\int_{1/N}^1 t^{-2\ga}\dd t
\\&
=
N^{2\ga-1}
\int_0^N t^{-2\ga}\bigsqpar{(1-a_N\tpsi(t))^{-\ga}-(1-a_N\chi(t))^{-\ga}}\dd t
\\&\qquad\qquad{}
+\bigsqpar{(1-a_N)^{-\ga}-1}\frac1{1-2\ga}
\bigpar{1-N^{2\ga-1}}.
  \end{split}
\raisetag{\baselineskip}
\end{equation}
Since $\tpsi(t)-\chi(t)=O\bigpar{|t|\land|t\qw|}$, 
and $a_N\chi(t)=O(a_N)=o(1)$, with $\chi(t)=0$ for $0 \leq t<1$,
a Taylor expansion yields,
uniformly for $t\in\bbR$,
\begin{align}\label{D13}
&(1-a_N\tpsi(t))^{-\ga}-(1-a_N\chi(t))^{-\ga}
\notag\\
&\hskip2em
=\ga a_N\bigpar{\tpsi(t)-\chi(t)}
  +\OHD\bigpar{a_N|\tpsi(t)-\chi(t)|(a_N|\tpsi(t)|+a_N\chi(t))}  
\notag\\
&\hskip2em
=\ga a_N\bigpar{\tpsi(t)-\chi(t)}+\OHD\bigpar{a_N^2(|t|^2\land|t|\qw)}.  
\end{align}
Using \eqref{D13}
and a Taylor expansion of $(1-a_N)^{-\ga}$ in \eqref{eleonora},
we obtain
for $\ga\in \Dbadm$,
	\begin{multline}\label{erika}
\intoi
\bigsqpar{(\rho_0(t)+\xgdn(t))^{-\ga}-\rho_0(t)^{-\ga}}\dd t
\\
=
N^{2\ga-1}
\int_0^N t^{-2\ga}\ga a_N\bigpar{\tpsi(t)-\chi(t)}\dd t
-\frac{\ga a_N}{1-2\ga} N^{2\ga-1}
\\
+\OHDx\bigpar{a_N^2 N^{2\ga-1}}
+\OHDx(a_N).	  
	\end{multline}
Furthermore,
using again $\tpsi(t)-\chi(t)=O\bigpar{|t\qw|}$, 
\begin{equation}\label{D14}
  \int_N^\infty\!t^{-2\ga}\bigpar{\tpsi(t)-\chi(t)}\dd t
=O\bigpar{N^{-2\rea}},
\end{equation}
so we may as well integrate to $\infty$ on the \rhs{} of \eqref{erika}.

For $\ga\in \Dbadm$, recalling \eqref{tpsi},
\begin{equation}\label{matt}
  \begin{split}
\intoo\!t^{-2\ga}\bigpar{\tpsi(t)-\chi(t)}\dd t
&=	
2\intoo \psi_1(t)t^{-2\ga-2}\dd t
+\intoi t^{-2\ga}\dd t
  \end{split}
\end{equation}
Furthermore, if $\ga\in \Dbadm$ and $\Re\zeta\ge0$,
then
\begin{equation}\label{per}
\intoo\!\bigpar{e^{-\zeta t}-1+\zeta t}t^{-2\ga-2} \dd t =
\zeta^{2\ga+1}\gG(-2\ga-1);
\end{equation}
the case $\zeta=1$ is well known \cite[5.9.5]{NIST},
the case $\zeta>0$ follows by a change of variables,
the case $\Re\zeta>0$ follows by analytic continuation,
and the case $\Re\zeta\ge0$ follows by continuity.
Recalling \eqref{psi1}, we take
$\zeta=-\ii$ in \eqref{per}, and obtain from \eqref{matt},
for $\ga\in \Dbadm$.
\begin{equation}\label{matta}
  \begin{split}
\intoo\!t^{-2\ga}\bigpar{\tpsi(t)-\chi(t)}\dd t
&=	
2(-\ii)^{2\ga+1}\gG(-2\ga-1)
+\frac{1}{1-2\ga}.
  \end{split}
\end{equation}

Combining \eqref{erika}--\eqref{D14} and \eqref{matta},
we obtain (for $\alpha \in \Dbadm$)
	\begin{multline}\label{ahm}
\intoi
\bigsqpar{(\rho_0(t)+\xgdn(t))^{-\ga}-\rho_0(t)^{-\ga}}\dd t
=
2\ga (-\ii)^{2\ga+1}\gG(-2\ga-1) a_N N^{2\ga-1}
\\
+\OHDx\bigpar{a_N^2 N^{2\ga-1}}
+\OHDx(a_N).	  
	\end{multline}
The integral over $(-1,0)$ yields the same result with 
$(-\ii)^{2\ga+1}$ replaced by $\ii^{2\ga+1}$, \eg{} by conjugating \eqref{ahm}
  and $\ga$.
Consequently,
	\begin{multline}\label{DF3}
F(\rho_0+\xgdn;\ga)-F(\rho_0;\ga)
=
2\ga \bigpar{\ii^{2\ga+1}+(-\ii)^{2\ga+1}}\gG(-2\ga-1) a_N N^{2\ga-1}
\\
+\OHDx\bigpar{a_N^2 N^{2\ga-1}}
+\OHDx(a_N).	  
	\end{multline}
For convenience, we write
\begin{equation}
  G(\ga):=2\ga\bigpar{\ii^{2\ga+1}+(-\ii)^{2\ga+1}}\gG(-2\ga-1)
=2\ga\bigpar{\ii e^{\ii\pi\ga}-\ii e^{-\ii\pi\ga}}\gG(-2\ga-1).
\end{equation}

Combining \eqref{DF1}, \eqref{DF2}, and \eqref{DF3} yield,
for $\ga\in \Dbadm$,
\begin{equation}\label{sw}
  \begin{split}
F(\rho_N;\ga)	
=
F(\rho+\tgD_N;\ga)	
&
=F(\rho;\ga)
+ a_N G(\ga) N^{2\ga-1}
\\&\qquad \qquad
+\OHDx\bigpar{a_N^2 N^{2\ga-1}}
+\OHDx(a_N)	  .
  \end{split}
\end{equation}
By \refL{Ldx}, 
all terms in \eqref{sw} are analytic in $\Dbadx$, and thus \eqref{sw} holds for
$\ga\in \Dbadx$.

Note that if $f_N$ and $g_N$ are functions such that
$f_N(\ga)=\OHDx(g_N(\ga))$,
then $f_N(\ga)=g_N(\ga)h_N(\ga)$ with $h_N(\ga)=\OHDx(1)$. By Cauchy's
estimate, $h_N'(\ga)=\OHDx(1)$, and it follows that
$f'_N(\ga)=\OHDx(g_N(\ga))+\OHDx(g_N'(\ga))$.
Hence, taking derivatives in \eqref{sw} and then putting $\ga=\frac12+\ii y$
for a fixed $y\neq0$ yields
\begin{align}\label{sjw}
F'(\rho_N;\ga)	
&
=F'(\rho;\ga)
+2 (\log N) a_N G(\ga) N^{2\ga-1}
+O\bigpar{a_N^2 \log N}
+O(a_N)
\notag\\&
=
2 G(\ga) (\log N) a_N  N^{2\ii y}+O(1)
.
\end{align}
Since $G(\ga)=-4\ga (\cosh \pi y) \gG(-2-2y\ii)\neq0$, $|N^{2\ii y}|=1$ and
$a_N \log N=(\log N)\qq\to\infty$, 
\eqref{sjw} shows that 
$|F'(\xi_N;\frac12+\ii y)|=|F'(\rho_N;\frac12+\ii y)|\to\infty$ 
as \Ntoo.
\end{proof}

\begin{proof}[Proof of \refT{Tbad}]
  Let $(y_n)_1^\infty$ be an enumeration of all non-zero rational numbers.
We shall construct sequences $x_n\in(\frac{1}4,\frac12)$ and $\xi_n\in\ppp$,
$n=1,2,\dots$,
such that, with $z_n:=x_n+\ii y_n\in \Dbadm$, 
\begin{equation}\label{ind}
  |F'(\xi_n,z_k)|>k, \qquad k=1,\dots,n,
\end{equation}
and, 
furthermore, the total variation distance 
\begin{equation}
  \label{dtv}
\dtv(\xi_n,\xi_{n-1})<2^{-n}.
\end{equation}

We construct the sequences inductively. Suppose that $\xi_{n-1}$ is
constructed. (For $n=1$, we let $\xi_0$ be any element of $\ppp$.)
By \refL{LD}, there exists a sequence $\xi_{n-1,N}\in\ppp$ such that,
as \Ntoo,
$\xinn\dto\xi_{n-1}$ and $|F'(\xinn;\zzn)|\to\infty$.
By \refL{LY}, then $F(\xinn;\ga)\to F(\xi_{n-1};\ga)$ in $\cH(\Dbadm)$.
This implies
$F'(\xinn;\ga)\to F'(\xi_{n-1};\ga)$ in $\cH(\Dbadm)$, and in particular,
$F'(\xinn;z_k)\to F'(\xi_{n-1};z_k)$ for $1\le k\le n-1$.
Since \eqref{ind} holds for $n-1$ by the induction hypothesis,
it follows that
$|F'(\xinn;z_k)|>k$ for $1\le k\le n-1$ for all large $N$.
Furthermore, if we choose $N$ large enough, $|F'(\xinn;\zzn)|>n$ and
$\dtv(\xinn,\xi_{n-1})<2^{-n}$. 

We choose a large $N$ such that these properties hold and let
$\xi_n:=\xinn$.
Then \eqref{ind} holds for $k=1,\dots,n-1$. Furthermore, since
$\xi_n\in\ppp$, $F(\xi_n;\ga)\in\cH(\Dbadx)$, 
and thus $F'(\xi_n;\ga)$ is
continuous in $\Dbadx$. Hence $|F'(\xi_n;x+\ii y_n)|\to |F'(\xi_n;\zzn)|$ as
$x\to\frac12$, and we can choose $x_n\in(\frac{1}4,\frac12)$ with
  $\frac12-x_n<\frac{1}n$ such that 
$|F'(\xi_n;x+\ii y_n)|>n$.

This completes the construction of $x_n$ and $\xi_n$.
By \eqref{dtv}, the distributions $\cL(\xi_n)$ form a Cauchy sequence in
total variation distance, so there exists a random variable $\xi$ with
$\xi_n\dto\xi$. Clearly, $\xi$ is non-negative and integer-valued.
Moreover, since $\xi_n\in\ppp$ we have $\E\xi_n^2=\Var\xi_n+(\E\xi_n)^2=3$,
for every $n$, and thus the sequence $\xi_n$ is uniformly integrable, so
$\E\xi=\lim_\ntoo\E\xi_n=1$. Furthermore, by Fatou's lemma,
$\E\xi^2 \leq 3 <\infty$. 
Note that $\xi$ does not necessarily belong to $\ppp$; in
fact, it is easily seen from \eqref{Finf} below that $\xi\notin\ppp$.
Nevertheless \eqref{rer} holds for every $\xi_n$ (with the same $c$) and thus 
\eqref{rer} holds for $\xi$ too. In particular $\P(\xi\neq1)>0$ so $\Var\xi>0$.

\refL{LY} shows that $F(\xi_n;\ga)\to F(\xi;\ga)$ in $\cH(\Dbadm)$, and thus
$F'(\xi_n;\ga)\to F'(\xi;\ga)$ for every $\ga\in \Dbadm$.
Hence, \eqref{ind} implies
\begin{equation}\label{Finf}
  |F'(\xi;z_k)|\ge k
\end{equation}
for every $k$.
Thus, $|F'(\xi;z_n)|\to\infty$ as \ntoo.

Now take any $y\in\bbR$ and let $\ga_0:=\frac12+\ii y$. There is an infinite
number of points $y_n$ in each neighborhood of $y$, so we can find a
subsequence 
converging to $y$. Since $x_n\to\frac12$, it
follows that there is a subsequence of $z_n=x_n+\ii y_n$ 
that converges to $\ga_0$. 
Suppose first that $y\neq0$, so $\ga_0\neq\frac12$.
Then it follows from \refL{LP2} (with \refR{RP2}) and
\refT{TM} that, as \ntoo{} along the subsequence,
\begin{equation}
  \mu'(z_n)=\frac{1}{2\pi}\gG(z_n)F'(\xi;z_n)+O(1)
\end{equation}
and thus, by \eqref{Finf},  $|\mu'(z_n)|\to\infty$. 

This proves the claim in
\refT{Tbad} for every $\ga_0$ with $\Re\ga_0=\frac12$ and
$\ga_0\neq\frac12$.
The case $\ga_0=\frac12$ follows easily, 
either by noting that the set of $\ga_0$ for which the
claim holds is closed, or simply by \eqref{aaa}. 
\end{proof}

\section{Moments}\label{Smom} 
In this section we prove \refTs{T1mom} and \ref{TXmom} on
moments of $X_n(\ga)$ and of the limits $Y(\ga)$.
The section is largely based on \citet{FillK03} and \cite{FillK04}, and uses 
the methods of \cite{FillFK}, also presented in \cite[Section VI.10]{FS}.

We assume for simplicity throughout this section that $\xi$ has span 1.
The general case follows by minor modifications of standard type.

\subsection{More notation and preliminaries}
\label{S:more_notation}
Recall that $\cT$ is the random \GWt{} defined by the offspring distribution
$\xi$. 
Let $p_k:=\P(\xi=k)$ denote the values of the probability mass function
for~$\xi$,
and let $\Phi$ be its \pgf:
\begin{align}\label{Phi}
  \Phi(z):=\E z^\xi=\sumko p_k z^k.
\end{align}
Similarly, let $q_n:=\P(|\cT|=n)$, and let~$y$ denote the corresponding
\pgf:
\begin{align}\label{yz}
  y(z):=\E z^{|\cT|} = \sumn \P\bigpar{|\cT|=n}z^n
=\sumn q_n z^n.
\end{align}
If $\cT$ has root degree $k$, denote the subtrees rooted at the children of
the root by $\cT_1,\dots,\cT_k$; note that, conditioned on $k$, these are
independent copies of $\cT$.
By conditioning on the root degree, we thus obtain the standard formula
\begin{align}\label{ma}
  y(z) &
= \sumko p_k \E [z^{1+|\cT_1|+\dotsm+|\cT_k|}]
=\sumko p_k z \bigpar{\E [z^{|\cT|}]}^k
=z\sumko p_k  {y(z)}^k
\notag\\&\phantom:
=z\Phi\yz
.\end{align}

A \GDD{} is a complex domain of the type
\begin{align}\label{GDD}
\set{z:|z|<R,\, z\neq1,\,|\arg(z-1)|>\gth}  
\end{align}
where $R>1$ and $0<\gth<\pi/2$, see \cite[Section VI.3]{FS}.
A function is \emph{\gda} if it is analytic in some \GDD{}
(or can be analytically continued to such a domain). 
Under our standing assumptions $\E\xi=1$ and $0<\Var\xi<\infty$,
the generating function $y(z)$ is \gda; moreover, 
as $z\to1$ 
in some  \GDD,
\begin{align}\label{yz1}
  y(z)=1-\sqrt2 \gs\qw(1-z)\qq+o\bigpar{|1-z|\qq}
,\end{align}
see \cite[Lemma A.2]{SJ167}.
This is perhaps more well-known if $\xi$ has some exponential moment, and then 
\eqref{yz1} may be improved to a full asymptotic expansion, and in particular
\begin{align}\label{yz2}
  y(z)=1-\sqrt2 \gs\qw(1-z)\qq+\Oz{},
\end{align}
see \eg{} 
\cite[Theorem VI.6]{FS}.
In fact, \eqref{yz2} holds provided only $\E\xi^3<\infty$.  This
follows easily from \eqref{ma}, see \refL{Lgdy}.

In the present section, asymptotic estimates similar to \eqref{yz1} and
\eqref{yz2} should always be interpreted as holding when 
$z\to1$ in a suitable \GDD, even when not said so explicitly; the domain may be different each time.

\begin{remark}
In most parts of the present section, we will only use the assumption
$\E\xi^2<\infty$ and the general \eqref{yz1}.
If we assume the $\E\xi^3<\infty$, and thus \eqref{yz2}
holds, then the error estimates below can be improved, 
and explicit error estimates can be obtained in \refT{TXmom}; see \cite{FillK04} 
where this is done in detail for a special $\xi$ using similar arguments.
In fact, it can be checked that if $\E\xi^3<\infty$, then
all $o$ terms in the proof below can be  shown to be of (at most)  
the same order as the bounds given in \cite{FillK04} for
the corresponding terms. 
Further, when $\xi$ has an exponential moment, 
a full asymptotic expansion of the mean is derived in 
\cite[Section 5.2]{FillFK}; 
it seems possible that this can be extended to higher moments, but we have
not pursued this.
\end{remark}

In some formulas below, certain unspecified polynomials appear as 
``error terms''.
(These are best regarded as polynomials in $1-z$.)
Let $\cP$ be the set of all polynomials, and, for any real $a$, let
\begin{align}\label{Pa}
\cP_a:=\set{P(z)\in\cP:\deg(P(z))<a}.   
\end{align}
Note that if $a\le0$, then $\cP_a=\set0$, and thus terms in $\cP_a$ vanish
and can be ignored.
In the formulas below, a restriction of the type $P(z)\in\cP_{a}$, i.e.,
$\deg(P(z))<a$, will always be a triviality, 
since higher powers of $1-z$ can be absorbed in an $O$ or $o$ term.

Recall that
the polylogarithm function is defined, for $\ga\in\bbC$, by
\begin{align}\label{Li}
  \Li_\ga(z):=\sumn n^{-\ga}z^n,
\qquad |z|<1;
\end{align}
see \cite[Section VI.8]{FS}, 
\cite[\S25.12]{NIST}, 
or \refApp{Apoly}.
It is well known that $\Li_{\ga}(z)$ is \gda; in fact, it can be 
analytically continued
to $\bbC\setminus[1,\infty)$. Moreover, 
if $\ga\notin\set{1,2,\dots}$, then,
as $z\to1$,
\begin{align}\label{li}
\Li_{\ga}(z) = \gG(1-\ga)(1-z)^{\ga-1} +P(z)+ O\bigpar{|1-z|^{\rga}},
\qquad P(z)\in\cP_{\rga},
\end{align}
see
\cite[Theorem VI.7]{FS} 
or \cite{Flajolet1999}, 
where a complete asymptotic expansion is given;
see also \refApp{Apoly}.
In particular, if $\rga\le0$, then $P(z)$ vanishes and so \eqref{li}
simplifies. 

Recall also that the \emph{Hadamard product} $A(z)\odot B(z)$
of two power series
$A(z)=\sumno a_n z^n$ and $B(z)=\sumno b_n z^n$
is defined by 
\begin{align}\label{hadamard}
  A(z)\odot B(z) := \sumno a_n b_n z^n.
\end{align}
As a simple example, for any complex $\ga$ and $\gb$,
\begin{align}\label{lili}
  \Li_\ga(z)\odot\Li_\gb(z)=\Li_{\ga+\gb}(z).
\end{align}

We will use some results on Hadamard products, 
essentially taken from \cite{FillFK}.
In the next lemma,
Part \ref{LFFKO} 
is \cite[Propositions 9 and 10(i)]{FillFK}, and
\ref{LFFKo} 
follows by the same arguments;
the 
proof of $\gD$-analyticity of the Hadamard product
given for 
\cite[Proposition 9]{FillFK} holds for any \gda{} functions.
(For  the case $a+b+1\in\NNo$, see \cite{FillFK} and \cite{FS}.)

\begin{lemma}[\cite{FillFK}]\label{LFFK}
  If $g(z)$ and $h(z)$ are \gda{}, then $g(z)\odot h(z)$ is \gda.
Moreover, suppose that $a$ and $b$ are real with $a+b+1\notin\NNo$; then the
following holds, as $z\to1$ in a suitable \GDD.
\begin{romenumerate}
  
\item \label{LFFKO}
If $g(z)=O(|1-z|^a)$ and $h(z)=O(|1-z|^b)$, then
\begin{align}
  \label{lffkO}
g(z)\odot h(z)=P(z)+\Oz{a+b+1},
\qquad P(z)\in\cP_{a+b+1}.
\end{align}

\item \label{LFFKo}
If $g(z)=O(|1-z|^a)$ and $h(z)=o(|1-z|^b)$, then
\begin{align}
  \label{lffko}
g(z)\odot h(z)=P(z)+\oz{a+b+1},
\qquad P(z)\in\cP_{a+b+1}.
\end{align}
\end{romenumerate}
\end{lemma}

The next lemma is a simplified version of \cite[Proposition 8]{FillFK};
that proposition
gives 
 (when $\ga,\gb,\ga+\gb\notin\bbZ$)
a complete asymptotic expansion,
and in particular a more explicit error term for our~\eqref{lih2}.

\begin{lemma}[\cite{FillFK}]\label{LIH2}
  Suppose that $\rga+\Re\gb+1\notin\NNo$. Then, as $z\to1$ in a suitable \GDD,
\begin{multline}\label{lih2}
(1-z)^\ga\odot(1-z)^\gb 
\\
= \frac{\gG(-\ga-\gb-1)}{\gG(-\ga)\gG(-\gb)}  (1-z)^{\ga+\gb+1}
+P(z)
+ o\bigpar{|1-z|^{\Re\ga+\Re\gb+1}},
\\
P(z)\in\cP_{\rga+\rgb+1}
.\end{multline}
\end{lemma}
\begin{proof}
  The case when none of $\ga,\gb,\ga+\gb$ is an integer is part of
\cite[Proposition 8]{FillFK}.

In general, we use arguments from \cite{FillFK}.
If neither $\ga$ nor $\gb$ is a non-negative integer,
the result follows easily from \eqref{li}, \eqref{lili}, and \refL{LFFK}, 
which then imply that 
\begin{align}
&  \gG(-\ga)(1-z)^\ga\odot\gG(-\gb)(1-z)^\gb 
\notag\\&
=\bigpar{\Li_{\ga+1}(z)+ P_1(z)+ \oz{\rga}}
\odot
\bigpar{\Li_{\gb+1}(z)+ P_2(z)+ \oz{\rgb}}
\notag\\&
=\Li_{\ga+\gb+2}(z)+P_3(z)+\oz{\rga+\rgb+1}
\notag\\&
=\gG(-\ga-\gb-1)(1-z)^{\ga+\gb+1}+P_4(z)+\oz{\rga+\rgb+1},
\end{align}
where $P_i(z)$ are polynomials. 
[Note that $P(z)\odot f(z)$ is a polynomial for any polynomial $P$ and
analytic $f$, and that we may assume $\deg(P_4(z))<\rga+\rgb+1$
by the comment after \eqref{Pa}.]

Finally, if $\ga$ is a non-negative integer, then $(1-z)^\ga$ is a polynomial
and thus the \lhs{} of \eqref{lih2} is a polynomial, so
\eqref{lih2} holds trivially [with $1/\gG(-\ga)=0$].
The same holds if $\gb$ is a non-negative integer.
\end{proof}

\subsection{Generating functions}

Let $(b_n)_1^\infty$ be a given sequence of constants and 
consider the toll function $f(T):=b_{|T|}$ and the corresponding additive
functional $F(T)$ given by \eqref{F}.
We are mainly interested in the case $b_n=n^\ga$, but will also consider
$b_n=n^\ga-c$ below for a suitable constant $c$. 
In the present subsection, $b_n$ can be arbitrary if we
regard the generating functions as formal power series; if we assume
$b_n=O(n^K)$ for some $K$, then the generating functions below converge and
are analytic at least in the unit disc.

We are interested in the random variable $F(\cT_n)$. We denote its
moments by
\begin{align}\label{mell}
  m_n\xxl:=
\E[F(\cT_n)^\ell] 
\end{align}
for integer $\ell\ge0$.
Define the generating functions
\begin{align}\label{Mell}
  M_\ell(z):=\E \bigsqpar{F(\cT)^\ell z^{|\cT|}}
=\sumn q_n \E \bigsqpar{F(\cT)^\ell z^{|\cT|}\mid |\cT|=n}
=\sumn q_n m\xxl_n z^n
.\end{align}
Note that $M_0(z)=y(z)$, see \eqref{yz}.

The generating functions $M_\ell$ can be calculated recursively as follows,
using Hadamard products and the
generating function 
\begin{align}\label{Bz}
  B(z):=\sumn b_n z^n.
\end{align}

\begin{lemma}
  \label{LH}
For every $\ell\ge1$,
\begin{align}\label{lh}
  M_\ell(z)
=
\frac{z y'(z)}{y(z)}
\summl \frac{1}{m!}\sumxx \binom{\ell}{\ell_0,\dots,\ell_m}
B(z)^{\odot\ell_0} \odot 
\bigsqpar{zM_{\ell_1}(z)\dotsm M_{\ell_m}(z)\Phi\xxm\yz},
\end{align}
where $\sumxx$ is the sum over all $(m+1)$-tuples $(\ell_0,\dots,\ell_m)$ of
non-negative integers summing to $\ell$ such that 
$1\le\ell_1,\dots,\ell_m<\ell$.
\end{lemma}

\begin{proof}
Condition on the root degree $k$ of $\cT$, and let $\cT_1,\dots,\cT_k$ be
the principal subtrees as at the beginning of \refS{S:more_notation}.
Then \eqref{F2} can be written
\begin{align}\label{lh1}
  F(\cT)=f(\cT)+\sumik F(\cT_i)
= b_{|\cT|}+\sumik F(\cT_i).
\end{align}
Hence, the multinomial theorem yields the following, where for each $k$ 
we let $\sum$ denote the sum over all $(k+1)$-tuples $(\ell_0,\dots,\ell_k)$
summing to $\ell$ such that each $\ell_i\ge0$, and furthermore
$\cT_1,\dots,\cT_k$ are independent copies of $\cT$, 
and $|\cT|$ is
$1+|\cT_1|+\dots+|\cT_k|$:
\begin{align}\label{lh2}
  M_\ell(z)&
=
\sumko p_k \E \Bigsqpar{z^{|\cT|} \Bigpar{b_{|\cT|}+\sumik F(\cT_i)}^\ell}
\notag\\&
=\sumko p_k \sum \binom{\ell}{\ell_0,\dots,\ell_k}
\E \Bigsqpar{z^{|\cT|} b_{|\cT|}^{\ell_0}  F(\cT_1)^{\ell_1}\dotsm F(\cT_k)^{\ell_k}}
\notag\\&
=\sumko p_k \sum \binom{\ell}{\ell_0,\dots,\ell_k}
B(z)^{\odot\ell_0}\odot\E \Bigsqpar{z^{|\cT|} F(\cT_1)^{\ell_1}\dotsm F(\cT_k)^{\ell_k}}
\notag\\&
=\sumko p_k \sum \binom{\ell}{\ell_0,\dots,\ell_k}
B(z)^{\odot\ell_0}\odot\E \Bigsqpar{z\prodik \bigpar{z^{|\cT_i|} F(\cT_i)^{\ell_i}}}
\notag\\&
=\sumko p_k \sum \binom{\ell}{\ell_0,\dots,\ell_k}
B(z)^{\odot\ell_0}\odot \Bigsqpar{z\prodik\E\bigsqpar{ z^{|\cT_i|} F(\cT_i)^{\ell_i}}}
\notag\\&
=\sumko p_k \sum \binom{\ell}{\ell_0,\dots,\ell_k}
B(z)^{\odot\ell_0}\odot \Bigsqpar{z\prodik M_{\ell_i}(z)}.
\end{align}
We consider the terms where $\ell_i=\ell$ for some $1\le i\le k$
separately.
In this case, $\ell_0=0$ and $\ell_j=0$ for $j\neq i$, and thus the combined
contribution of these $k$ terms is, recalling $M_0(z)=y(z)$ and \eqref{Phi},
\begin{align}\label{lh3}
  \sumk p_k k \bigsqpar{z M_\ell(z) y(z)^{k-1}}
=z M_\ell(z)   \sumk p_k k y(z)^{k-1}
=z M_\ell(z) \Phi'(y(z)).
\end{align}
 Let
$\sumx$ denote the sum over the remaining terms, \ie, the terms with
$\ell_1,\dots,\ell_k<\ell$, and 
define
\begin{align}\label{lh4}
R_\ell(z):=
\sumko p_k \sumx \binom{\ell}{\ell_0,\dots,\ell_k}
B(z)^{\odot\ell_0}\odot \Bigsqpar{z\prodik M_{\ell_i}(z)}.
\end{align}
Using \eqref{lh3}--\eqref{lh4}, we can write \eqref{lh2} as
\begin{align}\label{lh5}
  M_\ell(z) = z\Phi'\yz M_\ell(z) + R_\ell(z).
\end{align}
Moreover, differentiating \eqref{ma} yields
\begin{align}\label{y'}
  y'(z)=\Phi\yz+z\Phi'\yz y'(z)
\end{align}
and thus, using \eqref{ma} again,
\begin{align}\label{y'a}
\bigpar{1-z\Phi'\yz}y'(z)
= \Phi\yz = y(z)/z
.\end{align}
Hence, \eqref{lh5} yields
\begin{align}\label{lh6}
  M_\ell(z) = \frac{R_\ell(z)}{1-z\Phi'\yz}
= \frac{zy'(z)}{y(z)}R_\ell(z).
\end{align}

Finally, in each term in the sum $\sumx$ in \eqref{lh4}, let $m\ge0$ be the
number of $\ell_1,\dots,\ell_k$ that equal 0.
By symmetry, we may assume that $\ell_1,\dots,\ell_m\ge1$ and
$\ell_{m+1}=\dots\ell_k=0$, and multiply by the symmetry factor $\binom
km$. Thus,
\begin{align}\label{lh7}
R_\ell(z)&=
\sumko p_k \summk\binom km  \sumxx \binom{\ell}{\ell_0,\dots,\ell_m}
B(z)^{\odot\ell_0}\odot \Bigsqpar{z \Bigpar{\prodim M_{\ell_i}(z)} y(z)^{k-m}}
\notag\\&
=
\summo\sumxx \binom{\ell}{\ell_0,\dots,\ell_m}
B(z)^{\odot\ell_0}\odot \Bigsqpar{z \Bigpar{\prodim M_{\ell_i}(z)}
\sum_{k = m}^{\infty} p_k \binom km y(z)^{k-m}}
\notag\\&
=\summo\sumxx \binom{\ell}{\ell_0,\dots,\ell_m}
B(z)^{\odot\ell_0}\odot \Bigsqpar{z \Bigpar{\prodim M_{\ell_i}(z)}
\frac{1}{m!}\Phi\xxm\bigpar{ y(z)}}
.\end{align}
The result \eqref{lh} follows from \eqref{lh6} and \eqref{lh7}, noting that
the sum $\sumxx$ is empty if $m>\ell$.
\end{proof}

\subsection{The mean}

For $\ell=1$, \eqref{lh} contains only the term $m=0$ and thus
$\ell_0=\ell=1$.
Hence, \refL{LH} yields, recalling \eqref{ma},
\begin{align}\label{sx1}
  M_1(z)=\frac{zy'(z)}{y(z)}\cdot \bigpar{B(z)\odot \bigsqpar{z\Phi(y(z))}}
=\frac{zy'(z)}{y(z)}\cdot\bigpar{ B(z)\odot y(z)}
.\end{align}

Let us first consider the factor $zy'(z)/y(z)$. 
It follows from \eqref{ma} that $y(z)=0$ implies $z=0$, and thus $z/y(z)$ is
analytic in any domain where $y(z)$ is.
Hence, $zy'(z)/y(z)$ is \gda, since $y(z)$ is.
Moreover, by Cauchy's estimates as in \cite[Theorem 6]{FillFK},
\eqref{yz1} implies, as $z\to1$,
\begin{align}\label{sx2}
  y'(z)=2\qqw\gs\qw(1-z)\qqw+o\bigpar{|1-z|\qqw}.
\end{align}
Consequently,
\begin{align}\label{sx3}
\frac{z y'(z)}{y(z)}=2\qqw\gs\qw(1-z)\qqw+o\bigpar{|1-z|\qqw}.
\end{align}

We turn to the second factor $B(z)\odot y(z)$. 
We consider first the case 
\begin{align}
  \label{bn1}
f(n)=
b_n=n^\ga,
\qquad n\ge1,
\end{align}
for some $\ga\in\bbC$;
then 
$F=F_\ga$ and, by \eqref{Xn},
\begin{align}\label{bnf}
F(\ctn)=X_n(\ga).  
\end{align}
By \eqref{bn1} and \eqref{Li}, $B(z)=\Li_{-\ga}(z)$, a  
polylogarithm function, and thus \eqref{li} yields,
at least  for $\rga>-1$,
\begin{align}
\label{Bo}
B(z) = \gG(1+\ga)(1-z)^{-\ga-1}+\oz{-\rga-1}.
\end{align}
Furthermore, 
by the definitions,
\begin{align}\label{by}
  B(z)\odot y(z) 
= \sumn b_n q_n z^n
=\sumn q_n n^\ga z^n
=
\E \bigsqpar{|\cT|^\ga z^{|\cT|}}
.\end{align}

\begin{lemma}\label{LC1>}
Let $\rga>\frac12$ and let $b_n:=n^{\ga}$.
Then, as $z\to1$ in some \GDD,
  \begin{align}\label{lc1>}
 M_1(z) = 
\frac{\gs\qww}{2\sqrt\pi}\gG\bigpar{\ga-\tfrac12}
(1-z)^{-\ga}
+\oz{-\rga}.
  \end{align}
\end{lemma}

\begin{proof}
By \eqref{Bo} and \eqref{yz1} together with \refLs{LFFK} and  \ref{LIH2},
and the fact that
$B(z)\odot1=0$,
\begin{align}\label{sixten}
  B(z) \odot y(z )
&=  
-\gG(1+\ga)(1-z)^{-\ga-1}\odot \sqrt2\gs\qw (1-z)\qq +\oz{-\rga+\frac12}
\notag\\&
=- 2\qq\gs\qw\frac{\gG(\ga-\half)}{\gG(-\frac12)}(1-z)^{-\ga+\half}
+\oz{-\rga+\frac12}.
\end{align}
 The result follows by \eqref{sx1} and \eqref{sx3}.
\end{proof}

\subsection{The mean when $0<\rga<\frac12$}\label{SSmean<}

Consider now the case 
$\rga<\frac12$. 
If we still take $b_n=n^\ga$ as in \eqref{bn1},
then \eqref{by} and \eqref{pct} show that $B(z)\odot y(z)$ is continuous in
the closed unit disc, 
and  a comparison with \eqref{mu} yields
\begin{align}\label{by1}
  (B\odot y)(1) = \E |\cT|^\ga = \mu(\ga)
.\end{align}
Hence, \eqref{sixten} cannot hold, since the \rhs{} tends to 0 as $z\to1$.
Actually, it follows from the arguments below that
the leading term in $B(z)\odot y(z)$ is the constant
$\mu(\ga)$, which by \eqref{sx1} and singularity analysis corresponds to the
fact that the leading term in \eqref{te-} is $\mu(\ga)n$.
We recall from \refS{S:intro} that when $\rga<\frac12$, we want to subtract
this term. In the present setting, we achieve this by 
modifying \eqref{bn1} and 
instead taking
\begin{align}
  \label{bn2}
f(n)=b_n:=n^\ga-\mu(\ga)
.\end{align}
Then \eqref{bnf} is modified to
\begin{align}\label{bnf2}
F(\ctn)=
\sum_{v\in\tn}\bigsqpar{|\tnv|^\ga-\mu(\ga)}
=X_n(\ga)-\mu(\ga)n,
\end{align}
and
\eqref{by} is modified to
\begin{align}\label{by2}
  B(z)\odot y(z) 
=\sumn q_n [n^\ga-\mu(\ga)] z^n
=
\E \bigsqpar{|\cT|^\ga z^{|\cT|}}-\mu(\ga)y(z)
.\end{align}
In particular,
\begin{align}\label{by0}
  (B\odot y)(1)=\E|\cT|^\ga- \mu(\ga)=0.
\end{align}

\begin{lemma}\label{LC1<}
Let $0<\rga<\frac12$ and let $b_n:=n^{\ga}-\mu(\ga)$.
As $z\to1$ in some \GDD,
  \begin{align}\label{lc1<}
 M_1(z) = 
\frac{\gs\qww}{2\sqrt\pi}\gG\bigpar{\ga-\tfrac12}
(1-z)^{-\ga}
+\oz{-\rga}.
  \end{align}
\end{lemma}

\begin{proof}
We now have, by \eqref{bn2} and \eqref{li},
\begin{align}\label{Box}
  B(z)&
=\Li_{-\ga}(z)-\mu(\ga)z(1-z)\qw
\notag\\&
= \gG(1+\ga)(1-z)^{-\ga-1}+\oz{-\rga-1},
\end{align}
just as in \eqref{Bo}.
Then, arguing 
as for~\eqref{sixten} 
using \eqref{yz1} and \refLs{LFFK} and \ref{LIH2} 
now yields
\begin{align}\label{sixten2}
  B(z)&\odot y(z )
=  
-\gG(1+\ga)(1-z)^{-\ga-1}\odot \sqrt2\gs\qw (1-z)\qq +P_1(z)
+\oz{-\rga+\frac12}
\notag\\&
=- 2\qq\gs\qw\frac{\gG(\ga-\half)}{\gG(-\frac12)}(1-z)^{-\ga+\half}
+P_2(z)
+\oz{-\rga+\frac12},
\end{align}
where $P_1(z),P_2(z)\in\cP_{\frac12-\rga}$ and thus are constants.
Letting $z\to1$ in \eqref{sixten2}
shows that $P_2(z)=(B\odot y)(1)=0$, by \eqref{by0}.
Hence, the result in \eqref{sixten} holds in the present case too, 
and the result follows again by \eqref{sx1} and \eqref{sx3}.
\end{proof}

\subsection{Higher moments}

In the remainder of this \refS{Smom}, we assume that $\rga>0$,
and that we have chosen $b_n$ by 
\eqref{bn1} or \eqref{bn2}
so that
\begin{align}\label{bn12}
  b_n:=
  \begin{cases}
    n^\ga,& \rga\ge\frac12,
\\
n^\ga-\mu(\ga),& 0<\rga<\frac12.
  \end{cases}
\end{align}
In the present subsection we also assume $\rga\neq\frac12$.

We need one more general lemma. 
\begin{lemma}
  \label{LC}
Under our standing assumptions $\E\xi=1$ and\/ \mbox{$0<\Var\xi<\infty$},
the function $\Phi\xxm(y(\cdot))$ is \gda{} for every $m\ge0$, and as $z\to1$ in some \GDD,
\begin{align}\label{lc}
  \Phi\xxm\yz
=
  \begin{cases}
    O(1), & m\le 2,
\\
\oz{1-\frac{m}{2}}, & m\ge3.
  \end{cases}
\end{align}
\end{lemma}

\begin{proof}
As noted at the beginning of \refS{S:more_notation}, 
$y(z)$ is \gda.
It follows from~\eqref{yz1} that for 
some \GDD{}  $\gD_1$,
if $z\in\gD_1$
with $|1-z|$ small enough, then
\begin{align}\label{ca1}
  |y(z)|<1-c|1-z|\qq.
\end{align}
Moreover, the definition \eqref{yz} implies that $|y(z)|\le1$ for $|z|=1$
with strict inequality unless $z=1$.
Hence, by continuity, for some $\gd,\eta>0$, $|y(z)|\le 1-\eta$ when
$z\in\gD_1$, $|1-z|\ge\eps$, and $|z|\le1+\gd$.
It follows that \eqref{ca1} holds (with a new $c>0$) for all $z$ in the
\GDD{} $\gD_2:=\set{z\in\gD_1:|z|<1+\gd}$.

In particular, $|y(z)|<1$ in $\gD_2$ and thus $\Phi\xxm\yz$ is
analytic in $\gD_2$.

The assumption $\E\xi^2<\infty$ implies that $\Phi$, $\Phi'$ and $\Phi''$ are
bounded and continuous functions on the closed unit disc. Hence,  \eqref{lc}
holds for $m\le2$. 

Now suppose $m\ge3$.
Since $\Phi''$ is continuous, we have $\Phi''(z)-\Phi''(1)=o(1)$ as $z\to1$
with $|z|<1$. Hence it folows from Cauchy's estimates that
\begin{align}\label{ca2}
  \Phi\xxm(z) = o\bigpar{(1-|z|)^{2-m}}
\qquad \text{as  $z\to 1$ with $|z|<1$}
.\end{align}
The result \eqref{lc} for $m\ge3$ follows from \eqref{ca2} and \eqref{ca1}.
\end{proof}

\begin{lemma}\label{LC2}
Assume $\rga\in(0,\frac12)\cup(\frac12,\infty)$ and that \eqref{bn12} holds.
Then,   for every $\ell\ge1$, 
$M_\ell(z)$ is \gda, and 
as $z\to1$ in some \GDD,
  \begin{align}\label{lc2}
 M_\ell(z) = 
\kkk_\ell\gs^{-\ell-1}(1-z)^{-\ell(\ga+\frac12)+\frac12}
+\oz{-\ell(\rga+\frac12)+\frac12},
  \end{align}
where
the constants $\kkk_\ell$ are given recursively by
\begin{align}\label{kkk1}
\kkk_1&
= \frac{1}{2\sqrt\pi}\gG\bigpar{\ga-\tfrac12},
\\\label{kkk2}
\kkk_\ell&=
2^{-3/2} \sum_{j=1}^{\ell-1} \binom{\ell}{j}\kkk_j\kkk_{\ell-j}
+ 2\qqw \ell \frac{\gG\bigpar{\ell(\ga+\frac12)-1}}
 {\gG\bigpar{(\ell-1)(\ga+\frac12)-\frac12}}\kkk_{\ell-1}
.\end{align}
\end{lemma}

The \GDD{} may depend on $\ell$.
We write $\kkk_\ell$ in \eqref{kkk1}--\eqref{kkk2} as $\kkk_\ell(\ga)$
when we want to emphasize the dependence on $\ga$.

\begin{proof}
We use induction on $\ell$, based on \refL{LH}.
First, this shows that $M_\ell$ is \gda, using 
the fact that~$B$ and, by \refL{LC}, $\Phi\xxm(y(\cdot))$ are,
together with \refL{LFFK}.

To show \eqref{lc2} by induction, we note that
the base case $\ell=1$ is \refLs{LC1>} and \ref{LC1<}.

Assume thus $\ell\ge2$, and 
let $A:=-\ell(\ga+\frac12)+\frac12<-\frac12$
be the exponent of $1 - z$ in \eqref{lc2}.
Consider one of the terms in \eqref{lh}.
By the induction hypothesis and \refL{LC}, we have
\begin{align}\label{lc1qq}
& zM_{\ell_1}(z)\dotsm M_{\ell_m}(z)\Phi\xxm\yz
= O\bigpar{|1-z|^{-\sumim \ell_i(\Re \ga+\frac12)+\frac{m}{2}}\Phi\xxm\yz}
\notag\\
&\hskip4em
=
  \begin{cases}
    O\bigpar{|1-z|^{-(\ell-\ell_0)(\Re \ga+\frac12)+\frac{m}2}}, & m\le 2,
\\
  o\bigpar{|1-z|^{-(\ell-\ell_0)(\Re \ga+\frac12)+1}}, & m\ge3.
  \end{cases}
\end{align}
Since $\ell - \ell_0 \ge m$, the exponent here is
\begin{align}
\label{exp_real_part}
  -(\ell-\ell_0)(\rga+\tfrac12)+\frac{m\bmin 2}{2}
\le
  -m(\rga+\tfrac12)+\frac{m\bmin 2}2
\le-m\,\rga\le0.
\end{align}
Furthermore, 
\eqref{Bo} and \eqref{Box} show that, 
for both $\rga>\frac12$ and $\rga<\frac12$,
\begin{align}\label{Boy}
B(z)=\Oz{-\rga-1}  
\end{align}
and thus \refL{LFFK} 
applies $\ell_0$ times and yields
\begin{align}\label{lc3}
&B(z)^{\odot\ell_0} \odot 
\bigsqpar{zM_{\ell_1}(z)\dotsm M_{\ell_m}(z)\Phi\xxm\yz}
\notag\\
&\hskip4em
=
  \begin{cases}
    O\bigpar{|1-z|^{-(\ell-\ell_0)(\rga+\frac12)+\frac{m}2-\ell_0\rga}}, & m\le 2,
\\
  o\bigpar{|1-z|^{-(\ell-\ell_0)(\rga+\frac12)+1-\ell_0\rga}}, & m\ge3.
  \end{cases}
\end{align}
The exponent here is
\begin{align}\label{lc4}
  -(\ell-\ell_0)(\rga+\tfrac12)+\frac{m\bmin2}2-\ell_0\rga
=   -\ell(\rga+\tfrac12)+\frac{\ell_0+m\bmin2}2.
\end{align}
For $m\ge3$, this is at least $-\ell(\rga+\frac12)+1=\REA+\frac12$,
and thus the term is
\begin{align}\label{lc5}
  \oz{\REA+\frac12}.
\end{align}
We will see that this contributes only to the error term in \eqref{lc2},
so such terms may be ignored.
Similarly, for every term with $m\le2$ and $\ell_0+m > 2$, 
 the exponent considered in 
\eqref{lc4} is strictly larger than $\REA+\frac12$, and thus
such terms also satisfy \eqref{lc5} and may be ignored.

If $m=1$, then $\ell_1<\ell$, and thus $\ell_0\ge1$.
Hence, the only remaining terms to consider are 
(1) $m=0$ and thus $\ell_0=\ell$;
(2) $m=1$ and $\ell_0=1$; 
(3) $m=2$ and $\ell_0=0$.

Furthermore, also the term with $m=0$ can be ignored, since it is
\begin{align}\label{lc6}
B(z)^{\odot\ell} \odot \bigsqpar{z\Phi\yz}  
&=
B(z)^{\odot\ell} \odot y(z)
\notag\\&
=
B(z)^{\odot\ell} \odot 1
+ B(z)^{\odot\ell} \odot\bigpar{y(z)-1}  
,\end{align}
where $B(z)^{\odot\ell} \odot1$ vanishes
and $y(z)-1=\Oz{\frac12}=\oz{0}$ by \eqref{yz1}; 
hence \refL{LFFK}\ref{LFFKo} yields 
\begin{align}\label{lc7}
B(z)^{\odot\ell} \odot \bigsqpar{z\Phi\yz}  
=\oz{-\ell\,\rga}
=\oz{\REA+\frac12}.
\end{align}

Consequently, recalling 
\eqref{lh7}, we have
\begin{align}\label{lc8}
  R_l(z) = \ell B(z) \odot \bigsqpar{z M_{\ell-1}(z) \Phi'\yz}
&+ \frac12\sum_{j=1}^{\ell-1}\binom{\ell}{j} z M_j(z)M_{\ell-j}(z) \Phi''\yz
\notag\\
&+\oz{\REA+\frac12}
.\end{align}
Since $\Phi'$ is continuous in the unit disc with $\Phi'(1)=1$,
the induction hypothesis implies that
\begin{align}\label{lc9}
z M_{\ell-1}(z)\Phi'\yz = 
\kkk_{\ell-1}\gs^{-\ell}(1-z)^{-(\ell-1)(\ga+\frac12)+\frac12}
+\oz{-(\ell-1)(\rga+\frac12)+\frac12},
  \end{align}
Hence, \eqref{Bo}, \eqref{Box}, and Lemmas~\ref{LFFK} and~\ref{LIH2} yield
\begin{multline}\label{lc10}
B(z)\odot\bigsqpar{z M_{\ell-1}(z)\Phi'\yz} 
\\= 
\kkk_{\ell-1}\gs^{-\ell}\frac{\gG\bigpar{\ell(\ga+\frac12)-1}}
{\gG\bigpar{(\ell-1)(\ga+\frac12)-\frac12}}(1-z)^{A+\frac12}
+\oz{\REA+\frac12}.
  \end{multline}
Similarly, the induction hypothesis yields, using $\Phi''(1)=\gss$,
\begin{multline}\label{lc11}
 \sum_{j=1}^{\ell-1}\binom{\ell}{j} z M_j(z)M_{\ell-j}(z) \Phi''\yz  
\\
=\sum_{j=1}^{\ell-1}\binom{\ell}{j}  \kkk_j\kkk_{\ell-j}\gs^{-\ell}(1-z)^{A+\frac12}
+\oz{\REA+\frac12}.
\end{multline}
The result \eqref{lc2} now follows from 
\eqref{lh6}, \eqref{sx3}, and \eqref{lc8}--\eqref{lc11}.
\end{proof}

\subsection{Mixed moments. Proof of \refT{TXmom}}
\label{SSmom-mix}
We may extend \refT{TXmom} to mixed moments of $X_n(\ga_1),\dots,X_n(\ga_m)$,
for several given $\ga_1,\dots,\ga_m$,
using the
same arguments with only notational differences.
For convenience, define 
\begin{align}\label{ebbe}
  \QX_n(\ga):= 
  \begin{cases}
     n^{-\ga-\frac12}X_n(\ga), & \rga>\frac12,
\\
    n^{-\ga-\frac12}\bigpar{X_n(\ga)-\mu(\ga)n}, & 0<\rga<\frac12.
  \end{cases}
\end{align}
We consider for simplicity  only two different 
values of $\ga$; the general case is similar but left to the reader.

\begin{theorem}\label{Tmix}
Let $\rga_1,\rga_2\in(0,\frac12)\cup(\frac12,\infty) $,
and write $\ga_i':=\ga_i+\frac12$. 
Then, for any integers $\ell_1,\ell_2\ge0$
with $\ell_1+\ell_2\ge1$,
\begin{align}\label{tmix}
\gs^{\ell_1+\ell_2}\E\bigsqpar{\QX_n(\ga_1)^{\ell_1}\QX_n(\ga_2)^{\ell_2}}
\to
  \E \bigsqpar{Y(\ga_1)^{\ell_1}Y(\ga_2)^{\ell_2}}
=\frac{\sqrt{2\pi}}{\gG(\ell_1\ga_1'+\ell_2\ga_2'-\frac12)}\kkk_{\ell_1,\ell_2},
\end{align}
where
$\kkk_{1,0}=\kkk_1(\ga_1)$ 
and
$\kkk_{0,1}=\kkk_1(\ga_2)$
are given by \eqref{kkk1}, 
and,
for $\ell_1+\ell_2\ge2$,
\begin{align}\label{kkkll}
  \kkk_{\ell_1,\ell_2}&
=
2^{-3/2} \sum_{0<j_1+j_2<\ell_1+\ell_2} 
\binom{\ell_1}{j_1}\binom{\ell_2}{j_2}\kkk_{j_1,j_2}\kkk_{\ell_1-j_1,\ell_2-j_2}
\notag\\&\qquad
+ 2\qqw \ell_1 \frac{\gG\bigpar{\ell_1\ga_1'+\ell_2\ga_2'-1}}
 {\gG\bigpar{\ell_1\ga_1'+\ell_2\ga_2'-1-\ga_1}}\kkk_{\ell_1-1,\ell_2}
\notag\\&\qquad
+ 2\qqw \ell_2 \frac{\gG\bigpar{\ell_1\ga_1'+\ell_2\ga_2'-1}}
 {\gG\bigpar{\ell_1\ga_1'+\ell_2\ga_2'-1-\ga_2}}\kkk_{\ell_1,\ell_2-1}
.\end{align}
\end{theorem}

\begin{proof}[Proof of \refTs{TXmom} and \ref{Tmix}]
For a given $\ga$, we continue to use the choice 
\eqref{bn12} of $b_n$. This yields 
\eqref{bnf} ($\rga\ge\frac12$) or \eqref{bnf2} ($\rga<\frac12$),
i.e., now writing $\dF_\ga$ for $F$,
\begin{align}\label{abba}
\dF_\ga(\ctn)=
  \begin{cases}
     X_n(\ga), & \rga\ge\frac12,
\\
    {X_n(\ga)-\mu(\ga)n}, & 0<\rga<\frac12.
  \end{cases}
\end{align}

Hence, in both cases, 
$\QX_n(\ga)=n^{-\ga-\frac12}\dF_\ga(\ctn)$, and
\refT{TX} yields 
\begin{align}\label{frater}
\QX_n(\ga)=n^{-\ga-\frac12}\dF_\ga(\ctn)\dto \gs\qw Y(\ga);
\end{align}
moreover, this holds jointly for any number of $\ga$
by the proof of \refT{TX}.

The asymptotic formula \eqref{lc2} yields, 
by   \eqref{Mell} and standard singularity analysis
\cite[Chapter VI]{FS},
  \begin{align}\label{tm1}
q_n m_n\xxl  = [z^n]M_\ell(z) \sim \kkk_\ell \gs^{-\ell-1} 
\frac{1}{\gG \bigpar{\ell(\ga+\frac12)-\frac12}} n^{\ell(\ga+\frac12)-\frac32}. 
  \end{align}
Together with \eqref{pct} (with $h = 1$) for $q_n$, this yields
  \begin{align}\label{morm}
 m_n\xxl   
\sim 
\frac{\sqrt{2\pi}\kkk_\ell \gs^{-\ell}}
{\gG \bigpar{\ell(\ga+\frac12)-\frac12}} n^{\ell(\ga+\frac12)}. 
  \end{align}
Recall that $m_n\xxl:=\E \dF_\ga(\ctn)^\ell$ by  \eqref{mell}.
Hence,
\eqref{morm} can
be written as 
\begin{align}\label{farf}
\gs^\ell  \E \QX_n(\ga)^\ell\to 
\frac{\sqrt{2\pi}}
{\gG \bigpar{\ell(\ga+\frac12)-\frac12}}  \kkk_\ell
=:\kk_\ell
,\end{align}
where we thus denote the \rhs{} by $\kk_\ell $.
The recursion \eqref{kk1}--\eqref{kk2} then follows from
\eqref{kkk1}--\eqref{kkk2}.

This shows most parts of \refT{TXmom},
but it remains to show that the $\kk_\ell$'s (the limits of moments) 
are the moments of the limit (in distribution) $Y(\ga)$ of $\gs \QX_n(\ga)$.
For real $\ga$, this follows from \eqref{farf} by a standard argument,
but for general complex $\ga$ we want to consider absolute moments, so we
postpone the proof of this, and first turn to \refT{Tmix}.

Define, in analogy with \eqref{mell}--\eqref{Mell}, 
  \begin{align}
    \label{mell2}
m_n\xxll&:=\E\bigsqpar{\dF_{\ga_1}(\ctn)^{\ell_1}\dF_{\ga_2}(\ctn)^{\ell_2}},
\\\label{Mell2}
  M_{\ell_1,\ell_2}(z)
&:=\E \bigsqpar{\dF_{\ga_1}(\cT)^{\ell_1}\dF_{\ga_2}(\cT)^{\ell_2} z^{|\cT|}}
=\sumn q_n m\xxll_n z^n
.\end{align}
It is straightforward to extend \refL{LH} to the following, valid for every $\ell,r\ge0$
with 
\mbox{$\ell+r\ge1$}: 
\begin{multline}\label{lhz}
  M_{\ell,r}(z)
=
\frac{z y'(z)}{y(z)}
\sum_{m=1}^{\ell+r} \frac{1}{m!}\sumxx \binom{\ell}{\ell_0,\dots,\ell_m}
\binom{r}{r_0,\dots,r_m}
B_{\ga_1}(z)^{\odot\ell_0} 
\\
\odot 
B_{\ga_2}(z)^{\odot r_0} \odot 
\bigsqpar{zM_{\ell_1,r_1}(z)\dotsm M_{\ell_m,r_m}(z)\Phi\xxm\yz},
\end{multline}
where $\sumxx$ is the sum over all pairs of
$(m+1)$-tuples $(\ell_0,\dots,\ell_m)$  and $(r_0,\dots,r_m)$ 
of non-negative integers that sum to $\ell$ and $r$, respectively,
such that $1\le\ell_i+r_i<\ell + r$ for every $i$.

Then, the inductive proof of \refL{LC2} is easily extended to show that
in some \GDD{} (possibly depending on $\ell_1$ and $\ell_2$)
  \begin{align}\label{lc2ll}
 M_{\ell_1,\ell_2}(z) = 
\kkk_{\ell_1,\ell_2}\gs^{-\ell_1-\ell_2-1}(1-z)^{-\ell_1\ga_1'-\ell_2\ga_2'+\frac12}
+\oz{-\ell_1\rga_1'-\ell_2\rga_2'+\frac12},
  \end{align}
with $\kkk_{\ell_1,\ell_2}$ given by \eqref{kkk1} and \eqref{kkkll}.
Singularity analysis yields, as for the special case \eqref{morm},
\begin{align}\label{varin}
\gs^{\ell_1+\ell_2}  \E \bigsqpar{\QX_n(\ga_1)^{\ell_1}\QX_n(\ga_2)^{\ell_2}}
\to 
\frac{\sqrt{2\pi}}
{\gG \bigpar{\ell_1\ga_1'+\ell_2\ga_2'-\frac12}}  \kkk_{\ell_1,\ell_2}
=:\kk_{\ell_1,\ell_2}
.\end{align}

In particular, for any $\ga$ in the domain, we may take $\ga_1:=\ga$ and
$\ga_2:=\bga$. Then \eqref{varin} shows, in particular,
that for any integer $\ell\ge1$,
$\E\bigabs{\QX_n(\ga)}^{2\ell}=
\E \bigsqpar{\QX_n(\ga)^{\ell}\QX_n(\bga)^{\ell}}$
converges as $\ntoo$.

By a standard argument, 
see \eg{} \cite[Theorems 5.4.2 and 5.5.9]{Gut},
this implies uniform
integrability of each smaller power of $\QX_n(\ga)$, 
which together with \eqref{frater} implies
convergence of all lower moments to the moments of the limit $\gs\qw Y(\ga)$.
This completes the proof of
\eqref{mtx1}--\eqref{mtx<} with
\begin{align}
  \label{tm2}
\kk_\ell:=\E Y(\ga)^\ell
=\frac{\sqrt{2\pi}}{\gG \bigpar{\ell(\ga+\frac12)-\frac12}} 
\kkk_\ell.
\end{align}

Similarly, using \Holder's inequality, the sequence
$\QX_n(\ga_1)^{\ell_1}\QX_n(\ga_2)^{\ell_2}$ is uniformly integrable for
every fixed $\ga_1,\ga_2,\ell_1,\ell_2$, and \eqref{tmix} follows from
the joint convergence in~\eqref{frater} and~\eqref{varin}.
\end{proof}

\begin{example}\label{Emix11}
  Taking $\ell_1=\ell_2=1$ in \eqref{kkkll}, we obtain, 
with obvious  notation
and using \eqref{kkk1},
\begin{align}\label{kkk11}
&  \kkk_{1,1}(\ga,\gb)
= 2^{-\frac12}\kkk_1(\ga)\kkk_1(\gb)
+ 2^{-\frac12} \frac{\gG(\ga+\gb)}{\gG(\gb)}\kkk_1(\gb)
+ 2^{-\frac12} \frac{\gG(\ga+\gb)}{\gG(\ga)}\kkk_1(\ga)
\notag\\&
=
\frac{2^{-\frac{5}2}}{\pi}\gG(\ga-\tfrac12)\gG(\gb-\tfrac12)
+\frac{2^{-\frac{3}2}}{\sqrt\pi}\frac{\gG(\ga+\gb)\gG(\gb-\tfrac12)}{\gG(\gb)}
+\frac{2^{-\frac{3}2}}{\sqrt\pi}\frac{\gG(\ga+\gb)\gG(\ga-\tfrac12)}{\gG(\ga)}.
\end{align}
In particular, taking $\gb=\bga$ and using \eqref{tmix} and \eqref{varin},
\begin{align}\label{e|2|}
\E|Y(\ga)|^2&
=  \kk_{1,1}(\ga,\bga)
=\frac{\sqrt{2\pi}}{\gG(2\rga+\frac12)}
\kkk_{1,1}(\ga,\bga)
\notag\\&
=
\frac{|\gG(\ga-\tfrac12)|^2}{4\sqrt{\pi}\gG(2\rga+\frac12)}
+\frac{\gG(2\rga)}{\gG(2\rga+\frac12)}\Re\frac{\gG(\ga-\tfrac12)}{\gG(\ga)}.
\end{align}
\end{example}

\begin{example}\label{Emix12}
As mentioned in \refE{Ega=1}, for 
the case of joint moments of
$Y(1)$ and $Y(2)$, \refT{Tmix} 
yields the recursion formula given in \cite{SJ146};
the method used there is related to the one used here,
but seems to apply only for integer $\ga$.
\end{example}

\begin{remark}\label{Runique}
The mixed moments of $Y(\ga)$ and $\overline{Y(\ga)}=Y(\bga)$ 
determine  the distribution
of $Y(\ga)$ uniquely, for any $\ga\neq\frac12$ with $\rga>0$.
In fact, there exists $C(\ga)>0$ such that for every $\ell\ge1$,
\begin{align}\label{expb}
  \E|Y(\ga)|^\ell \le C(\ga)^\ell \ell!,
\end{align}
and thus $(\Re Y(\ga),\Im Y(\ga))$ has a finite moment generating function
in a neighborhood of the origin. The estimate \eqref{expb} was shown for
real $\ga$ in \cite[Lemma 3.4]{FillK04} (with proof in \cite{FillK04v1});
the general case is similar, considering even $\ell$ and using induction
and \eqref{kkkll}. 

The constant $C(\ga)$ in \eqref{expb} 
can be taken uniformly bounded on compact subsets of
$H_+\setminus\set{\frac12}$. 
Moreover, \eqref{expb} obviously implies the same estimate for
$\tY(\ga)=Y(\ga)-\E Y(\ga)$ [with $C(\ga)$ replaced by $2C(\ga)$], and then
we can argue using analyticity as  in the proof of \refL{LU4} below and
conclude that \eqref{expb} holds also for $\tY(\frac12)$, which thus also is
determined by its moments, as noted in \cite{FillK04}.
\end{remark}

\subsection{Uniform estimates}\label{SSuniform}

In this \refS{Smom}, we have so far estimated moments for a fixed $\ga$, or
mixed moments for a fixed set of different $\ga$.
We turn to uniform estimates for $\ga$ in suitable sets.
This is rather straightforward 
if $\rga$ stays away from $\frac12$. However, we want
uniformity also for $\rea$ approaching (or equalling) $\frac12$, and this is
more complicated.
For our proofs, we assume throughout the present subsection the 
weak moment condition
\begin{align}\label{2+gd}
  \E\xi^{2+\gd}<\infty,
\end{align}
for some $\gd>0$. 
Throughout this subsection, $\gd$ is fixed;
we assume without loss of generality that $\gd\le1$.

\begin{problem}
  Do \refLs{LU1}--\ref{LU4} and \refT{T1mom}
hold without the extra condition \eqref{2+gd}?
(Cf.~\refR{RT1mom}.)
\end{problem}

We begin with some preliminaries.
We start with a standard estimate, included for completeness.
\begin{lemma}\label{Lgd}
  If \eqref{2+gd} holds with $0<\gd\le1$, then
  \begin{align}\label{lgd}
    \Phi(z)=z+\tfrac12\gss (1-z)^2 + O\bigpar{|1-z|^{2+\gd}}, 
\qquad |z|\le1.
  \end{align}
\end{lemma}
\begin{proof}
Let $z=1-w$, with $|z|\le1$.  Taylor's theorem yields the two estimates,
uniformly for $|z|\le1$ and $k\ge0$,
\begin{align}\label{ru1}
  z^k &=(1-w)^k=1-kw + O\bigpar{k^2|w|^2} 
=1-kw+\binom k2 w^2 + O\bigpar{k^2|w|^2}, 
\\\label{ru2}
z^k&=(1-w)^k=1-kw+\binom k2 w^2 + O\bigpar{k^3|w|^3} ,
\end{align}
and thus, taking a geometric mean of the $O$ terms in \eqref{ru1} and
\eqref{ru2}, 
\begin{align}\label{ru3}
  z^k&=1-kw+\binom k2 w^2 + O\bigpar{k^{2+\gd}|w|^{2+\gd}}.
\end{align}
Hence, \eqref{Phi} yields, using the assumption~\eqref{2+gd},
\begin{align}
  \Phi(z)
= \sumk p_k\Bigsqpar{1-kw+\binom k2 w^2 + O\bigpar{k^{2+\gd}|w|^{2+\gd}}}
= 1- w + \frac{\gss}{2} w^2 + O\bigpar{|w|^{2+\gd}}, 
\end{align}
which is \eqref{lgd}.
\end{proof}

This enables us to improve \eqref{yz1}.
\begin{lemma}\label{Lgdy}
  If \eqref{2+gd} holds with $0<\gd\le1$, then, for $z$ in some \GDD,
  \begin{align}\label{lgdy}
  y(z)=1-\sqrt2 \gs\qw(1-z)\qq+\Oz{\frac12+\gdd}.
  \end{align}
\end{lemma}
\begin{proof}
By  \cite[Lemma A.2]{SJ167},
$y(z)$ is analytic in some \GDD{} $\gD$ such that
$|y(z)|<1$ for $z\in\gD$ and \eqref{yz1} holds as $z\to1$ in $\gD$.
To show the improvement \eqref{lgdy},
it suffices to consider $z\in\gD$ close to 1, since the estimate is trivial when
$|1-z|$ is bounded below.

Let $w:=1-y(z)$. By \eqref{yz1} we have
$|w|=\Theta\bigpar{|1-z|^{\frac12}}$. 
The functional equation \eqref{ma} and \refL{Lgd} yield
\begin{align}
    y(z)/z=\Phi\yz = y(z)+\frac{\gss}2 w^2 + O\bigpar{|w|^{2+\gd}}
= y(z)+\frac{\gss}2 w^2\bigsqpar{1+  O\bigpar{|w|^{\gd}}}
\end{align}
and thus, for $|1-z|$ small,
\begin{align}
  \frac{\gss}2 w^2 
=\frac{1-z}{z}y(z)\bigsqpar{1+ O\bigpar{|w|^{\gd}}}
=(1-z)\bigsqpar{1+ \Oz{\gd/2}}.
\end{align}
The result \eqref{lgdy} follows.
\end{proof}

We need also  a uniform version of 
\refL{LFFK}\ref{LFFKO}. We state it in a rather general form.

\begin{lemma}
  \label{LU}
Let $\cI$ be an arbitrary index set, and suppose that $a_\iota,b_\iota$,
$\iota\in\cI$, are real numbers such that
$\sup_\cI|a_\iota|<\infty$, $\sup_\cI|b_\iota|<\infty$ and
$\sup_\cI(a_\iota+b_\iota+1)<0$.
Suppose that $g_\iota(z)$ and $h_\iota(z)$ are \gdaf{s} such that, in some
fixed \GDD{} $\gD$, 
$g_\iota(z)=\Oz{a_\iota}$ and $h_\iota(z)=\Oz{b_\iota}$,
uniformly in $\iota$. Then
\begin{align}
  g_\iota(z)\odot h_\iota(z) = \Oz{a_\iota+b_\iota+1},
\end{align}
in some fixed \GDD{} $\gD'$, uniformly in $\iota$.
\end{lemma}
\begin{proof}
  This follows from the proof of \cite[Proposition 9]{FillFK},
taking there the same  integration contour for all $\iota$.
\end{proof}

As a final preparation, we state a uniform version of a special case
of the asymptotic expansion of polylogarithms by \citet{Flajolet1999},
cf.~\eqref{li}.
A proof is given in \refApp{Apoly}.

\begin{lemma}
  \label{LUL1}
For every \GDD{} $\gD$ and every compact set 
$K\subset\bbC\setminus\set{1,2,\dots}$
we have 
\begin{align}
  \label{lul1}
\Li_{\ga}(z) = \gG(1-\ga)(1-z)^{\ga-1}  + \Ozo{\rga}
\end{align}
uniformly for $z\in\gD$ and $\ga\in K$.
\end{lemma}

We continue to assume \eqref{bn12}.
We now denote the generating function \eqref{Bz} by  $B_\ga(z)$; thus
\begin{align}\label{ululu}
  B_\ga(z)=
  \begin{cases}
    \Li_{-\ga}(z), & \rga\ge\frac12,
\\
    \Li_{-\ga}(z)-\mu(\ga)z(1-z)\qw, & \rga<\frac12.
  \end{cases}
\end{align}

The following lemma is the central step to establishing uniformity in the
estimates above. (Cf.~\refLs{LC1>} and \ref{LC1<}.)
Note that the lemma does not hold for $\ga=\frac12$; it is easily seen from 
\eqref{pct} that $B_{\xfrac12}(z)\odot y(z) =\Theta\bigpar{|\log|1-z||}$ as
$z\upto1$.
\begin{lemma}\label{LU1}
Assume that $\E\xi^{2+\gd}<\infty$.
  Let $K$ be a compact subset of $\set{\ga:\rga>0}\setminus\set{\frac12}$.
Then, 
\begin{align}\label{lu1}
  B_\ga(z)\odot y(z) = \Oz{\frac12-\rga}
\end{align}
in some fixed \GDD,
uniformly for $\ga\in K$.
\end{lemma}

\begin{proof}
  We consider three different cases, and therefore define
$K_1:=\set{\ga\in K:\rga\ge\frac12+\gdq}$,
$K_2:=\set{\ga\in K:\frac12\le \rga<\frac12+\gdq}$,
$K_3:=\set{\ga\in K: \rga<\frac12}$.
Estimates of the type $\Oz{a}$ below are valid
in some fixed \GDD, which may change from line to line.

\pfcase1{$\rga\ge\frac12+\gdq$.}
In this range, we have by \refL{LUL1}
\begin{align}\label{u1}
    B_\ga(z)=\Li_{-\ga}(z) = \Oz{-\rga-1},
\end{align}
uniformly in $\ga\in K_1$.
Furthermore, $y(z)=1+\Oz{\frac12}$ by \eqref{yz1} (or \refL{Lgdy}),
and $\frac12-\rga\le -\gdq$ for $\ga\in K_1$.
Hence, \refL{LU} yields
\begin{align}\label{u2}
  B_\ga(z)\odot y(z)
=
  B_\ga(z)\odot \bigpar{y(z)-1}
=\Oz{\frac12-\rga},
\end{align}
uniformly in $\ga\in K_1$.

\pfcase{2 and 3}{$0<\rga<\frac12+\gdq$.}
We have, by \eqref{lgdy} and \eqref{li},
\begin{align}\label{u3}
  y(z) = 1- c_1(1-z)\qq + \Oz{\frac12+\gdd}
= 1+ c_2\Li_{3/2}(z)+P(z) + \Oz{\frac12+\gdd},
\end{align}
where $P(z)$ is a polynomial that can be assumed to have degree less than
$\frac12+\gdd$, and thus $P(z)=C_1$, a constant.
Let
\begin{align}\label{u4}
  h(z):=y(z)-c_2\Li_{3/2}(z)-1-C_1
=\Oz{\frac12+\gdd}.
\end{align}

Let $\zddz$ denote the differential operator $z\ddz$.
Note the identity 
$\zddz(g_1(z) \odot g_2(z)) = \zddz g_1(z) \odot g_2(z)$
and that
$\zddz\Li_\ga(z)=\Li_{\ga-1}(z)$.
Thus,
\begin{align}\label{u5}
  \zddz\bigpar{\Li_{-\ga}(z)\odot h(z)}
=  {\zddz\Li_{-\ga}(z)\odot h(z)}
=  \Li_{-\ga-1}(z)\odot h(z)
.\end{align}
We have $\Li_{-\ga-1}(z)=\Oz{-\rea-2}$ uniformly in $\ga\in K$  by \refL{LUL1},
which together with \eqref{u4} and \refL{LU} yields
\begin{align}\label{u6}
  \zddz\bigpar{\Li_{-\ga}(z)\odot h(z)}
&=\Oz{-\rea-\frac12+\gdd}
,\end{align}
uniformly in $\ga\in K_2\cup K_3$.

Furthermore, by \refL{LUL1},
\begin{align}\label{u7}
\zddz\bigpar{ \Li_{-\ga}(z)\odot\Li_{3/2}(z)}&
=\zddz\Li_{-\ga+3/2}(z) 
=\Li_{-\ga+1/2}(z) 
\notag\\&
= \gG(\ga+\tfrac12) (1-z)^{-\ga-\frac12}+O\bigpar{|1-z|^{-\rga+\frac12}+1}
\end{align}
uniformly in $\ga\in K_2\cup K_3$.

The exponent $-\rga-\frac12+\gdd$ in \eqref{u6} lies in $[-1+\gdq,0)$,
and thus \eqref{u6} and \eqref{u7} yield, after division by $z$,
\begin{align}\label{u8}
\ddz\bigpar{\Li_{-\ga}(z) \odot y(z)}
&=
c_2\ddz\bigpar{\Li_{-\ga}(z) \odot \Li_{3/2}(z))}
+\ddz\bigpar{\Li_{-\ga}(z) \odot h(z))}
\notag\\&
=c_2 \gG(\ga+\tfrac12) (1-z)^{-\ga-\frac12}+\Oz{-\rga-\frac12+\gdd},
\end{align}
again uniformly in $\ga \in K_2 \cup K_3$.

We now consider Cases 2 and 3 separately.

\pfcase2{$\frac12\le \rga<\frac12+\gdq$.}
By integrating \eqref{u8} along a suitable contour, for example from 0 along
the negative real axis to $-|z|$ and then along the circle with radius $|z|$
to~$z$, 
\begin{align}\label{u9}
B_\ga(z)\odot y(z)=
{\Li_{-\ga}(z) \odot y(z)}
=c_2 \gG(\ga-\tfrac12) (1-z)^{-\ga+\frac12}+ O(1),
\end{align}
uniformly in $\ga\in K_2$, which implies \eqref{lu1}.

\pfcase{3}{$0<\rga<\frac12$.}
Recall that now
\begin{align}
  \label{u11}
B_\ga(z)\odot y(z)=\Li_{-\ga}(z)\odot y(z)-\mu(\ga)y(z),
\end{align}
see \eqref{ululu} and 
\eqref{by2}.
The estimate \eqref{yz1} implies, in a smaller \GDD, 
\begin{align}
\label{u12}
  y'(z)=\Oz{-\frac12}.
\end{align}
Furthermore, $\mu(\ga)=O(1)$ on $K_3$, as a consequence of \refT{TM}.
Hence \eqref{u11}, \eqref{u8}, and \eqref{u12} 
imply
\begin{equation}
\label{u13}
\ddz\bigpar{B_{\ga}(z) \odot y(z)}
=c_2 \gG(\ga+\tfrac12) (1-z)^{-\ga-\frac12}+\Oz{-((\rga+\frac12-\gdd)\vee\frac12)}.
\end{equation}
We now have $(B_\ga\odot y)(1)=0$ by \eqref{by0}, 
and thus \eqref{lu1} follows from \eqref{u13} by integration,
noting that the exponents in \eqref{u13} stay away from $-1$ for $\ga\in K_3$.
\end{proof}

\begin{lemma}
  \label{LU2}
Assume that $\E\xi^{2+\gd}<\infty$.
  Let $K$ be a compact subset of $\set{\ga:\rga>0}\setminus\set{\frac12}$.
Then, 
with notations as in \eqref{Mell} and \eqref{Mell2},
for every $\ell\ge1$, 
\begin{align}\label{lu2i}
  M_\ell(z) = \Oz{-\ell(\rga+\frac12)+\frac12}
\end{align}
in some fixed \GDD{} (depending on $\ell$)
uniformly for all $\ga\in K$. 
 More generally,
\begin{align}\label{lu2ii}
  M_{\ell_1,\ell_2}(z) = \Oz{-\ell_1\rga_1'-\ell_2\rga_2'+\frac12},
\end{align}
in some fixed \GDD{} (depending on $\ell_1,\ell_2$), 
uniformly for all $\ga_1,\ga_2\in K$. 
\end{lemma}
\begin{proof}
For \eqref{lu2i}, 
the case $\ell=1$ follows from \eqref{sx1}, \eqref{sx3},  and \refL{LU1}.
We then proceed by induction as in the proof of \refL{LC2}. [But the induction is now simpler; it
suffices to note that \eqref{lc4} is at least $\REA+\frac12$.]

The proof of \eqref{lu2ii} is essentially the same, see the proof of
\refT{Tmix}. 
\end{proof}

\begin{lemma}
  \label{LU3}
Assume that $\E\xi^{2+\gd}<\infty$.
  Let $K$ be a compact subset of $\set{\ga:\rga>0}\setminus\set{\frac12}$.
Then, 
for every fixed $r>0$, 
\begin{align}\label{lu3<}
\E\bigsqpar{|X_n(\ga)-n\mu(\ga)|^r} = O\bigpar{n^{r(\rqgay)}}, 
\end{align}
uniformly for all $\ga\in K$ with $\rga<\frac12$, and
\begin{align}\label{lu3>}
\E\bigsqpar{|X_n(\ga)|^r} = O\bigpar{n^{r(\rqgay)}},
\end{align}
uniformly for all $\ga\in K$ with $\rea\ge\frac12$.
\end{lemma}
\begin{proof}
Using the notation \eqref{abba}, \eqref{lu3<} and \eqref{lu3>} can be
combined as
\begin{align}\label{lu3a}
  \E |\dF_\ga(\ctn)|^r  = O\bigpar{n^{r(\rqgay)}},
\end{align}
uniformly in $\ga\in K$.
By \Holder's (or Lyapounov's) inequality, it suffice to prove \eqref{lu3a}
  when $r=2\ell$, an even integer.
In this case, 
we let
$\ga_1=\ga$, $\ga_2=\bga$ and 
$\ell_1=\ell_2=\ell$; 
then \eqref{mell2}--\eqref{Mell2} show that, using also \eqref{pct},
\begin{align}\label{lu3b}
  \E |\dF_\ga(\ctn)|^{2\ell}&
=
  \E \bigsqpar{\dF_\ga(\ctn)^{\ell}\dF_{\bga}(\ctn)^\ell}
=
m_n\xxllo 
= q_n\qw [z^n] M_{\ell,\ell}(z)
\le C n^{3/2} [z^n] M_{\ell,\ell}(z),
\end{align}
and the desired result \eqref{lu3a} (with $r=2\ell$) follows
from \eqref{lu3b} and \eqref{lu2ii} by standard singularity analysis,
see \cite[Proof of Theorem VI.3, p.~390--392]{FS}.
\end{proof}

\begin{lemma}
  \label{LU4}
Assume that $\E\xi^{2+\gd}<\infty$.
  Let $K$ be a compact subset of $\set{\ga:\rga>0}$.
Then, 
for every $r>0$, 
\begin{align}\label{lu4}
\E\bigsqpar{|X_n(\ga)-\E X_n(\ga)|^r} = O\bigpar{n^{r(\rqgay)}},
\end{align}
uniformly for all $\ga\in K$.
\end{lemma}
\begin{proof}
  It suffices to show this for $r\ge1$.
Let
$L^r$ be the Banach space of all complex random variables $X$ defined on
our underlying probability space such that 
\begin{align}\label{normr}
  \norm{X}_r:=\bigpar{\E|X|^r}^{1/r}<\infty.
\end{align}

\pfcase1{$\frac12\notin K$.}
In this case, \refL{LU3} applies and thus \eqref{lu3<} and \eqref{lu3>}
hold, uniformly for $\ga$ in the specified sets.
We may write these as
$\norm{X_n(\ga)-n\mu(\ga)}_r \le C n^{\rqgay}$ 
and $\norm{X_n(\ga)}_r\le C n^{\rqgay}$,
respectively. As is well known, for any (complex) random variable $X$,
\begin{align}\label{lu4a}
\norm{X-\E X}_r\le \norm{X}_r +|\E X| \le 2\norm{X}_r.
\end{align}
Hence we obtain in both cases, and thus uniformly for all $\ga\in K$,
\begin{align}\label{lu4b}
  \norm{X_n(\ga)-\E X_n(\ga)}_r\le C n^{\rqgay},
\end{align}
which is equivalent to \eqref{lu4}.

\pfcase2{$\frac12\in K$.}
Consider first the special case $K_1:=\set{\ga\in \bbC:|\ga-\frac12|\le0.1}$
and let $K_2:=\partial K_1=\set{\ga\in \bbC:|\ga-\frac12|=0.1}$.
Then Case 1 applies to $K_2$.
Moreover,  
recalling the notation \eqref{tY}, 
we can write \eqref{lu4} and \eqref{lu4b}  as
\begin{align}\label{lu4c}
  \norm{\tY_n(\ga)}_r \le C,
\end{align}
where $\tY_n(\ga)=n^{-\ga-\frac12}\bigpar{X_n(\ga)-\E X_n(\ga)}$ is, for each
$n\ge1$,  
an $L^r$-valued analytic function of $\ga$.
[Recall that for a fixed~$n$, there are only finitely many choices for the
tree $\ctn$, and for each choice, \eqref{Xn} is an entire function of~$\ga$.]
The maximum modulus principle holds for Banach space valued analytic functions,
see \eg{} \cite[p.~230]{Dunford-Schwartz},
and thus, using \eqref{lu4c} for $K_2$, 
\begin{align}
  \sup_{\ga\in K_1}\norm{\tY_n(\ga)}_r
=
  \sup_{\ga\in K_2}\norm{\tY_n(\ga)}_r
\le C
.\end{align}
Hence, \eqref{lu4c} holds uniformly for $\ga\in K_1$, and thus so does
\eqref{lu4}. 

For a general compact set $K$, Case 1 applies to 
$\set{\ga\in K:|\ga-\frac12|\ge 0.1}$, which together with the case $K_1$
just proved yields the result \eqref{lu4} uniformly for all $\ga\in K$.
\end{proof}

\begin{proof}[Proof of \refT{T1mom}]
We give the proof for ordinary moments, \ie, \eqref{t1mom}. The other cases
are similar, with mainly notational differences.

Let $\ell\ge1$ and choose $r:=\ell+1$.
  First, consider a fixed $\ga$ with $\rga>0$.
Then
\refL{LU4} shows that
$\E|\tY_n(\ga)|^r=O(1)$, and thus the sequence
$\tY_n(\ga)^\ell$ is \ui, which 
together with  \eqref{t1} implies \eqref{t1mom}.
(See again \cite[Theorems 5.4.2 and 5.5.9]{Gut}.)

To show uniform convergence on compact sets of $\ga$,
consider first a convergent sequence $(\ga_k)$ in $H_+$ with
$\ga_k\to\gaoo\in H_+$ as \ktoo, and a sequence $n_k\to\infty$.
By \refT{T1}, $\tY_n(\ga)\dto \gs\qw\tY(\ga)$ in $\cH(H_+)$, and
by the Skorohod coupling theorem \cite[Theorem~4.30]{Kallenberg},
we may assume that \as{} $\tY_n(\ga)\to \gs\qw\tY(\ga)$ in $\cH(H_+)$, 
\ie, uniformly on compact sets. It then follows that
$\tY_{n_k}(\ga_k)\asto \gs\qw\tY(\gaoo)$ as \ktoo.
Furthermore, 
\refL{LU4} applies to the compact set $\{\alpha_1, \alpha_2, \dots\} \cup \set{\gaoo}$, and
thus \eqref{lu4c} holds and shows that
$\E|\tY_{n_k}(\ga_k)|^r\le C$.
Hence, similarly to the case of a fixed $\ga$,
the sequence $\tY_{n_k}(\ga_k)^\ell$ is \ui, and
\begin{align}\label{ul1}
  \E \tY_{n_k}(\ga_k)^\ell \to \gs^{-\ell}\E \tY(\gaoo)^\ell,
\qquad\text{as }\ktoo
.\end{align}
This holds for any sequence $n_k\to\infty$. 
In particular, we may for each $k$, 
using \eqref{t1mom} which we just have proved for each fixed $\ga$, 
choose $n_k$ so large that 
$\bigabs{ \E \tY_{n_k}(\ga_k)^\ell -\gs^{-\ell} \E \tY(\ga_k)^\ell}<1/k$
for each $k$. Then \eqref{ul1} implies
\begin{align}\label{ul2}
  \E \tY(\ga_k)^\ell \to \E \tY(\gaoo)^\ell
\qquad\text{as }\ktoo
.\end{align}
Since this hold for any sequence $\ga_k\to\gaoo$, \eqref{ul2} shows that
$\E\tY(\ga)^\ell$ is a continuous function of $\ga\in H_+$.

Moreover, \eqref{ul1} and \eqref{ul2} show that for any convergent sequence 
$(\ga_k)$ in $H_+$, and any $n_k\to\infty$,
\begin{align}\label{ul3}
  \E \tY_{n_k}(\ga_k)^\ell- \gs^{-\ell}\E \tY(\ga_k)^\ell \to0.
\end{align}

Let $K\subset H_+$ be compact.
We claim that $  \E \tY_{n}(\ga)^\ell\to \gs^{-\ell}\E \tY(\ga)^\ell $
uniformly for  $\ga\in K$. 
Suppose not. Then there exists $\eps>0$, a subsequence
$n_k\to\infty$ and a sequence $(\ga_k)\in K$ such that
$\bigabs{\E \tY_{n_k}(\ga_k)^\ell-\gs^{-\ell} \E \tY(\ga_k)^\ell }>\eps$ for
every $k$. 
Since $K$ is compact, we may by selecting a subsequence assume that
$\ga_k\to\gaoo$ for some $\gaoo\in K$.
But then \eqref{ul3} holds, which is a contradiction.
This shows the claimed uniform convergence on $K$.

Finally $\E\tY(\ga)^\ell$ is an analytic function of $\ga\in H_+$ since it
is the uniform limit on compact sets of the sequence of analytic functions
$\E \tY_n(\ga)^\ell$.
\end{proof}

\subsection{Final remark}

\begin{remark}
  In this \refS{Smom} we have only considered the case $\rga>0$. It seems likely
  that similar arguments can be used to show moment convergence  in
  \refT{T<0} for $\rga<0$, but we have not pursued this, and we leave it as an
  open problem.
\end{remark}

\appendix

\section{Some examples of $\mu(\ga)$} \label{Amuexamples} 

Although $\mu(\ga)$ easily can be evaluated numerically for a given $\xi$ by
\eqref{mua} or perhaps \eqref{pyr1}, neither formula seems to yield exact
values for a given $\ga$ in any simple form, not even for, e.g., $\ga=-1$.
We give here alternative formulas that can be used to find
exact values  in some important examples when $\ga$ is a negative integer.

Let $U\sim U(0,1)$ and $E:=-\log U \sim\Exp(1)$ be independent of $\cT$.
Define the random variable
\begin{equation}
  \VZ:=U^{1/|\cT|}=e^{-E/|\cT|}.
\end{equation}
Then $0<\VZ<1$, and $\VZ$ has the distribution function,
for $ 0\le x\le 1$,
\begin{equation}\label{g}
  \P(\VZ\le x) = \P\bigpar{U\le x^{|T|}}=\sumn \P(|\cT|=n)x^n=:g(x),
\end{equation}
the \pgf{} of $|\cT|$.
Hence, the density function of $\VZ$ is, for $0\le x<1$,
\begin{equation}\label{g'}
  g'(x)=\sumn n\P(\act=n)x^{n-1}
=\sumn \P(S_n=n-1)x^{n-1}.
\end{equation}

Since $-\log \VZ = E/\act$, we have, for $\rea<\frac12$,
\begin{equation}
  \E(-\log \VZ)^{-\ga} = \E (E/\act)^{-\ga}
=\E E^{-\ga} \E\act^{\ga}
=\gG(1-\ga)\mu(\ga)
\end{equation}
and thus
\begin{equation}\label{E5}
  \mu(\ga)=\frac{1}{\gG(1-\ga)}\E(-\log \VZ)^{-\ga}
=\frac{1}{\gG(1-\ga)}\intoi (-\log x)^{-\ga}\dd g(x).
\end{equation}
This can also be written as
\begin{equation}\label{E6}
  \mu(\ga)
=\frac{1}{\gG(1-\ga)}\intoi (-\log x)^{-\ga}g'(x)\dd x
=\frac{1}{\gG(1-\ga)}\intoo y^{-\ga}g'(e^{-y})e^{-y}\dd y.
\end{equation}

Define the generating function
\begin{equation}
  \label{H}
H(z):=\sumko\mu(-k)z^k,
\end{equation}
which converges absolutely for $|z| < 1$, since $|\mu(-k)| = \mu(-k) = \E| \cT |^{-k} \le~1$.
Then \eqref{E5} yields, for $z\in[0,1)$ say,
using an integration by parts in the final equality,
\begin{align}
  H(z)&=\sumko\frac{1}{k!}\E(-\log V)^{k}z^k
=\E e^{-z\log V}
=\E V^{-z}
\label{Hg1}\\&
= \intoi x^{-z}g'(x) \dd x
= 1+z\intoi x^{-z-1}g(x) \dd x.\label{Hg2}
\end{align}
Note that both integrals in \eqref{Hg2} converge for all $z$ with $\Re z<1$;
hence, \eqref{Hg2} shows that $H(z)$ extends analytically to this halfplane.

We will see below several examples where $H(z)$ can be found explicitly;
then $\mu(-k)$ can be found by extracting Taylor coefficients.
In particular, by \eqref{H} and \eqref{Hg2},
\begin{equation}\label{E7}
  \mu(-1)=H'(0)=\intoi \frac{g(x)}x\dd x,
\end{equation}
which also follows directly from \eqref{mu} and \eqref{g}.

\begin{example}[labelled trees; $\Po(1)$]
Consider uniformly random labelled trees; this is the case $\xi \sim \Po(1)$.
Then $S_n\sim\Po(n)$, and thus \eqref{g} and \eqref{ptk} give
\begin{equation}
  g(x)=\sumn \frac{1}n\P(S_n=n-1)x^n
=\sumn \frac{n^{n-1}e^{-n}}{n\cdot (n-1)!}x^n
=\sumn \frac{n^{n-1}}{n!}(x/e)^n
=T(x/e),
\end{equation}
where~$T$ is the well-known tree function, satisfying 
\begin{equation}
  \label{T}
T(x)e^{-T(x)}=x
, \qquad |x|\le e\qw.
\end{equation}

Since $\VZ$ has the distribution function $g$,
\begin{equation}
  U\eqd g(\VZ)=T(\VZ/e)
\end{equation}
and thus, using \eqref{T},
\begin{equation}
  \VZ/e = T(\VZ/e)e^{-T(\VZ/e)}\eqd U e^{-U}.
\end{equation}
Hence,
\begin{equation}\label{gryta}
  \log \VZ \eqd 1+\log U -U
\end{equation}
and
\begin{equation}
  \E(-\log \VZ)^{-\ga} = \E (U-\log U-1)^{-\ga}
=\intoi (u-\log u-1)^{-\ga}\dd u
.
\end{equation}
Consequently, by \eqref{E5}, for $\rea<\frac12$,
\begin{equation}
  \begin{split}
  \mu(\ga)&
=\frac{1}{\gG(1-\ga)}\intoi (u-\log u-1)^{-\ga}\dd u
\\&
=\frac{1}{\gG(1-\ga)}\intoo \bigpar{e^{-x}-1+x}^{-\ga}e^{-x}\dd x.	
  \end{split}
\end{equation}

In particular, when $\ga$ is a negative integer, $\mu(\ga)$ can be evaluated
as a finite combination of gamma integrals, yielding a rational value.
For example,
$\mu(0)=1$ (as always!), 
$\mu(-1)=1/2$, 
$\mu(-2)=5/12$, 
$\mu(-3)=7/18$.
$\mu(-4)=1631/4320$,
$\mu(-5)=96547/259200$.

In this example, by \eqref{Hg1} and \eqref{gryta},
\begin{align}
  H(z)&=\E e^{-z(1+\log U-U)}=e^{-z}\E\bigpar{U^{-z}e^{zU}}
=e^{-z}\intoi u^{-z}e^{zu}\dd u
\notag\\&
=e^{-z}\gG(1-z)\gamma^*(1-z,-z)
=e^{-z}(-z)^{z-1}\gamma(1-z,-z),
\end{align}
where~$\gamma$ is an incomplete gamma function and $\gamma^*$ is closely
related,
see \cite[\S8.2(i)]{NIST} for both.
\end{example}

\begin{example}[Ordered trees; $\Ge(1/2)$]
For uniformly random ordered trees we have $\xi\sim\Ge(1/2)$, with
$\P(\xi=k)=2^{-k-1}$, $k\ge0$.
Thus $S_n$ has a Negative Binomial distribution,
and, using \eqref{ptk},
\begin{align}
  \P(\act=n)
&=\frac{1}n\P(S_n=n-1)
=\frac{1}n 2^{1-2n}\binom{2n-2}{n-1}
=2^{1-2n}\frac{(2n-2)!}{n!\,(n-1)!}
\notag\\&
=(-1)^{n-1}\binom{\frac12}{n}.	
\end{align}
  Hence, the distribution function $g(x)$ of $V$ is by \eqref{g}
  \begin{equation}
	g(x)
=\sumn (-1)^{n-1}\binom{\frac12}{n} x^n
=1 - (1-x)\qq
  \end{equation}
and the density function is
\begin{equation}\label{g'Ge}
  g'(x)=\frac12(1-x)\qqw.
\end{equation}
Thus, $\VZ$ has a Beta distribution: $\VZ\sim B(1,\frac12)$.

By \eqref{Hg2} and~\eqref{g'Ge},
\begin{align}\label{Hge}
  H(z) 
=\frac12\intoi x^{-z}(1-x)\qqw \ddx x
=\frac12B\bigpar{1-z,\tfrac12}
=\frac{\gG(1-z)\gG(\frac32)}{\gG(\frac32-z)}.
\end{align}
By repeated differentiations we obtain for example, 
assisted by \cite[\S5.15]{NIST} and Maple, and
using again
$\psi(x):=\gG'(x)/\gG(x)$,
\begin{align}
  \mu(-1)&=H'(0)=\psi\bigpar{\tfrac32}-\psi(1)
=2-2\log 2 \doteq 0.6137,
\\
\mu(-2)&=\frac12H''(0)
=\frac12\Bigpar{\bigpar{\psi\bigpar{\tfrac32}-\psi(1)}^2
-\bigpar{\psi'\bigpar{\tfrac32}-\psi'(1)}}
\notag\\&=
2 \log^2 2-4\log 2
-\tfrac16{\pi }^{2}+4 \doteq 0.5434,
\\
\mu(-3)&
=
\tfrac{1}3\xpar{\log2-1}{\pi }^{2}
  -\tfrac{4}3 \log^3 2 +4 \log^2 2 -8\log  2 
-2\zeta(3) + 8 \nonumber\\ 
&\doteq 0.5190,
\\
\mu(-4)&=
  - \tfrac{1}{40} {\pi }^{4} 
+\bigpar{-\tfrac{1}{3} \log^2 2 + \tfrac{2}{3} \log 2 - \tfrac{2}{3}}{\pi }^{2}
 + \tfrac{2}{3} \log^4 2 
\nonumber \\&\qquad{}
- \tfrac{8}{3} \log^3 2 
 +8 \log^2 2 -16 \log 2
 + (4\log 2 -4) \zeta(3) 
 +16 \nonumber \\
 &\doteq 0.5088.
\end{align}
%
%
%
%
\end{example}

\begin{example}[Binary trees; $\Bi(2,\frac12)$]
 Uniformly random binary trees is an example with $\xi\sim\Bi(2,\frac12)$,
Thus $S_n\sim \Bi(2n,\frac12)$ and, using \eqref{ptk},
\begin{align}
 \P(\act=n)
&= \frac{1}n\P(S_n=n-1)
=\frac{1}n2^{-2n}\binom{2n}{n-1}
=2^{-2n}\frac{(2n)!}{n!\,(n+1)!}
\notag\\&
=2(-1)^n\binom{\frac12}{n+1}.	
\end{align}
Hence, 
\begin{equation}
  g(x)=\sumn 2(-x)^n \binom{\frac12}{n+1}
=\frac2{-x}\bigpar{(1-x)\qq-1+\tfrac12x}
=\frac{2-x-2\sqrt{1-x}}{x}
\end{equation}
and \eqref{Hg2} yields, first for $z<-1$ and then for $\Re z<1$ by analytic
continuation,
\begin{align}
  H(z)&=1+z\intoi\Bigpar{ 2x^{-z-2}-x^{-z-1}-2x^{-z-2}(1-x)\qq}\dd x
\notag\\&
=1+\frac{2z}{-z-1}-\frac{z}{-z}-2z\frac{\gG(-z-1)\gG(\frac32)}{\gG(\frac12-z)}
\notag\\&
= \frac{2}{1+z}
-\gG(\tfrac12)\frac{\gG(1-z)}{(1+z)\gG(\frac12-z)}.\label{HBi}
\end{align}
Taking Taylor coefficients at~$0$ yields, for example, 
again using \cite[\S5.15]{NIST} and Maple, 
\begin{align}
  \mu(-1)&=H'(0)=-1+\psi(1)-\psi(\tfrac12)=2\log 2 -1 
\doteq 0.3863,
\\
\mu(-2)&=
\tfrac{1}6\pi^2-2\log^22-2\log2+1
\doteq0.2977.
\end{align}
\end{example}

\begin{example}[Full binary trees; $2\Bi(1,\frac12)$]
 Uniformly random full binary trees is an example with 
$\xi/2\sim\Bi(1,\frac12)$, \ie, $\P(\xi=0)=\P(\xi=2)=\frac12$.
Thus $S_n/2\sim \Bi(n,\frac12)$ and, using \eqref{ptk},
if $n=2m+1$ is odd,
\begin{equation}
  \begin{split}
 \P(\act=n)
&= \frac{1}n\P(S_n=n-1)
=\frac{1}n2^{-n}\binom{n}{m}
=2^{-2m-1}\frac{(2m)!}{m!\,(m+1)!}
\\&
=(-1)^m\binom{\frac12}{m+1}.	
  \end{split}
\end{equation}
Hence,
\begin{equation}
  g(x)=\sumko (-1)^m x^{2m+1} \binom{\frac12}{m+1}
=\frac{1-\sqrt{1-x^2}}{x}
\end{equation}
and \eqref{Hg2} yields, 
similarly to \eqref{HBi}, 
omitting some details,
\begin{align}
  H(z)&=1+z\intoi\!x^{-z-2}\bigpar{1-\sqrt{1-x^2}}\dd x
=\frac{1}{1+z}+\frac{\gG(-\frac{1+z}{2})\gG(\frac{3}{2})}{\gG(-\frac{z}{2})}.
\end{align}
This yields, for example,
\begin{align}
  \mu(-1) &
=\frac{\pi}2 -1
\doteq0.5708,
\\
\mu(-2)&=1-\tfrac12(1-\log2)\pi
\doteq0.5180.
\end{align}
\end{example}

\section{Polylogarithms}\label{Apoly}

As said in \eqref{Li},
the polylogarithm function is defined, for $\ga\in\bbC$, by
\begin{align}\label{pLi}
  \Li_\ga(z):=\sumn n^{-\ga}z^n,
\qquad |z|<1;
\end{align}
the function is then extended analytically to $z\in\bbC\setminus\ooo$,
for example by the integral formula \cite[(VI.48)]{FS}.
As a bivariate function, $\Li_\ga(z)$ is analytic in
both variables $(\ga,z)\in\bbC\times(\bbC\setminus\ooo)$. 

Let $U:=\set{z\in\bbC\setminus(-\infty,0]:|\log z|<2\pi}$
(where $\log z$  denotes the principal value),
and note that $U$
is a neighborhood of 1. 
In particular, $U$ contains, for example, 
the disc $U_1:=\set{z:|z-1|<\frac12}$.
If $\ga\notin\set{1,2,\dots}$, $z\notin [1,\infty)$, and 
furthermore $z\in U':=U\setminus\ooo$,
then,
see \cite[25.12.2]{NIST}
and \cite[(1.11.8)]{ErdelyiI},
\begin{align}\label{pla}
  \Li_\ga(z) 
= \gG(1-\ga)\xpar{-\log z}^{\ga-1}+\sumno \zeta(\ga-n)\frac{(\log z)^n}{n!}.
\end{align}
We denote the infinite sum in \eqref{pla} by $h_\ga(z)$, and
note that it
converges absolutely for $z\in U$, and thus is analytic there, 
since the reflection formula 
for the Riemann zeta function \cite[25.4.2]{NIST} easily implies
\begin{align}\label{plb}
\frac{ |\zeta(\ga-n)|}{n!}
=O\Bigpar{\frac{(2\pi)^{-n}\gG(n+1-\ga)\zeta(n+1-\ga)}{n!}}
=O\Bigpar{n^{-\Re \ga}(2\pi)^{-n}}.
\end{align}
for each fixed complex $\ga$ and $n\ge\Re\ga+1$.

Moreover, we define  the analytic function
\begin{align}\label{plc}
  G(z):=\frac{-\log z}{1-z},
\qquad z\in U_1,
\end{align}
where by continuity $G(1)=1$. Since $G(z)\neq0$ in $U_1$,
\begin{align}\label{pld}
  g(z):=\log(G(z)),
\qquad z\in U_1,
\end{align}
also defines an analytic function in $U_1$, with $g(1)=0$.
Then, for $z\in U_1':=U_1\setminus[1,\infty)$,
\begin{align}\label{ple}
  (-\log z)^{\ga-1}
=\bigpar{(1-z)G(z)}^{\ga-1}
=\bigpar{(1-z)e^{g(z)}}^{\ga-1}
=(1-z)^{\ga-1}e^{(\ga-1)g(z)}.
\end{align}

Consequently, \eqref{pla} yields
\begin{align}\label{plf}
  \Li_\ga(z) =\gG(1-\ga)(1-z)^{\ga-1}e^{(\ga-1)g(z)}+h_\ga(z),
\qquad z\in U_1'.
\end{align}

The functions $e^{(\ga-1)g(z)}$ and $h_\ga(z)$ are analytic functions of
$z\in U_1$, and can thus be expanded as Taylor series in $1-z$.
Hence, \eqref{plf} yields, for $z\in U_1'$, an absolutely convergent expansion
\begin{align}\label{plg}
  \Li_{\ga}(z)=\sumjo a_j(\ga)(1-z)^{\ga-1+j}
+\sumko b_k(\ga)(1-z)^k
\end{align}
for some coefficients $a_j(\ga)$ and $b_k(\ga)$.
This is the asymptotic expansion given in 
\citet{Flajolet1999} and 
\cite[Theorem VI.7]{FS}; 
we see now that the
expansion actually converges for $z\in U_1'$.

The coefficients $a_j(\ga)$ and $b_k(\ga)$
can be found from the formulas above by repeated
differentiations at $z=1$, or 
(as in \cite{Flajolet1999} and \cite{FS})
by substitution 
in \eqref{plf} of
\begin{align}\label{plh}
  \log z =\log\bigpar{1-(1-z)}
=-\sumk \frac{(1-z)^k}{k}
=-(1-z)\sumko \frac{(1-z)^k}{k+1}
\end{align}
and its consequence
\begin{align}\label{pli}
g(z)=\log\Bigsqpar{1+\sumk\frac{(1-z)^k}{k+1}}
=\summ\frac{(-1)^{m-1}}{m}\Bigsqpar{\sumk\frac{(1-z)^k}{k+1}}^m,
\end{align}
followed by rearrangements into single power series.
Note that $a_j(\ga)$ and $b_k(\ga)$ are analytic functions of
$\ga\in\bbC\setminus\set{1,2,\dots}$. 

In particular, $a_0(\ga)=\gG(1-\ga)$, and thus by keeping only the first term
in the first sum in \eqref{plg}, we obtain
\eqref{li}.

\begin{proof}[Proof of \refL{LUL1}]
It is easily checked that the estimate \eqref{plb} holds uniformly for
$\ga\in K$ and large enough $n$.
Hence, uniformly for $\ga\in K$ and $z\in U_1$, 
\begin{align}
 | h_\ga(z)| = O(1),
\end{align}
Similarly, since $g(1)=0$, we have $g(z)=O(|1-z|)$ in $U_1$, and
\begin{align}
  e^{(\ga-1)g(z)}= 1+\Oz{},
\end{align}
again uniformly for $\ga\in K$ and $z\in U_1$.
Hence, for $z\in U_1'$, \eqref{lul1} follows from \eqref{plf}, 
with the $O$ term uniform for $\ga\in K$.
The case $z\in\gD\setminus U_1'$ is trivial, since $|1-z|$ is bounded above
and below in that set, and $\Li_\ga(z)$ is uniformly bounded 
in the compact set $\overline{\gD}\setminus U_1$ by continuity.
\end{proof}

In the same way we see that we may expand the two sums in \eqref{plg} to any
number of finite terms, and the resulting expansion will have error terms
that are uniform in $\ga\in K$, for any compact
$K\subset\bbC\setminus\set{1,2,\dots}$.

\section{The limit as $\ga\to0$}\label{A0}

We show here the claim in \refR{R0} about limits (in distribution) of
$Y(\ga)$ as $\ga\to0$ (with $\rga>0$; recall that $Y(\ga)$ is defined only for such $\ga$).
It turns out that the limit depends on how $\ga$
appoaches 0. We consider for simplicity only the case when $\ga$ approaches
on a straight line, i.e., with 
constant argument (necessarily with $|\arg\ga|<\pi/2$).
In this case, $\ga\qw Y(\ga)$ has a complex normal limiting distribution,
but the limit depends on $\arg\ga$.

\begin{theorem}\label{TD}
Let $\ga=re^{\ii\gth}$ with $|\gth|<\pi/2$, and let $r\to0$ with $\gth$
fixed.
Then
\begin{align}\label{td1}
  \ga\qqw Y(\ga) \dto \zeta,
\end{align}
where $\zeta$ is a centered complex normal variable, which is characterized by
the covariance matrix
\begin{align}\label{td2}
\Cov
\begin{pmatrix}\Re\zeta\\\Im\zeta\end{pmatrix}
=
\frac{1-\log 2}{\cos\gth} \begin{pmatrix} 1+\cos\gth & 0 \\ 
0& 1-\cos\gth\end{pmatrix}
.\end{align}
In other words, $\Re\zeta$ and $\Im\zeta$ are independent centered normal
variables with respective variances $(1-\log2)[(1/\cos\gth)\pm1]$;
equivalently, with
\begin{align}\label{td3}
  \E \zeta^2= 2(1-\log2)
\qquad\text{and}\qquad
  \E |\zeta|^2= 2(1-\log2)/\cos\gth.
\end{align}
\end{theorem}

The case $\ga$ real, i.e., $\gth=0$, was noted
in \cite[Remark 3.6(e)]{FillK04}. As stated in \eqref{r0}, then
$\zeta$
is a real normal variable $N\bigpar{0,2(1-\log2)}$.

We prove \refT{TD} by the method of moments, using \refT{Tmix}.
We procced via a series of lemmas that are stated for somewhat more general
situations.  

\begin{lemma}\label{LOA}
  As $\ga,\gb\to0$, with $\rga,\rgb>0$, we have
  \begin{align}\label{loa}
\E Y(\ga)
&=
\kk_1(\ga)=
\frac{\sqrt{2\pi}} {\gG(\ga)}\kkk_1(\ga)  
\sim -\sqrt{2\pi}\,\ga,
\\\label{lob}
\E\bigsqpar{Y(\ga)Y(\gb)}
&
=\kk_{1,1}(\ga,\gb)
\sim\sqrt2 \kkk_{1,1}(\ga,\gb)
\sim 4(1-\log 2)\frac{\ga\gb}{\ga+\gb}
.  \end{align}
\end{lemma}

\begin{proof}
All asymptotic notions in the proof are as $\ga,\gb\to0$.
We assume throughout  that $|\ga|$ and $|\gb|$ are small.

Recall again the standard notation
\begin{align}\label{psi}
\psi(x):=
\ddxx \log \gG(x)
=\frac{\gG'(x)}{\gG(x)}.
\end{align}

First, by \eqref{kkk1} and \eqref{psi},
  \begin{align}\label{loc}
    \kkk_1(\ga) 
= \frac{\gG(\ga-\frac12)}{2\sqrt\pi}
= -\frac{\gG(-\frac12+\ga)}{\gG(-\frac12)}
= -\bigpar{1+\psi(-\tfrac12)\ga+O(|\ga|^2)}
.  \end{align}
In particular, $\kkk_1(\ga)\sim -1$ and thus \eqref{loa} follows 
by~\eqref{farf} [or~\eqref{kk1}].

For the second moment \eqref{lob}, we first note that by \eqref{tmix} and
\eqref{varin}, 
\begin{align}\label{lof}
  \E\bigsqpar{Y(\ga)Y(\gb)}=\kk_{1,1}(\ga,\gb)
=\frac{\sqrt{2\pi}}{\gG(\frac12+\ga+\gb)}\kkk_{1,1}(\ga,\gb)
\sim \sqrt{2} \kkk_{1,1}(\ga,\gb).
\end{align}

Finally, by \eqref{kkkll}, as in \eqref{kkk11},
\begin{align}\label{kkk110}
 \sqrt2 \kkk_{1,1}(\ga,\gb)
&= \kkk_1(\ga)\kkk_1(\gb)
+ \frac{\gG(\ga+\gb)}{\gG(\gb)}\kkk_1(\gb)
+ \frac{\gG(\ga+\gb)}{\gG(\ga)}\kkk_1(\ga)
\notag\\&
=
 \kkk_1(\ga)\kkk_1(\gb)\Bigsqpar{
1
+\frac{\gG(\ga+\gb)}{\gG(\gb)\kkk_1(\ga)}
+\frac{\gG(\ga+\gb)}{\gG(\ga)\kkk_1(\gb)}
}
.\end{align}
We have, using \eqref{psi},
\begin{align}\label{lod}
  \frac{\gG(\ga+\gb)}{\gG(\ga)}
&=
\frac{\ga}{\ga+\gb}\cdot  \frac{\gG(1+\ga+\gb)}{\gG(1+\ga)}
=
\frac{\ga}{\ga+\gb}\Bigsqpar{1+\psi(1+\ga)\gb + O(|\gb|^2)}
\notag\\&
=\frac{\ga}{\ga+\gb}\Bigsqpar{1+\psi(1)\gb + O(|\ga\gb|+|\gb|^2)},
\end{align}
which together with \eqref{loc} yields
\begin{align}\label{loe}
  \frac{\gG(\ga+\gb)}{\gG(\ga)\kkk_1(\gb)}
&
=-\frac{\ga}{\ga+\gb}
\Bigsqpar{1+\bigpar{\psi(1)-\psi(-\tfrac12)}\gb + O(|\ga\gb|+|\gb|^2)}.
\end{align}
Using \eqref{loe}, and the same with $\ga$ and $\gb$ interchanged, in
\eqref{kkk110} we obtain, recalling $\kkk_1(\ga)\sim\kkk_1(\gb)\sim-1$,
\begin{align}\label{kkk1100}
 \sqrt2 \kkk_{1,1}(\ga,\gb)
& 
\sim
1
+\frac{\gG(\ga+\gb)}{\gG(\gb)\kkk_1(\ga)}
+\frac{\gG(\ga+\gb)}{\gG(\ga)\kkk_1(\gb)}
\notag\\&
=-2\frac{\ga\gb}{\ga+\gb}\Bigsqpar{\psi(1)-\psi(-\tfrac12)
+ O\bigpar{|\ga|+|\gb|}}
.\end{align}
The result \eqref{lob} now follows because $\psi(1)=-\gam$
and 
\begin{align}\label{psi-1/2}
 \psi(-\tfrac12)=\psi(\tfrac12)+2=-\gam-2\log 2+2, 
\end{align}
see \cite[5.4.12--13 and 5.5.2]{NIST}.
\end{proof}

\begin{lemma}\label{LOB}
Let $\ga=re^{\ii\gthx}$ and $\gb=re^{\ii\gthy}$ with
$\gthx,\gthy\in(-\frac{\pi}2,\frac{\pi}2)$,
and let $r\to0$ with $\gthx,\gthy$ fixed.
Then,
for every fixed $\ell_1,\ell_2\ge0$ with $\ell_1+\ell_2\ge2$,
\begin{align}\label{lobb}
r^{-(\ell_1+\ell_2)/2} \kkk_{\ell_1,\ell_2} (\ga,\gb) \to \vpx_{\ell_1,\ell_2},
\end{align}
where $\vpx_{\ell_1,\ell_2}$ is given recursively by
\begin{align}\label{lob0}
  \vpx_{0,0}&=0,
\\\label{lob1}
\vpx_{\ell_1,\ell_2}&=0,
\qquad\text{when $\ell_1+\ell_2=1$},
\\\label{lob20}
\vpx_{2,0}&=\sqrt2(1-\log2){e^{\ii\gth_1}},
\\\label{lob02}
\vpx_{0,2}&=\sqrt2(1-\log2){e^{\ii\gth_2}},
\\\label{lob11}
\vpx_{1,1}&=2\sqrt2(1-\log2)\frac{e^{\ii(\gth_1+\gth_2)}}{e^{\ii\gth_1}+e^{\ii\gth_2}},
\\\label{lob3}
\vpx_{\ell_1,\ell_2}&=
2^{-3/2} \sum_{j_1,j_2} 
\binom{\ell_1}{j_1}\binom{\ell_2}{j_2}\vpx_{j_1,j_2}\vpx_{\ell_1-j_1,\ell_2-j_2},
\qquad\text{when $\ell_1+\ell_2\ge3$}
.
\intertext{Moreover,}
\label{lobu}
\vpx_{\ell_1,\ell_2}&=0,
\qquad\text{when $\ell_1+\ell_2$ is odd}
.\end{align}
\end{lemma}

\begin{proof}
 We define for convenience $\vpx_{\ell_1,\ell_2}:=0$ for
 $\ell_1+\ell_2\le1$, and note 
  that then \eqref{lob0}--\eqref{lob1} hold, but not \eqref{lobb}.

For $\ell_1+\ell_2=2$, \eqref{lob20}--\eqref{lob11} hold by \refL{LOA}.

It remains to treat the case $\ell_1+\ell_2\ge3$, where we use induction on
$\ell_1+\ell_2$. 
We use \eqref{kkkll}. In the double sum there, the two terms with
$(j_1,j_2)=(1,0)$ and $(j_1, j_2) = (\ell_1-1,\ell_2)$ are equal, and together,
using \eqref{loc}, 
sum to
\begin{align}\label{foa}
  2\qqw \ell_1\kkk_{1}(\ga)\kkk_{\ell_1-1,\ell_2}(\ga,\gb)
=
-  2\qqw \ell_1\bigsqpar{1+O(r)}\kkk_{\ell_1-1,\ell_2}(\ga,\gb).
\end{align}
On the other hand, 
the second of the three terms on the right in~\eqref{kkkll}
is
\begin{align}\label{fob}
  2\qqw &\ell_1
\frac{\gG((\ell_1+\ell_2-2)/2+\ell_1\ga+\ell_2\gb)}
{\gG((\ell_1+\ell_2-2)/2+(\ell_1-1)\ga+\ell_2\gb)}
\kkk_{\ell_1-1,\ell_2}(\ga,\gb)
\notag\\&=
  2\qqw\ell_1
\bigsqpar{1+O(r)}
\kkk_{\ell_1-1,\ell_2}(\ga,\gb)
.\end{align}
Hence the main terms of the contributions \eqref{foa} and \eqref{fob} cancel, 
and 
together, using the induction hypothesis, \eqref{foa} and~\eqref{fob} sum to
\begin{align}\label{foc}
  O(r)\cdot \kkk_{\ell_1-1,\ell_2}(\ga,\gb)
=O\bigpar{r^{1+(\ell_1-1+\ell_2)/2}}
=o\bigpar{r^{(\ell_1+\ell_2)/2}}
.\end{align}
Similarly, the terms in the double sum with
$(j_1,j_2)=(0,1)$ and $(\ell_1,\ell_2-1)$ together
cancel the last term in \eqref{kkkll} up to another error 
$o\bigpar{r^{(\ell_1+\ell_2)/2}}$.

This shows that \eqref{kkkll} yields
\begin{align}\label{kkkllq}
  \kkk_{\ell_1,\ell_2}(\ga,\gb)&
=
2^{-3/2} \sum_{2\le j_1+j_2\le \ell_1+\ell_2-2} 
\binom{\ell_1}{j_1}\binom{\ell_2}{j_2}\kkk_{j_1,j_2}\kkk_{\ell_1-j_1,\ell_2-j_2}
+o\bigpar{r^{(\ell_1+\ell_2)/2}},
\end{align}
and \eqref{lobb} together with \eqref{lob3} follows by the induction
hypothesis,
noting that the terms in \eqref{lob3} with $j_1+j_2\le1$ or $j_1+j_2\ge
\ell_1+\ell_2-1$ vanish by \eqref{lob0}--\eqref{lob1}.

The conclusion~\eqref{lobu} follows from~\eqref{lob1} and~\eqref{lob3} by induction, since each of the terms in~\eqref{lob3} vanishes.
\end{proof}

Recall that if $\ell=2k$ is an even integer, then
\begin{align}\label{!!}
  (\ell-1)!!=(2k-1)!!:=1\cdot3\cdot\dotsm\cdot(2k-1)=\frac{(2k)!}{2^kk!}
=2^k\frac{\gG(k+\frac12)}{\gG(\frac12)}
.\end{align}

\begin{lemma}
  \label{LOC}
Let $\ga=re^{\ii\gthx}$ and $\gb=re^{\ii\gthy}$ with
$\gthx,\gthy\in(-\frac{\pi}2,\frac{\pi}2)$,
and let $r\to0$ with $\gthx,\gthy$ fixed.
Let $t$ and $u$ be fixed complex numbers.
Then,
for every  $\ell\ge1$,
\begin{align}\label{locc}
r^{-\ell/2} 
\E\bigpar{t Y(\ga)+u Y(\gb)}^\ell
\to
  \begin{cases}
    0, & \text{$\ell$ is odd},\\
(\ell-1)!!\, \gS^{\ell/2}, & \text{$\ell$ is even},
  \end{cases}
\end{align}
where 
\begin{align}\label{loca}
\gS & 
=2(1-\log2)\Bigpar{t^2 e^{\ii\gthx} + u^2e^{\ii\gthy}
 +4tu \frac{e^{\ii(\gth_1+\gth_2)}}{e^{\ii\gth_1}+e^{\ii\gth_2}}
}
.\end{align}
\end{lemma}

\begin{remark}\label{RC}
If $\gS\ge0$, then the limits in \eqref{locc} are the moments of a normal
distribution $N(0,\gS)$. 
Hence, if $tY(\ga)+uY(\gb)$ is a real random variable (and $\gS\neq0$), then
\refL{LOC} implies asymptotic normality by the method of moments.
However, in general, $tY(\ga)+uY(\gb)$ is a complex random variable
and $\gS$ is complex. 
Nevertheless, the \rhs{} can be interpreted as the moments of a complex normal
random variable, since the relation
$\E\zeta^{2\ell}=(2\ell-1)!!\,(\E\zeta^2)^\ell$ holds for arbitrary 
centered complex normal variables, 
see \eg{}
\cite[Theorem 1.28 and Section I.4]{SJIII}.
\end{remark}

\begin{proof}
\refT{Tmix} and \refL{LOB} imply that, if $\ell_1+\ell_2=\ell\ge2$, then
\begin{align}\label{ick1}
r^{-\ell/2} \E\bigsqpar{Y(\ga)^{\ell_1}Y(\gb)^{\ell_2}}
=r^{-\ell/2} \kk_{\ell_1,\ell_2}(\ga,\gb)
\to \frac{\sqrt{2\pi}}{\gG((\ell-1)/2)}\vpx_{\ell_1,\ell_2}.
\end{align}
For $\ell_1+\ell_2=1$, \eqref{lobb} does not hold, but a direct appeal to
\eqref{kk1} yields
\begin{align}\label{ick2}
 \E Y(\ga)=\frac{\gG(\ga-\frac12)}{\sqrt2\gG(\ga)}
= O(|\ga|)=O\bigpar{r}, 
\end{align}
and similarly $\E Y(\gb)=O(r)$; 
hence, \eqref{ick1} holds in the case $\ell=1$ too, with the limit 0.
(Recall that $1/\gG(0)=0$.)

By the binomial formula, 
  \begin{align}\label{ick3}
    \E(tY(\ga)+uY(\gb))^\ell
=\sum_{\ell_1+\ell_2=\ell}\binom{\ell}{\ell_1}t^{\ell_1}u^{\ell_2}
\kk_{\ell_1,\ell_2}(\ga,\gb) 
  \end{align}
which together with \eqref{ick1} yields, for every $\ell\ge1$,
\begin{align}\label{ick4}
  r^{-\ell/2}\E(tY(\ga)+uY(\gb))^\ell
&\to 
\sum_{\ell_1+\ell_2=\ell}\binom{\ell}{\ell_1}t^{\ell_1}u^{\ell_2}
\frac{\sqrt{2\pi}}{\gG((\ell-1)/2)}\vpx_{\ell_1,\ell_2}
\notag\\&
=
\frac{\sqrt{2\pi}}{\gG((\ell-1)/2)}\xxi_{\ell}(t,u)
,\end{align}
where we define
\begin{align}\label{ick5}
\xxi_\ell(t,u)
:=
\sum_{\ell_1+\ell_2=\ell}\binom{\ell}{\ell_1}t^{\ell_1}u^{\ell_2}
\vpx_{\ell_1,\ell_2}
.\end{align}

We have $\xxi_\ell(t,u)=0$ when $\ell$ is odd or $\ell=0$, by \eqref{ick5} 
together with
\eqref{lobu} and \eqref{lob0}. Hence \eqref{locc} for odd $\ell$ follows
from \eqref{ick4}.

Moreover,
if $\ell\ge3$, then \eqref{ick5} (thrice) and the recursion \eqref{lob3} imply
\begin{align}\label{ick6}
2^{3/2}
\xxi_\ell(t,u)
&=
\sum_{\ell_1+\ell_2=\ell}\,\sum_{j_1,j_2}
\binom{\ell}{\ell_1}\binom{\ell_1}{j_1}\binom{\ell_2}{j_2}
t^{\ell_1}u^{\ell_2}
\vpx_{j_1,j_2}
\vpx_{\ell_1-j_1,\ell_2-j_2}
\notag\\&
=\sum_j \sum_{j_1+j_2=j}\,\sum_{\ell_1+\ell_2=\ell}
\binom{\ell}{j}\binom{j}{j_1}\binom{\ell-j}{\ell_1-j_1}
t^{\ell_1}u^{\ell_2}
\vpx_{j_1,j_2}
\vpx_{\ell_1-j_1,\ell_2-j_2}
\notag\\&
=\sum_j\binom\ell{j}\xxi_j(t,u)\xxi_{\ell-j}(t,u)
.\end{align}
Since $\xxi_\ell(t,u)=0$ when $\ell$ is odd or $\ell=0$,
\eqref{ick6} yields
\begin{align}\label{ick6b}
2^{3/2}
\xxi_{2\ell}(t,u)
=\sum_{j=1}^{\ell-1}\binom{2\ell}{2j}\xxi_{2j}(t,u)\xxi_{2(\ell-j)}(t,u)
.\end{align}
The recursion \eqref{ick6b} is easily
solved, by defining 
\begin{align}\label{ick7}
  d_\ell&:=2^{-3/2}\xxi_{2\ell}(t,u)/(2\ell)!,
\\
e_\ell&:=d_1^{-\ell}d_\ell.  \label{icke}
\end{align}
Then \eqref{ick6b} yields
\begin{align}\label{lriderec}
  d_\ell:=\sum_{j=1}^{\ell-1}d_jd_{\ell-j}&&&\text{and}&&
  e_\ell:=\sum_{j=1}^{\ell-1}e_je_{\ell-j},&&
\ell\ge2.
\end{align}
This is a version of the Catalan recursion, and since $e_1=1$, it is solved by
\begin{align}\label{lric}
  e_\ell = C_{\ell-1}=\frac{(2\ell-2)!}{(\ell-1)!\,\ell!},
\qquad \ell\ge1;
\end{align}
%
and thus, by \eqref{ick7} and \eqref{icke},
\begin{align}\label{ick8}
  \xxi_{2\ell}(t,u)=2^{3/2}\frac{(2\ell)!\,(2\ell-2)!}{(\ell-1)!\,\ell!} d_1^\ell.
\end{align}
Hence, \eqref{ick4} yields
\begin{align}\label{mars}
    r^{-\ell}\E(tY(\ga)+uY(\gb))^{2\ell}
&\to 
\frac{\sqrt{2\pi}}{\gG(\ell-\frac12)}\xxi_{2\ell}(t,u)
=
\frac{4\sqrt{\pi}}{\gG(\ell-\frac12)}
\frac{(2\ell)!\,(2\ell-2)!}{(\ell-1)!\,\ell!} d_1^\ell
\notag\\&=
2^{2\ell}\frac{(2\ell)!}{\ell!}d_1^\ell
=(2\ell-1)!!\, (8d_1)^\ell
.\end{align}
This proves \eqref{locc} for even $\ell$ 
with, recalling \eqref{ick7} and \eqref{ick5},
\begin{align}
  \gS:=8d_1=\sqrt2\xxi_2(t,u)
=\sqrt2\bigpar{t^2\vpx_{2,0}+u^2\vpx_{0,2}+2tu\vpx_{1,1}}.
\end{align}
Finally, \eqref{loca} follows from \eqref{lob20}--\eqref{lob11}.
\end{proof}

\begin{proof}  [Proof of \refT{TD}]
We apply \refL{LOC} with  $\gb:=\bga$ and thus $\gthx=\gth$ and $\gthy=-\gth$.
Let $t\in\bbC$ and take
$u:=\bar t$. Then
$tY(\ga)+uY(\gb) = 2\Re\bigpar{tY(\ga)}$ is a real random variable, and thus
\eqref{locc} shows by  the method of moments that 
\begin{align}\label{luc}
 2 r\qqw \Re\bigpar{tY(\ga)} \dto N(0,\gS),
\end{align}
with $\gS=\gS(t)$ (now real) given by \eqref{loca}.
Since $t\in\bbC$ is arbitrary and
$\Re\bigpar{tY(\ga)}$ can be regarded as the (real) scalar
product of $\bar t$ and $Y(\ga)$ if we identify $\bbC$ and $\bbR^2$,
\eqref{luc} and the Cram\'er--Wold device
show that
\begin{align}\label{luca}
 2 r\qqw Y(\ga) \dto \zeta',
\end{align}
for some centered complex normal variable $\zeta'$.
Consequently,
\begin{align}\label{lucb}
 \ga\qqw Y(\ga)
=e^{-\ii\gth/2} r\qqw Y(\ga) \dto \zeta
:=\frac{e^{-\ii\gth/2}}2\zeta',
\end{align}
which proves \eqref{td1}.
Moreover, the argument above shows that
\eqref{luca} holds with all moments 
(including mixed moments with the complex conjugate),
and thus so does
\eqref{lucb}.
Taking $t=1$, $u=0$ and $\ell=2$ in \eqref{locc}--\eqref{loca} yields
\begin{align}
  \E\bigpar{\ga\qqw Y(\ga)}^2
=e^{-\ii\gth}r\qw  \E{Y(\ga)}^2 \to 2(1-\log2).
\end{align}
Similarly, by extracting the $tu$ terms in \eqref{locc} and \eqref{loca},
\begin{align}
  \E\bigabs{\ga\qqw Y(\ga)}^2
=r\qw  \E\bigabs{Y(\ga)}^2 \to 2(1-\log2)\frac{2}{e^{\ii\gth}+e^{-\ii\gth}}
= \frac{2(1-\log2)}{\cos\gth}.
\end{align}
This shows \eqref{td3}, and \eqref{td2} follows by elementary calculations.
\end{proof}

\section{The limit towards the imaginary axis}\label{Ait}

Let $\ga=a+\ii b \to\ii t$ in the right half-plane, \ie, with $a=\rga>0$.
The case $t=0$ is treated in \refApp{A0}; recall that then, 
if say $\ga$ is real for simplicity, $Y(\ga)\pto0$
and that  $a\qqw Y(\ga)$ converges in distribution
to a normal limit; see also \refR{R0} and \cite[Remark 3.6(e)]{FillK04}.

Assume in the sequel $t\neq0$.  In this case, we have instead
$|Y(\ga)|\pto\infty$, and we obtain a complex normal limit by the following
normalization. (Note that, unlike the case $t=0$ in \refT{TD}, here $\ga$ can
approach its limit $\ii t$ in any way, as long as $\rga>0$.)
\begin{theorem}\label{TRI}
Let $a\downto0$ and $b\to t\neq0$.
Then
  \begin{align}\label{ri}
    a\qq Y(a+\ii b) \dto \zeta,
  \end{align}
where $\zeta$ is a symmetric complex normal variable with 
\begin{align}\label{ri2}
    \E|\zeta|^2
=\frac{1}{2\sqrt\pi} \Re\frac{\gG(\ii t-\frac12)}{\gG(\ii t-1)}>0.
\end{align}
\end{theorem}
That $\zeta$ is symmetric complex normal
means that $\zeta \eqd \go\zeta$ for every complex
constant $\go$ with $|\go|=1$; equivalently, $\E\zeta=0$ and the real and
imaginary parts are independent and have the same variance.
(See \eg{} \cite[Proposition 1.31]{SJIII}.)


\begin{proof}
We use the method of moments, and argue similarly as for the related \refT{TD}.
Take $\ga_1:=a+\ii b$ and $\ga_2:=\overline{\ga_1}=a-\ii b$ in \refT{Tmix}.
We claim that, for any $\ell_1,\ell_2\ge0$,
\begin{align}\label{lri}
  a^{(\ell_1+\ell_2)/2}\kkk_{\ell_1,\ell_2}(\ga_1,\ga_2) \to \rho_{\ell_1,\ell_2},
\end{align}
where
\begin{align}\label{lri0}
  \rho_{\ell_1,\ell_2}&=0   \qquad\text{if $\ell_1\neq\ell_2$},
\\\label{lri1}
\rho_{1,1}&=\frac{1}{\sqrt{8\pi}}\Re\frac{\gG(\ii t-\frac12)}{\gG(\ii t-1)},
\\\label{lri2}
\rho_{\ell,\ell}&=2^{-3/2}\sum_{j=1}^{\ell-1}\binom{l}{j}^2\rho_{j,j}\rho_{\ell-j,\ell-j},
\qquad l\ge2.
\end{align}
We prove this using induction on $\ell_1+\ell_2$.
First, if $\ell_1+\ell_2=1$, so $(\ell_1,\ell_2)=(1,0)$ or $(0,1)$, then
\eqref{kkk1} shows that $\kkk_{\ell_1,\ell_2}$ is bounded (and converges) as
$\ga\to\ii t$, so \eqref{lri} holds with $\rho_{\ell_1,\ell_2}=0$ as stated
in \eqref{lri0}.

If $\ell_1+\ell_2\ge2$, we use \eqref{kkkll}. We have
\begin{align}
\ell_1\ga_1'+\ell_2\ga'_2-1 
\to \ell_1\bigpar{\ii t+\tfrac12} + \ell_2\bigpar{-\ii t+\tfrac12}-1
=\bigpar{\ell_1+\ell_2}/2-1 + \bigpar{\ell_1-\ell_2}\ii t.
\end{align}
If $\ell_1+\ell_2\ge3$, or if $\ell_1\neq \ell_2$, the limit is not a pole
of $\gG(z)$, and thus the factor $\gG\bigpar{\ell_1\ga'_1+\ell_2\ga'_2-1}=O(1)$;
hence, \eqref{kkkll} together with the induction hypothesis
yields  \eqref{lri} with \eqref{lri0} and \eqref{lri2}.

In the remaining case $\ell_1=\ell_2=1$,
$\gG\bigpar{\ell_1\ga'_1+\ell_2\ga'_2-1}=\gG(2a)\sim (2a)\qw$, and
\eqref{kkkll} yields, using \eqref{kkk1},
\begin{align}
  a\kkk_{1,1} 
= 2\qqw \frac{1}{2\gG(-\ii t)}\kkk_{0,1} + 2\qqw \frac{1}{2\gG(\ii t)}\kkk_{1,0} 
+o(1)
\to 2\qqw \Re \frac{\gG\bigpar{\ii t-\frac12}}{2\sqrt\pi\gG(\ii t)},
\end{align}
which verifies \eqref{lri} with \eqref{lri1}.

This proves \eqref{lri}--\eqref{lri2}. 
The recursion \eqref{lri2} is similar to \eqref{ick6b} and can be solved in
the same way.
We now define, 
instead of \eqref{ick7},
\begin{align}\label{lrid}
d_\ell&:=2^{-3/2}\rho_{\ell,\ell}/\ell!^2
.\end{align}
With \eqref{icke} as above, we again have \eqref{lriderec}--\eqref{lric},
%
Hence, using \eqref{icke} and \eqref{lrid},
\begin{align}\label{lril}
  \rho_{\ell,\ell}=2^{3/2}\ell!^2 d_1^\ell e_\ell
= 2^{3/2} \frac{\ell!\,(2\ell-2)!}{(\ell-1)!}d_1^\ell.
\end{align}
Finally, \eqref{tmix} and \eqref{lri} yield, using the duplication formula
for the Gamma function,
\begin{align}\label{lrimom}
  \E\bigsqpar{a^\ell|Y(\ga)|^{2\ell}}
\to  \frac{\sqrt{2\pi}}{\gG(\ell-\frac12)}\rho_{\ell,\ell}
= 4\sqrt\pi\frac{\ell!\,\gG(2\ell-1)}{\gG(\ell-\frac12)\gG(\ell)}d_1^\ell
=2^{2\ell}d_1^\ell \ell!
=(4d_1)^\ell \ell!
,\end{align}
and, whenever $\ell_1\neq\ell_2$,
\begin{align}\label{lrineq}
    \E\bigsqpar{a^{(\ell_1+\ell_2)/2}Y(\ga)^{\ell_1}\overline{Y(\ga)}^{\ell_2}}\to0.
\end{align}
These moment limits are the moments of a symmetric complex normal variable
with
\begin{align}\label{rio}
  \E|\zeta|^2=4d_1 
,\end{align}
(See \eg{} \cite[Theorem 1.28]{SJIII}.)
Hence, \eqref{ri} follows by the method of moments, 
with \eqref{ri2} following by \eqref{rio}, \eqref{lrid}, and \eqref{lri1}.

It remains to prove that the expression in \eqref{ri2} in non-zero.
(It can obviously not be negative by the case $\ell=1$ in the argument above.)
In other words, we must show that $\gG(\ii t-\frac12)/\gG(\ii t-1)$ cannot be imaginary
when $t\neq0$.
To see this, we first use the reflection formula for the Gamma function
\cite[5.5.3]{NIST}
to obtain
\begin{align}\label{gg1}
  \frac{\gG(\ii t-\frac12)}{\gG(\ii t-1)}
=
  \frac{\gG(2-\ii t)\sin\xpar{(\ii t-1)\pi}}{\gG(\frac32-\ii t)
\sin\xpar{(\ii t-\frac12)\pi}}
=
\frac{\ii\sinh\xpar{\pi t}}{\cosh\xpar{\pi t}}\cdot
  \frac{\gG(2-\ii t)}{\gG(\frac32-\ii t)}
.\end{align}
Hence, it is enough to show that
$  \xfrac{\gG(2-\ii t)}{\gG(\frac32-\ii t)}$ is not real for $t\neq0$.
Since $(\log \gG(z))'=\gG'(z)/\gG(z)=\psi(z)$, we have
\begin{align}\label{gg2}
\arg \frac{\gG(2-\ii t)}{\gG(\frac32-\ii t)}
= \Im \log\frac{\gG(2-\ii t)}{\gG(\frac32-\ii t)}
=\Im \int_{3/2}^2\psi(s-\ii t)\dd s.
\end{align}
Moreover, 
see \cite[5.7.6]{NIST},
\begin{align}\label{gg3}
  \psi(s+\ii t)
=
-\gam + \sumko\Bigpar{\frac{1}{k+1}-\frac{1}{k+s+\ii t}},
\end{align}
and thus
\begin{align}\label{gg4}
\Im \psi(s+\ii t)
= 
- \sumko\Im\frac{1}{k+s+\ii t}
= 
 \sumko\frac{t}{(k+s)^2+t^2}
.\end{align}
Hence,
if $s\ge1$ and $t>0$, then 
\begin{align}\label{gg5}
0<\Im \psi(s+\ii t)
< \int_{s-1}^\infty\frac{t}{x^2+t^2}\dd x
\le 
\int_{0}^\infty\frac{t}{x^2+t^2}\dd x
=\frac{\pi}2.
\end{align}
Consequently, if $t<0$, then \eqref{gg2} yields
$0 < \arg\bigpar{\gG(2-\ii t)/\gG(\frac{3}2-\ii t)} < \pi/4$,
and thus ${\gG(2-\ii t)/\gG(\frac{3}2-\ii t)}$ is not real.
The case $t>0$ follows by conjugation. 
As said above, using \eqref{gg1}, this completes the proof that
$\E|\zeta|^2>0$. 
\end{proof}

\begin{remark}
  A similar argument shows that if also $\ga'=a'+\ii b'\to \ii t'$, 
for some  $t'\notin\set{0,\pm t}$, then the covariances 
$\Cov\bigpar{Y(\ga),Y(\ga')}$ and 
$\Cov\bigpar{Y(\ga),\overline{Y(\ga')}}
=\Cov\bigpar{Y(\ga),Y(\bga')}$ are $O(1)$, and thus after normalization as
in \eqref{ri}, the covariances tend to 0.
It follows that we have joint convergence in \eqref{ri} with independent
complex normal limits,
for any number of
$\ga_k=a_k+\ii b_k\to\ii t_k$ with $t_k>0$.
We thus find as limits an uncountable family of
independent
complex normal variables.
\end{remark}

As a corollary to \refT{TRI} we see that $|Y(\ga)|\pto\infty$ as $\ga\to\ii
t$, with $t\neq0$.
\begin{problem}\label{Pitoo}
For $t\neq0$,   does $|Y(\ga)|\asto\infty$ as $\ga\to\ii t$?
\end{problem}

Nevertheless, the divergence in probability is enough to show the following.

\begin{corollary}\label{Cit}
Almost surely,  the imaginary axis is a natural boundary for the analytic functions
$Y(\cdot)$ and $\tY(\cdot)$. 
\end{corollary}
\begin{proof}
  Let $t\neq0$. Then \refT{TRI} implies that
$|Y(s+\ii t)|\pto\infty$ as $s\downto0$. Hence, there exists a sequence
$s_n\to0$ such that $|Y(s_n+\ii t)|\to\infty$ a.s.
In particular, \as{} $Y(\ga)$ cannot be extended analytically to a
neighbourhood of $\ii t$. 

Almost surely, this holds for every rational $t\neq0$, and thus $Y(\cdot)$ cannot be
extended analytically across the imaginary axis at any point.
The same holds for $\tY(\cdot)$ by \eqref{Y}.
\end{proof}

\newcommand\AAP{\emph{Adv. Appl. Probab.} }
\newcommand\JAP{\emph{J. Appl. Probab.} }
\newcommand\JAMS{\emph{J. \AMS} }
\newcommand\MAMS{\emph{Memoirs \AMS} }
\newcommand\PAMS{\emph{Proc. \AMS} }
\newcommand\TAMS{\emph{Trans. \AMS} }
\newcommand\AnnMS{\emph{Ann. Math. Statist.} }
\newcommand\AnnPr{\emph{Ann. Probab.} }
\newcommand\CPC{\emph{Combin. Probab. Comput.} }
\newcommand\JMAA{\emph{J. Math. Anal. Appl.} }
\newcommand\RSA{\emph{Random Struct. Alg.} }
\newcommand\ZW{\emph{Z. Wahrsch. Verw. Gebiete} }
\newcommand\DMTCS{\jour{Discr. Math. Theor. Comput. Sci.} }

\newcommand\AMS{Amer. Math. Soc.}
\newcommand\Springer{Springer-Verlag}
\newcommand\Wiley{Wiley}

\newcommand\vol{\textbf}
\newcommand\jour{\emph}
\newcommand\book{\emph}
\newcommand\inbook{\emph}
\def\no#1#2,{\unskip#2, no. #1,} 
\newcommand\toappear{\unskip, to appear}

\newcommand\arxiv[1]{\texttt{arXiv:#1}}
\newcommand\arXiv{\arxiv}
\newcommand\xand{\& }

\end{document}